\documentclass[reqno,11pt,letterpaper]{amsart}

\usepackage{amsmath}
\usepackage{amssymb}
\usepackage{amsthm}
\usepackage{mathrsfs}
\usepackage{accents}
\usepackage{calc}
\usepackage{arydshln}
\usepackage{upgreek}
\usepackage{slashed}
\usepackage{xifthen}
\usepackage{graphicx}
\usepackage{longtable}
\usepackage[inline]{enumitem}

\usepackage{xcolor}
\definecolor{winered}{rgb}{0.6,0,0}
\definecolor{lessblue}{rgb}{0,0,0.7}

\usepackage[pdftex,colorlinks=true,linkcolor=winered,citecolor=lessblue,urlcolor=lessblue,breaklinks=true,bookmarksopen=true]{hyperref}

\makeatletter
\newcommand{\myitem}[3]{\item[#2]\def\@currentlabel{#3}\label{#1}}
\makeatother

\usepackage{titletoc}

\makeatletter

\def\@tocline#1#2#3#4#5#6#7{
\begingroup
  \par
    \parindent\z@ \leftskip#3 \relax \advance\leftskip\@tempdima\relax
                  \rightskip\@pnumwidth plus 4em \parfillskip-\@pnumwidth
    \ifcase #1 
       \vskip 0.6em \hskip 0em 
       \or
       \or \hskip 0em 
       \or \hskip 1em 
    \fi%
    #6
    \nobreak\relax{\leavevmode\leaders\hbox{\,.}\hfill}
    \hbox to\@pnumwidth {\@tocpagenum{#7}}
  \par
\endgroup
}

 \def\l@section{\@tocline{0}{0pt}{0pc}{}{}}

\renewcommand{\tocsection}[3]{%
  \indentlabel{\@ifnotempty{#2}{ 
    \ignorespaces\bfseries{#2. #3}}}
  \indentlabel{\@ifempty{#2}{\ignorespaces\bfseries{#3}}{}} 
    \vspace{1.5pt}}

\renewcommand{\tocsubsection}[3]{%
  \indentlabel{\@ifnotempty{#2}{
    \ignorespaces#2. #3}}
  \indentlabel{\@ifempty{#2}{\ignorespaces #3}{}}
    \vspace{1.5pt}}

\renewcommand{\tocsubsubsection}[3]{%
  \indentlabel{\@ifnotempty{#2}{
    \ignorespaces#2. #3}}
  \indentlabel{\@ifempty{#2}{\ignorespaces #3}{}}
    \vspace{1.5pt}}

\makeatother

\makeatletter
\def\@nomenstarted{0}
\newlength{\@nomenoldtabcolsep}

\newcommand{\nomenstart}
  {%
    \def\@nomenstarted{1}%
    \setlength{\@nomenoldtabcolsep}{\tabcolsep}%
    \setlength{\tabcolsep}{3.5pt}%
    \begin{longtable}{p{0.11\textwidth} p{0.86\textwidth}}
  }

\newcommand{\nomenitem}[2]{%
    \ifcase\@nomenstarted%
      \or 
      \or \\ 
    \fi%
    #1\,{\leavevmode\leaders\hbox{\,.}\hfill} & #2%
    \def\@nomenstarted{2}%
  }%
\newcommand{\nomenend}
  {\\%
      \end{longtable}%
      \setlength{\tabcolsep}{\@nomenoldtabcolsep}%
      \def\@nomenstarted{0}%
  }
\makeatother

\makeatletter
\newcommand{\BIG}{\bBigg@{3.5}}
\newcommand{\vast}{\bBigg@{4}}
\newcommand{\Vast}{\bBigg@{5}}
\newcommand{\VAST}[1]{\bBigg@{#1}}
\makeatother

\allowdisplaybreaks

\numberwithin{equation}{section}
\numberwithin{figure}{section}
\newtheorem{thm}{Theorem}[section]

\newtheorem{prop}[thm]{Proposition}
\newtheorem{lemma}[thm]{Lemma}
\newtheorem{cor}[thm]{Corollary}

\newtheorem*{thm*}{Theorem}
\newtheorem*{prop*}{Proposition}
\newtheorem*{cor*}{Corollary}
\newtheorem*{conj*}{Conjecture}

\theoremstyle{definition}
\newtheorem{definition}[thm]{Definition}

\theoremstyle{remark}
\newtheorem{rmk}[thm]{Remark}
\newtheorem{example}[thm]{Example}

\makeatletter
\newcommand{\fakephantomsection}{%
  \Hy@MakeCurrentHref{\@currenvir.\the\Hy@linkcounter}
  \Hy@raisedlink{\hyper@anchorstart{\@currentHref}\hyper@anchorend}%
  \Hy@GlobalStepCount\Hy@linkcounter%
}
\makeatother

\newcommand{\mc}{\mathcal}
\newcommand{\cA}{\mc A}

\newcommand{\cC}{\mc C}
\newcommand{\cD}{\mc D}
\newcommand{\cE}{\mc E}
\newcommand{\cF}{\mc F}

\newcommand{\cH}{\mc H}

\newcommand{\cL}{\mc L}
\newcommand{\cM}{\mc M}
\newcommand{\cN}{\mc N}
\newcommand{\cO}{\mc O}
\newcommand{\cP}{\mc P}

\newcommand{\cR}{\mc R}

\newcommand{\cV}{\mc V}

\newcommand{\cX}{\mc X}

\newcommand{\ms}{\mathscr}

\newcommand{\sD}{\ms D}

\newcommand{\C}{\mathbb{C}}
\newcommand{\N}{\mathbb{N}}
\newcommand{\R}{\mathbb{R}}
\newcommand{\Z}{\mathbb{Z}}

\newcommand{\Sph}{\mathbb{S}}

\newcommand{\sfr}{\mathsf{r}}
\newcommand{\sfs}{\mathsf{s}}

\newcommand{\sfH}{\mathsf{H}}

\newcommand{\sfZ}{\mathsf{Z}}

\newcommand{\bfA}{\mathbf{A}}

\newcommand{\bfR}{\mathbf{R}}

\newcommand{\fm}{\mathfrak{m}}

\newcommand{\ft}{\mathfrak{t}}

\newcommand{\slpa}{\slashed{\partial}{}}

\newcommand{\End}{\operatorname{End}}

\renewcommand{\Re}{\operatorname{Re}}
\renewcommand{\Im}{\operatorname{Im}}
\newcommand{\Id}{\operatorname{Id}}

\newcommand{\supp}{\operatorname{supp}}
\newcommand{\sgn}{\operatorname{sgn}}

\newcommand{\dv}{\operatorname{div}}

\newcommand{\rank}{\operatorname{rank}}

\newcommand{\diag}{\operatorname{diag}}

\newcommand{\eps}{\epsilon}

\newcommand{\la}{\langle}

\newcommand{\ol}{\overline}
\newcommand{\pa}{\partial}
\newcommand{\dd}{{\mathrm d}}
\newcommand{\ra}{\rangle}
\newcommand{\spec}{\operatorname{spec}}
\newcommand{\specb}{\operatorname{spec}_\bop}

\newcommand{\wh}{\widehat}

\newcommand{\xra}{\xrightarrow}

\newcommand{\pfstep}[1]{$\bullet$\ \underline{\textit{#1}}}
\newcommand{\pfsubstep}[2]{{\bf#1}\ \textit{#2}}

\newcommand{\bop}{{\mathrm{b}}}

\newcommand{\scop}{{\mathrm{sc}}}

\newcommand{\eop}{{\mathrm{e}}}

\newcommand{\cp}{{\mathrm{c}}}

\newcommand{\Diff}{\mathrm{Diff}}

\DeclareMathOperator{\Op}{Op}

\newcommand{\Ope}{\Op_\eop}

\newcommand{\Vb}{\cV_\bop}
\newcommand{\Ve}{\cV_\eop}

\newcommand{\Diffb}{\Diff_\bop}
\newcommand{\Diffe}{\Diff_\eop}

\newcommand{\Psib}{\Psi_\bop}
\newcommand{\Psie}{\Psi_\eop}

\newcommand{\Vsc}{\cV_\scop}
\newcommand{\Diffsc}{\Diff_\scop}

\newcommand{\WF}{\mathrm{WF}}
\newcommand{\Ell}{\mathrm{Ell}}
\newcommand{\Char}{\mathrm{Char}}

\newcommand{\WFe}{\WF_{\eop}}

\newcommand{\Elle}{\mathrm{Ell}_\eop}
\newcommand{\Omegab}{{}^{\bop}\Omega}
\newcommand{\Omegae}{{}^{\eop}\Omega}

\newcommand{\Tb}{{}^{\bop}T}

\newcommand{\Tsc}{{}^{\scop}T}

\newcommand{\Te}{{}^{\eop}T}

\newcommand{\Se}{{}^{\eop}S}

\newcommand{\Ssc}{{}^{\scop}S}

\newcommand{\sigmae}{{}^\eop\upsigma}

\newcommand{\loc}{{\mathrm{loc}}}
\newcommand{\CI}{\cC^\infty}
\newcommand{\CIdot}{\dot\cC^\infty}

\newcommand{\CIc}{\cC^\infty_\cp}

\newcommand{\CmI}{\cC^{-\infty}}

\newcommand{\Hb}{H_{\bop}}
\newcommand{\He}{H_{\eop}}

\newcommand{\phg}{{\mathrm{phg}}}

\newcommand{\Ric}{\mathrm{Ric}}

\newcommand{\bhm}{\fm}

\newcommand{\openbigpmatrix}[1]
  {%
    \def\@bigpmatrixsize{#1}%
    \addtolength{\arraycolsep}{-#1}%
    \begin{pmatrix}%
  }
\newcommand{\closebigpmatrix}
  {%
    \end{pmatrix}%
    \addtolength{\arraycolsep}{\@bigpmatrixsize}%
  }

\newlength{\enummargin}

\newcommand{\usref}[1]{{\upshape\ref{#1}}}

\DeclareGraphicsExtensions{.mps}

\makeatletter
\newcommand*{\fwbw}[1]{\expandafter\@fwbw\csname c@#1\endcsname}
\newcommand*{\@fwbw}[1]{\ifcase #1 \or {\rm fw}\or {\rm bw}\fi}
\AddEnumerateCounter{\fwbw}{\@fwbw}
\makeatother

\begin{document}

\title[Wave equations with timelike curves of conic singularities]{Local theory of wave equations with timelike curves of conic singularities}

\date{\today}

\begin{abstract}
  We develop a general theory for the existence, uniqueness, and higher regularity of solutions to wave-type equations on Lorentzian manifolds with timelike curves of cone-type singularities. These singularities may be of geometric type (cone points with time-dependent cross sectional metric), of analytic type (such as asymptotically inverse square singularities or first order asymptotically scaling-critical singular terms), or any combination thereof. We can treat tensorial equations without any symmetry assumptions; we only require a condition of mode stability type for the stationary model operators defined at each point along the curve of cone points. In symmetric ultrastatic settings, we recover the solvability theory given by the functional calculus for the Friedrichs extension.
\end{abstract}


\author{Peter Hintz}
\address{Department of Mathematics, ETH Z\"urich, R\"amistrasse 101, 8092 Z\"urich, Switzerland}
\email{peter.hintz@math.ethz.ch}

\maketitle

\section{Introduction}
\label{SI}

We develop a general uniqueness, solvability, and regularity theory for a large class of wave equations $P u=f$ on Lorentzian manifolds $(\cM,g)$ where the metric $g$ and the operator $P$ are permitted to feature conic or scaling-critical singularities along a timelike curve $\cC\subset\cM$.

We give an illustration in a contrived special case for technical simplicity. We work on Minkowski space $\cM=\R_t\times\R^n_x$, $g=-\dd t^2+\dd x^2$, with $\cC=\R_t\times\{0\}$; the wave operator is $\Box_g=-D_t^2+\sum_{j=1}^n D_{x^j}^2$ where $D=i^{-1}\pa$. Introducing spatial polar coordinates $x=r\omega$, $r=|x|\geq 0$, $\omega=\frac{x}{|x|}\in\Sph^{n-1}$, we consider domains
\[
  \Omega=\{ t_-\leq t\leq t_+,\ r\leq r_+ + \tau(t_+-t) \},
\]
where $t_-<t_+$ and $\tau>1$; see Figure~\ref{FigIDom}. As function spaces on $M=\R_t\times[0,\infty)_r\times\Sph^{n-1}$ with the metric volume density $r^{n-1}|\dd t\,\dd r\,\dd g_{\Sph^{n-1}}|$, we consider \emph{weighted edge Sobolev spaces}\footnote{We shall only work in compact subsets of $M$ here, so $r$ is bounded away from $\infty$.} \cite{MazzeoEdge} which for $s\in\N_0$ and $\ell\in\R$ are defined as
\[
  \He^{s,\ell}(M) = \{ u = r^\ell u_0 \colon V_\eop^\alpha u_0\in L^2(M),\ |\alpha|\leq s \},\qquad
  V_\eop=(r\pa_t, r\pa_r, \Omega_1,\ldots,\Omega_N),
\]
where $\Omega_1,\ldots,\Omega_N$ span $\cV(\Sph^{n-1})$ over $\CI(\Sph^{n-1})$. We then define the Hilbert space
\[
  \He^{s,\ell}(\Omega)^{\bullet,-} = \{ u|_{\Omega^\circ} \colon u\in\He^{s,\ell}(M),\ u|_{t<t_-}\equiv 0 \}.
\]
We shall also consider the stronger notion of \emph{b-regularity}, which is defined via testing with the vector fields
\[
  V_\bop := (\pa_t, r\pa_r, \Omega_1,\ldots,\Omega_N).
\]

\begin{thm}[Inverse square potential with damping]
\label{ThmIBaby}
  Let $V(t,r,\omega)=r^{-2}V_0(t,r,\omega)$ with $V_0(t,0,\omega)=V_0(t)$, and assume that $\Re\sqrt{(\frac{n-2}{2})^2+V_0(t)}>\frac12$ for $t\in[t_-,t_+]$. Let $a(t,r,\omega)=r^{-1}a_0(t,r,\omega)$, where $a_0(t,0,\omega)=a_0(t)>0$ is small. Then for $s=1$ and $\ell=\frac32$ (or for sufficiently small $|\ell-\frac32|$), the equation
  \[
    (\Box_g + V + a\pa_t)u = f \in \He^{s-1,\ell-2}(\Omega)^{\bullet,-}
  \]
  has a forward solution $u\in\He^{s,\ell}(\Omega)^{\bullet,-}$ which is unique in this space and satisfies
  \[
    \|u\|_{\He^{s,\ell}(\Omega)^{\bullet,-}}\leq C\|f\|_{\He^{s-1,\ell-2}(\Omega)^{\bullet,-}}
  \]
  for a constant $C=C(V,a,\ell)$. Furthermore, if $f$ has $k$ additional degrees of \emph{b-regularity}, i.e.\ $V_\bop^\beta f\in\He^{s-1,\ell-2}(\Omega)^{\bullet,-}$ for all $|\beta|\leq k$, then also $u$ has $k$ orders of b-regularity as well, i.e.\ $V_\bop^\beta u\in\He^{s,\ell}(\Omega)^{\bullet,-}$ for all $|\beta|\leq k$.
\end{thm}

\begin{figure}[!ht]
\centering
\includegraphics{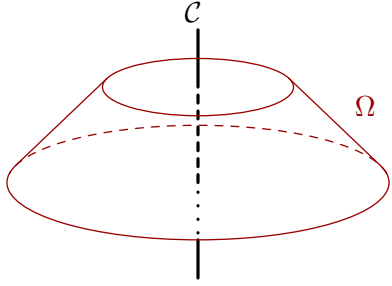}
\caption{The spacetime domain $\Omega$ on which, in Theorem~\ref{ThmIBaby}, we solve wave equations whose coefficients and source terms feature singularities at the timelike curve $\cC$ of cone points.}
\label{FigIDom}
\end{figure}

The uniqueness part in particular implies that finite speed of propagation holds, i.e.\ $\supp u$ is contained in the causal future of $\supp f$. We explain in Remark~\ref{RmkExScVMild} how Theorem~\ref{ThmIBaby} follows from the main results of this paper, Theorems~\ref{ThmSUeNonrf} and \ref{ThmSUb},\footnote{We also mention Theorems~\ref{ThmSUbPhg} and \ref{ThmSUIVP} on polyhomogeneous expansions at $r=0$ and initial value problems, respectively.} which concern wave-type equations generalizing the above setup in the following ways.
\begin{enumerate}
\item\label{ItIg} We work on \emph{arbitrary spacetimes} $(\cM,g)$ and with metrics $g$ which may have \emph{conic singularities along a timelike curve} $\cC$. In polar coordinates around $\cC$, such metrics take the form
\begin{equation}
\label{EqIMetric}
  g=-\dd t^2+\dd r^2+r^2 h(t,\omega;\dd\omega)+\cO(r)
\end{equation}
where $h$ is a (time-dependent) Riemannian metric on $\Sph^{n-1}$, and $\cO(r)$ indicates terms which are smooth in polar coordinates and whose coefficients, expressed in terms of $\dd t,\dd r,r\,\dd\omega$,\footnote{We do allow for all possible cross terms, so $\dd t^2$, $\dd t\,\dd r$, $r\,\dd t\,\dd\omega$, $\dd r^2$, $r\,\dd r\,\dd\omega$, $r^2\,\dd\omega\,\dd\omega'$.} vanish simply at $r=0$. The spacetime domains $\Omega$ on which our theory takes the simplest form are subject to a \emph{non-refocusing} condition: null-geodesics cannot start and end at $\cC$ while staying inside $\Omega$.\footnote{Without this assumption on $\Omega$, we can still prove existence and regularity results, albeit lossy ones; see Remarks~\ref{RmkSUeGeneral} and \ref{RmkSUbGeneral}.} (This is satisfied in the setting of Theorem~\ref{ThmIBaby}.)
\item\label{ItITensor} We can deal with arbitrary \emph{tensorial} equations $P u=f$ where $P$ has the same principal symbol as $\Box_g$ and may in addition feature scaling-critical singular terms ($\sim r^{-2},r^{-1}\pa_t,r^{-1}\pa_r,r^{-2}\pa_\omega$) at $\cC$. We require spectral assumptions (of mode stability type) for \emph{stationary model operators} defined at each point $(t_0,0)\in\cC\cap\bar\Omega$.
\item\label{ItIVar} We can work on edge Sobolev spaces with optimal ranges of weights $\ell$. Our estimates typically require using \emph{variable} edge regularity orders. (The peculiar setup of Theorem~\ref{ThmIBaby} allows for the edge regularity order to be \emph{constant} and \emph{integer}.)
\end{enumerate}

See~\S\ref{SsIG} for details. The applications of our general theory given in~\S\ref{SEx} include wave equations on spacetimes with timelike curves of cone points (with possibly variable metric on the cross section) and their coupling with (complex) potentials which have asymptotically inverse square singularities. Inspired by the work of Baskin--Wunsch \cite{BaskinWunschDiracCoulomb}, we also discuss the Dirac--Coulomb operator, whose `square' fits into our framework; see~\S\ref{SssExScC}.

The author's main motivation for the present study comes from gluing problems in general relativity of the type discussed in \cite{HintzGlueLocI}: introducing on a vacuum spacetime $(M,g)$, so $\Ric(g)=0$, local coordinates $t\in\R$, $x\in\R^3$ near a timelike curve $\cC$ so that $\cC=\{x=0\}$ and $g=-\dd t^2+\dd x^2+\cO(|x|)$, one may wish to modify $g$ in an $\eps$-neighborhood of $x=0$ by replacing it with $g_\eps:=\chi(x/\eps)g^{\rm S}_\eps+(1-\chi(x/\eps))g$ where $g^{\rm S}_\eps=-(1-\frac{2\eps}{|x|})\dd t^2+(1-\frac{2\eps}{|x|})^{-1}\dd r^2+r^2 g_{\Sph^2}$ (in polar coordinates in $x$) is the metric of a mass $\eps$ Schwarzschild black hole \cite{SchwarzschildPaper}, and $\chi$ is equal to $1$ on a large ball and $0$ outside a larger ball. Correcting $g_\eps$ to a solution of the Einstein vacuum equations requires, among other things, estimates for the linearization of the Einstein equations around $(M,g)$, with the important caveat that one needs to regard $\cC$ as a singular locus: after all, for any positive $\eps$, waves entering the region $|x|\leq C\eps$ are scattered by the mass $\eps$ black hole, akin to scattering by a conic singularity. The tensorial nature of the linearized Einstein equations, the time-dependent nature of the metric $g$, and the absence of symmetry or self-adjointness properties thus motivated the development of a robust quantitative uniqueness, solvability, and regularity theory in this paper.

\bigskip

The main focus of the existing literature on wave equations in the presence of cone-type singularities (by which we mean geometric cone points, inverse square potentials, or other scaling-critical singular terms) is on the \emph{ultrastatic case}
\begin{equation}
\label{EqIP}
  P = -D_t^2 + \Delta_\gamma + V,
\end{equation}
where $\Delta_\gamma$ is the (nonnegative) Laplacian on a Riemannian manifold $(X,\gamma)$ with conic singularities, and $V$ is a suitable \emph{real-valued} potential with at worst inverse square singularities; here $\gamma$ and $V$ are \emph{independent of time}. Existence and uniqueness of solutions of initial value problems, or of forward solutions, is straightforward in spaces of distributions in $t$ with values in domains of powers of (the Friedrichs extension of) $\Delta_\gamma+V$. The objectives are then refined descriptions of the propagation of singularities, or proofs of Strichartz and local smoothing estimates. (By contrast, in the general setting considered in this paper, even basic questions of existence and uniqueness questions cannot be obtained by spectral or energy methods.)

Concretely, the diffraction by waves on conic manifolds, given by solutions of the operator~\eqref{EqIP} with $V=0$, was studied in detail by Cheeger--Taylor \cite{CheegerTaylorConicalI,CheegerTaylorConicalII} on exact cones (i.e.\ $\gamma=\dd r^2+r^2 h(\omega;\dd\omega)$ near $r=0$) using separation of variables and special function analysis; Yang \cite{YangDiffraction} improved their description of the diffracted wave front and clarified the relationship between diffraction by a conic singularity and scattering on the large end of an exact cone. Closely related to this is the work of Keeler--Marzuola \cite{KeelerMarzuolaCones} on dispersive estimates for Schr\"odinger equations on exact cones coupled to radial potentials; see also \cite{SchlagSofferStaubachConicI,SchlagSofferStaubachConicII}. Melrose--Wunsch \cite{MelroseWunschConic} gave the first treatment of the propagation of singularities on non-exact cones: they prove (under a necessary non-focusing assumption) that the strongest singularities propagate along geometric null-geodesics, i.e.\ limits of null-geodesics barely missing the cone point, whereas singularities along diffractive null-geodesics, i.e.\ those emanating from the cone point in all other directions, are weaker. The microlocal propagation of edge regularity into and out of the cone point (blown up to a full $(n-1)$-sphere) plays a central role both in \cite[\S8]{MelroseWunschConic} and in the present paper (\S\ref{SsSUPr}). More general results describing fine aspects of the propagation of singularities when $(X,\gamma)$ has edge singularities (which generalizes the conic setting) or corners are obtained in \cite{MelroseVasyWunschEdge,MelroseVasyWunschDiffraction} using the techniques developed in \cite{VasyPropagationCorners} which proves b-regularity relative to the quadratic form domain. These proofs utilize a mixed (edge-)differential and b-pseudodifferential operator algebra to accommodate the structure of the operator and the notion of regularity whose propagation is analyzed; in the present paper, we instead use a mixed b-differential and edge-pseudodifferential algebra in~\S\ref{SsSUb} to describe integer order b-regularity with respect to (possibly variable order) edge Sobolev spaces. Global-in-time propagation estimates in the presence of conic singularities are established for the wave equation on exact cones by Baskin--Marzuola \cite{BaskinMarzuolaCone}; the late-time behavior of solutions of the wave equation is shown to be described by explicitly computable resonances \cite{BaskinMarzuolaComp}. We also mention the work by Baskin--Wunsch \cite{BaskinWunschConicDecay} who prove a \emph{very weak Huyghens' principle} when $(X,\gamma)$ has finitely many cone points (subject to some mild geometric conditions) and is Euclidean near infinity: the wave, with compactly supported initial data, is as smooth as one wishes in any fixed compact subset of $X$ after a sufficiently long time.\footnote{Our results apply in the presence of several disjoint timelike curves of cone points, though the non-refocusing condition imposed on the spacetime domains now requires the non-existence of null-geodesics starting at one curve end ending at the same \emph{or another} curve. However, one can again concatenate our results to obtain solvability and regularity results on general spacetime domains.}

On Lorentzian manifolds with corners, and thus \emph{outside the ultrastatic regime}, Vasy \cite{VasyDiffractionForms} describes the diffraction of singularities for \emph{given} solutions of the wave equation on differential forms (which includes Maxwell's equations as a special case) with natural boundary conditions. We remark that such regularity results are closely related to, but do not by themselves imply, uniqueness and solvability of the underlying wave equation. Similarly, the works \cite{VasyWaveOnAdS,GannotWrochnaAdS,DappiaggiMartaAdS} describe the propagation of singularities for solutions of the Klein--Gordon equation at the conformal boundary of (not necessarily stationary) asymptotically anti--de~Sitter type spacetimes. These results are complemented by statements regarding existence and uniqueness of solutions which are proved by energy methods (see also \cite{HolzegelAdS,WarnickAdS}).

We now turn to wave operators on Minkowski space (so $\gamma$ in~\eqref{EqIP} is the Euclidean metric and $\Delta_\gamma=-\sum_{j=1}^n \pa_{x^j}^2$) coupled to singular potentials $V$. At a singularity of $V$, say $r=0$, the assumption $\liminf_{r\to 0}r^2 V>-(\frac{n-2}{2})^2$ implies, by Hardy's inequality, the positivity (locally near $r=0$ at least) of the quadratic form associated with $\Delta_\gamma+V$. Using the Friedrichs extension to define $\Delta_\gamma+V$, Burq and Planchon--Stalker--Tahvildar-Zadeh \cite{BurqPlanchonStalkerTahvildarZadehInvSq,PlanchonStalkerTahvildarZadehInvSq} establish Strichartz estimates in the case that $V=\frac{V_0}{r^2}$ is an exact inverse square potential; and Qian \cite{QianDiffractionInvSq} proves diffractive improvements for the propagation of singularities through an inverse square singularity (allowing $(r^2 V)(r,\theta)|_{r=0}$ to be variable). Duyckaerts \cite{DuyckaertsInvSq} proves local smoothing and Strichartz estimates for the Schr\"odinger flow for $\Delta+V$ for approximate inverse square potentials located at a finite collection of points; the main ingredient are (lossless) high energy estimates for the associated limiting resolvent $(\Delta+V-(\lambda\pm i0))^{-1}$ when $\lambda\to\infty$. Also the aforementioned \cite{BaskinWunschConicDecay} deduces lossless high energy resolvent estimates from their very weak Huyghens' principle via Vainberg's machinery \cite{VainbergAsymptotic,GalkowskiVainberg}. The location of resonances, i.e.\ poles of the meromorphic continuation of the resolvent, is described by Hillairet--Wunsch \cite{HillairetWunschConic}. In this resolvent context, we also mention Xi's precise parametrix for high frequency diffraction by non-exact conic singularities \cite{XiConeParametrix}. In \cite{HintzConicProp}, the author proves high energy estimates for semiclassical operators featuring conic-type singularities without symmetry or self-adjointness properties; such operators arise as spectral families, in the high frequency regime, of stationary ($t$-independent) wave operators of the type considered in the present paper.

\bigskip

To illustrate the limitations of spectral and energy methods, consider again
\begin{equation}
\label{EqIEq}
  (-D_t^2+\Delta+V)u=f
\end{equation}
where $\Delta$ is the (non-negative) Euclidean Laplacian, and $V$ is an inverse square potential; say $V=\frac{V_0}{r^2}\chi(r)$ where $\chi\in\CIc([0,\infty))$ equals $1$ near $0$, and $V_0\in\C$. By the Hardy inequality, factors of $r^{-1}=|x|^{-1}$ should be regarded as having the same strength at $r=0$ as derivatives $\pa_{x^j}$. Thus, the symmetry assumption $V_0\in\R$ is crucial if one wishes to apply energy methods; concretely, in the case $V_0>-(\frac{n-2}{2})^2$, the derivative of the energy
\begin{equation}
\label{EqIEnergy}
  E(t):=\frac12\int_{\R^n} |\pa_t u|^2+|\nabla_x u|^2+V|u|^2 + C|u|^2\,\dd x
\end{equation}
(where $C>0$ is chosen so that $E(t)$ is coercive) of a solution of~\eqref{EqIEq} can be bounded by $E$ itself (and the source term $f$). This works also in the case that $V_0$ is time-dependent, and applies also on nontrivial backgrounds $(M,g)$.

If one wishes to access the functional calculus for (the Friedrichs extension of) $\Delta+V$, which immediately gives a solvability and uniqueness theory, one must assume $V_0\in\R$ and $V_0>-(\frac{n-2}{2})^2$, and now also stationarity, i.e.\ the time-independence of $V$. (We present a detailed comparison of our theory with this approach in~\S\ref{SssExCUStat}.)

More permissively, one may study stationary operators also for complex-valued potentials $V$ (which model, in a simple manner, wave operators acting on sections of a vector bundle which does not possess a natural positive definite fiber inner product), and also on general stationary Lorentzian backgrounds, by means of the Fourier transform in time. In the special case under consideration here, this produces the spectral family $\Delta+V-\sigma^2$, $\sigma\in\C$, whose invertibility in $\Im\sigma\geq 0$ on suitable function spaces implies the forward solvability of~\eqref{EqIEq} by the Paley--Wiener theorem. (See Proposition~\ref{PropSUIFwd} for a similar argument.)

Equations which feature time-dependent coefficients but do not satisfy any symmetry or reality conditions do not seem to be accessible using these techniques. The approach of the present paper is to combine two ingredients:
\begin{enumerate}
\item a regularity theory, based on essentially classical microlocal propagation results in weighted edge Sobolev spaces \cite{MelroseWunschConic,MelroseVasyWunschEdge}---which natively handle bundles and non-symmetric operators, and for which only the geometry of $(M,g)$ matters (but no lower order terms of the wave-type operator except for the calculation of threshold regularities at certain \emph{radial sets}); and
\item the forward solvability of stationary model operators, effected via the Fourier transform, which are defined at each individual point of $\cC$ and capture the momentary structure of the wave operator there.
\end{enumerate}
Energy methods are used only near the boundary of $\Omega$ in order to `cap off' our otherwise microlocal estimates (which only ever provide estimates $u$ on some set in terms of weaker norms of $u$ on a slightly larger set). The general strategy of combining (microlocal) control of regularity and the invertibility of model operators (or normal operators) to get invertibility or Fredholm statements for operators on noncompact or singular spaces has been applied previously in a large variety of settings, in particular in the context of Melrose's geometric microlocal analysis program \cite{MelroseMendozaB,MazzeoMelroseHyp,MazzeoEdge,MelroseAPS,MazzeoMelroseFibred,VasyThreeBody,VasyManyBody,GuillarmouHassellResI,AlbinGellRedmanDirac,GrieserTalebiVertmanPhiLowEnergy}; applications to wave equations include \cite{HintzVasySemilinear,BaskinVasyWunschRadMink,GellRedmanHaberVasyFeynman,BaskinVasyWunschRadMink2,HintzNonstat}.

\subsection{Setup of the general theory}
\label{SsIG}

We only sketch the general setup here and refer the reader to~\S\ref{SSU} for details.. Focusing on a neighborhood $\R_t\times[0,1)_r\times\Sph^{n-1}\subset M$ of the curve of cone points, we consider metrics of the form~\eqref{EqIMetric} (see~\S\ref{SD} for the assumptions away from $r=0$), and wave-type operators $P$ (which may act on sections of a smooth vector bundle $\cE\to M$, cf.\ \S\ref{SSU}) which are thus, in local coordinates $\omega^1,\ldots,\omega^{n-1}$ on $\Sph^{n-1}$, of the form
\[
  P = \pa_t^2 - \pa_r^2 - \frac{n-1+b}{r}\pa_r + r^{-2}\Delta_{h(t)} + a r^{-1}\pa_t + r^{-2}w^j\pa_{\omega^j} + r^{-2}V + \tilde P.
\]
Here $a,b,w^j,V$ are functions of $(t,\omega)$, and $\tilde P$ is a second order operator whose coefficients (with respect to $\pa_t^2,\pa_r^2,r^{-2}\pa_{\omega^j}\pa_{\omega^k}$ and $r^{-1}\pa_t$, $r^{-1}\pa_r$, $r^{-2}\pa_{\omega^j}$ and $1$) are smooth and vanish at $r=0$; we require the principal part of $P$ to be the same as that of $\Box_g$. The operator $r^2 P$ is a Lorentzian edge differential operator \cite{MazzeoEdge}, i.e.\ it is constructed out of the edge vector fields $r\pa_t$, $r\pa_r$, $\pa_{\omega^j}$---the smooth vector fields tangent to the fibers of $\{r=0\}=\R_t\times\{0\}\times\Sph^{n-1}\to\R_t$---and its principal part is a Lorentzian signature quadratic form in these. Only using this information about the principal symbol, and mild information about subprincipal terms at the radial sets (the subsets of the appropriate edge phase space $\Te^*M\setminus o$ where lifts of null-geodesics hit or leave the `cone point' $r=0$), one can control regularity in edge Sobolev spaces fully microlocally; see \S\ref{SsSUPr}. A caveat here is that the degree of edge regularity propagating out of the curve of cone points is limited by some number $s_{\rm out}$ (see Definition~\ref{DefSULocAdm}), whereas propagation into the curve requires a lower bound $s_{\rm in}$; if $s_{\rm in}\geq s_{\rm out}$, we therefore cannot control edge regularity emanating from the curve and focusing back on it at a later point. This is why we work on \emph{non-refocusing} spacetime domains on which, by definition, such dynamics do not occur.\footnote{This is only a mild restriction, as any sufficiently small domain has this property.}

The normal operator $N_{\eop,t_0}(P)$ of $P$ at the point $t_0\in\R$ along the `curve' (really: along the front face of the blow-up of the curve) $r=0$ is given by freezing the coefficients of $a,b,w^j,V$ at $t=t_0$. The resulting operator is thus stationary (invariant under translations in $t$) and homogeneous with respect to dilations in $(t,r)$. To analyze it, we pass to the Fourier transform, thus replacing $\pa_t$ by $-i\sigma$ where $\sigma\in\C$ is the spectral parameter; and we moreover exploit the homogeneity in $(\sigma^{-1},r)$ by passing to $\hat r=r|\sigma|$ and $\hat\sigma=\frac{\sigma}{|\sigma|}$. This produces the reduced normal operator \cite{LauterPsdoConfComp}
\begin{align*}
  \hat N_{\eop,t_0}(P,\hat\sigma) &= -\pa_{\hat r}^2 - \frac{n-1+b(t_0,\omega)}{\hat r}\pa_{\hat r} + \hat r^{-2}\Delta_{h(t_0)} - \hat\sigma^2 \\
    &\qquad - \frac{i\hat\sigma a(t_0,\omega)}{\hat r} + \hat r^{-2}w^j(t_0,\omega)\pa_{\omega^j} + \hat r^{-2}V(t_0,\omega)
\end{align*}
on $[0,\infty)_{\hat r}\times\Sph^{n-1}$; see~\S\ref{SsSUN}; in the special case $P=\Box_g$, this is the spectral family of the Laplacian on an exact cone. We require this operator to be invertible not only for $\hat\sigma=\pm 1$, but also for $\Im\hat\sigma>0$, $|\hat\sigma|=1$; since we are not in a symmetric setting, this is not automatic (unlike for example in \cite{MelroseWunschConic}). The relevant function spaces are weighted b-Sobolev spaces near $\hat r=0$ (measuring regularity with respect to $\hat r^\ell L^2$, for suitable $\ell$, with respect to $\hat r\pa_{\hat r}$, $\pa_{\omega^j}$) \cite{MelroseAPS} and weighted scattering Sobolev spaces near $\hat r=\infty$ (now testing regularity using $\pa_{\hat r}$, $\hat r^{-1}\pa_{\omega^j}$) \cite{MelroseEuclideanSpectralTheory} for suitable choices of weights. Typically, the decay order at $\hat r=\infty$ needs to be variable in order to distinguish between incoming ($\sim e^{-i\hat\sigma\hat r}$) and outgoing ($\sim e^{i\hat\sigma\hat r}$) spherical waves; see~\S\ref{SsSUI} and also \cite[Proposition~5.28]{VasyMinicourse}.

The observation which lies at the heart of our theory is that the natural b-scattering estimates on $\hat N_{\eop,t_0}(P,\hat\sigma)^{-1}$ in Lemma~\ref{LemmaSUIInv} for $|\hat\sigma|=1$, $\Im\hat\sigma\geq 0$, upon passing back to $r,\sigma$, translate directly into estimates on edge Sobolev spaces, with the dictionary for the orders provided by Lemma~\ref{LemmaEInvFT}. (In the special case of $\He^1$, this is elementary: membership $u\in\He^1$ (with support in $r\lesssim 1$) is equivalent to the memberships $u,r\pa_t u,r\pa_r u,\pa_\omega u\in L^2$, so by Plancherel to $L^2(\R_\sigma;L^2)$-membership of $(1+r|\sigma|)\hat u$, $r\pa_r\hat u$, $\pa_\omega\hat u$, and thus upon rewriting in terms of $\hat r=r|\sigma|$ to a certain $L^2$-membership of $\la\hat r\ra\hat u$, $\hat r\pa_{\hat r}\hat u$, $\pa_\omega\hat u$, which means b-regularity of order $1$ near $\hat r=0$, and scattering regularity of order $1$ with respect to $\la\hat r\ra^{-1}L^2$.) Using a Paley--Wiener argument, we thus obtain a forward inverse of $N_{\eop,t_0}(P)$ on appropriate edge Sobolev spaces.

Combining the edge regularity estimates with the normal operator invertibility, one can show solvability and uniqueness for forward problems for $P$ on small domains around the point $t_0$. The concatenation of such results leads to the following semi-global result.

\begin{thm}[Main result, rough version]
\label{ThmI}
  Let $\Omega$ be a non-refocusing domain whose boundary is a union of spacelike hypersurfaces, and suppose that the reduced normal operator $\hat N_{\eop,t_0}(P,\hat\sigma)$ is invertible for all $|\hat\sigma|=1$, $\Im\sigma\geq 0$, at all points $t_0$ along the curve which lie in $\bar\Omega$ on function spaces with weight $\ell\in\R$ at $\hat r=0$ and encoding an outgoing property at $\hat r=\infty$. Then the forward problem for $P u=f$ has a unique solution $u\in\He^{\sfs,\ell}(\Omega)^{\bullet,-}$ for source terms $f\in\He^{\sfs-1,\ell-2}(\Omega)^{\bullet,-}$, where the (typically variable) edge regularity order $\sfs$ is monotone along the future null-geodesic flow (lifted to the edge phase space) and is required to satisfy explicit upper and lower bounds near $r=0$.
\end{thm}

See Theorem~\ref{ThmSUeNonrf} for the detailed statement which also gives uniqueness in $\He^{-\infty,\ell}$. We mention one important technical aspect involved in the proof of Theorem~\ref{ThmI}: in order to close the microlocal estimates, we need to complement them with energy estimates near the initial and final boundary hypersurfaces of $\Omega$. As explained around~\eqref{EqIEnergy}, we do not have access to such estimates for general $P$. Instead, we only use energy estimates in wedge-shaped domains of the form $\frac{t-t_0}{r}\in[\tau_-,\tau_+]\subset(-1,1)$, in which $\tau=\frac{t-t_0}{r}$ is a time function and where such energy estimates (on $L^2$ spaces with $r^\ell$ weights) are straightforward to prove using the vector field multiplier $r^{-2\ell}e^{-\digamma\tau}\pa_\tau$, $\digamma\gg 1$; see~\S\ref{SsSULoc}. The utility of these `edge-local' energy estimates and corresponding edge-local solvability results arises from the ability to localize edge microlocal estimates to such wedge-shaped domains. The point is that localizers in $\tau$, while having non-smooth coefficients in $M$, \emph{do} have edge-regular coefficients, which is sufficient to allow for a definition of such sharp (edge-pseudodifferential) localizers; see~\S\ref{SsEPsdo}.

Commuting $P u=f$ with a collection of vector fields (including $\pa_t$ which is \emph{not} an edge vector field) gives higher b-regularity at the expense of a matching loss in edge regularity; this loss can be removed via the microlocal propagation of edge regularity relative to spaces with fixed b-regularity. See \S\ref{SsSUb} and Theorem~\ref{ThmSUb} for details.

\begin{rmk}[Limitations]
\label{RmkIGLim2}
  Restricted to symmetric ultrastatic settings, Theorem~\ref{ThmI} produces the solution given by the functional calculus for the Friedrichs extension; see Propositions~\ref{PropCDUStat}--\ref{PropCDUStatSpec}, and also Proposition~\ref{PropExCDUMink}. It is not clear how to deal with other boundary conditions, e.g.\ those considered in \cite{VasyDiffractionForms}.
\end{rmk}

\subsection{Outline}

In~\S\ref{SD}, we describe the class of spacetimes and Lorentzian metrics, containing timelike curves of cone points, on which we will study wave-type equations; we also introduce the notion of non-refocusing spacetime domains and analyze (monotone functions along the lift to phase space of) the null-geodesic flow on them. The heart of the paper is~\S\ref{SSU} in which we introduce the class of wave-type operators under consideration and study their uniqueness and solvability properties. Applications are described in~\S\ref{SEx}. Appendix~\ref{SE} recalls notions of geometric singular analysis as well as aspects of (microlocal) analysis in Mazzeo's edge calculus \cite{MazzeoEdge}.

\subsection*{Acknowledgments}

I would like to thank Andr\'as Vasy for useful conversations. I gratefully acknowledge the hospitality of the Erwin Schr\"odinger Institute in Vienna in June and July 2023 during the writing of this paper.

\section{Geometry of spacetime domains}
\label{SD}

The reader unfamiliar with geometric singular analysis and the edge calculus is advised to consult Appendix~\ref{SE} before proceeding. We work with an $(1+n)$-dimensional smooth manifold $M$ with embedded boundary $\pa M$. We assume that $\pa M$ is the total space of a fibration
\[
  \Sph^{n-1} - \pa M \xra{\phi} I
\]
where $I\subseteq\R$ is an open interval. We then set
\begin{equation}
\label{EqDEquiv}
  \cM := M / {\sim}\,, \quad p\sim q\ \ \text{if and only if}\ \ p,q\in\pa M,\ \phi(p)=\phi(q),
\end{equation}
and equip $\cM$ with the quotient topology; then $\cM$ is a smooth manifold away from the singular locus $\cC:=\pa M/{\sim}\cong\R$ which we shall refer to as the `curve of cone points'. See Figure~\ref{FigD}.

Note that if, conversely, we start with a smooth manifold $\cM\cong I_t\times\cX$, and $\cC=I\times\{x_0\}$ (for some $x_0\in\cX$) is a curve in $\cM$, then the blow-up $M:=[\cM;\cC]$ has the above structure; in the spirit of \cite{MelroseWunschConic,MelroseVasyWunschEdge}, rather than working on $\cM$ with geometric and analytic objects which become singular along $\cC$, we work here with smooth objects on the manifold with boundary $M$ which degenerate in a structured manner as one approaches $\pa M$.

\begin{figure}[!ht]
\centering
\includegraphics{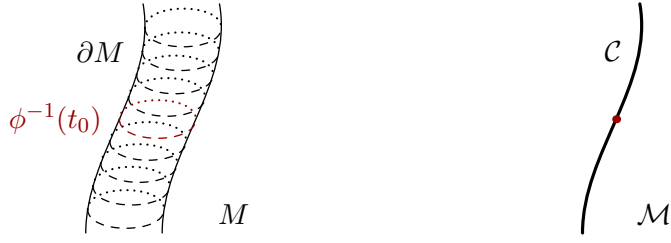}
\caption{\textit{On the left:} the spacetime manifold $M$ and the fibers of its boundary fibration (dashed). \textit{On the right:} the singular quotient space $\cM$.}
\label{FigDM}
\end{figure}

On $M^\circ=M\setminus\pa M$, we consider a smooth Lorentzian metric $g$ of signature $(-,+,\ldots,+)$; we assume that there exists a collar neighborhood
\begin{equation}
\label{EqDCollar}
  I_t \times [0,\bar r)_r \times \Sph^{n-1}_\omega
\end{equation}
of $\pa M$ with $r\in\CI(M)$ a boundary defining function, so that the fibration is given by $\phi(t,\omega)=t$ and so that $g$ takes the form
\begin{equation}
\label{EqDMetric}
  g = -\dd t^2 + \dd r^2 + r^2 h(t,\omega;\dd\omega) + \tilde g,
\end{equation}
where $\tilde g$ is a quadratic form in $\dd t,\dd r,r\,\dd\omega$ with coefficients of class $r\CI$. Therefore,
\[
  g_\eop := r^{-2}g
\]
is a smooth Lorentzian edge metric on $M$, and $r^{-2}\tilde g\in r\CI(M;S^2\,\Te^*M)$ is of lower order at $\pa M=r^{-1}(0)$. We moreover assume that $(M^\circ,g)$ is time-orientable, with $\pa_t$ declared to be future timelike near $\pa M$; and we assume that there exists a global time function $\ft\in\CI(M)$ (so $\dd\ft$ is everywhere past timelike) whose restriction to $\pa M$ is fiber-constant.

\begin{definition}[Spacetimes]
\label{DefDSpacetime}
  We call $(M,g)$ satisfying these conditions a \emph{spacetime with a timelike curve of cone points}.
\end{definition}

Since $\ft|_{\pa M}$ is a monotone reparameterization of $t|_{\pa M}$, we may reparameterize $\ft$ so that $\ft=t$ at $\pa M$; and since $\dd t$ and $\dd\ft$ are both past timelike near $\pa M$, we can further modify $\ft$ to be equal to $t$ in a neighborhood of $\pa M$ (see also \cite[Lemma~3.30]{HintzGlueLocI}); we may thus relabel $\ft$ as $t$. That is, the function $t$ in~\eqref{EqDCollar} is the restriction to the collar neighborhood of a global time function $t\in\CI(M)$.

Given $t_0\in I$, consider $g_\eop|_{\phi^{-1}(t_0)}$, obtained by freezing the coefficients (as an edge metric) at $t=t_0$ and $r=0$. Extending this by translation-invariance in $t$ and dilation-invariance in $(t,r)$, we obtain the model (edge-)metric
\begin{equation}
\label{EqDMetricModel}
  g_{\eop,t_0} = r'{}^{-2}\bigl(-\dd t'{}^2 + \dd r'{}^2 + r'{}^2 h(t_0,\omega;\dd\omega)\bigr)
\end{equation}
on the inward pointing normal bundle ${}^+N\phi^{-1}(t_0)$ of $\pa M$ over $\phi^{-1}(t_0)$, here identified with $\R_{t'}\times[0,\infty)_{r'}\times\Sph^{n-1}_\omega$ where we write $t'=\dd t$ and $r'=\dd r$.

On $(M,g)$, we shall consider wave equations on domains of the following type.

\begin{definition}[Spacetime domains]
\label{DefD}
  A \emph{spacetime domain} $\Omega$ in $(M,g)$ is a nonempty precompact open subset $\Omega\subset M$ of the form
  \begin{equation}
  \label{EqD}
    \Omega = \bigcap_{j=1}^{N_{\rm ini}}\{t_{\rm ini,j}>0\} \cap \bigcap_{j=1}^{N_{\rm fin}}\{t_{\rm fin,j}<0\}
  \end{equation}
  where $N_{\rm ini},N_{\rm fin}\geq 1$, and the functions $t_{\rm ini,j},t_{\rm fin,j}\in\CI(M)$ have past timelike differentials near their respective zero sets on $\bar\Omega$; we require the differentials of every collection of these functions are linearly independent at the joint zero set of the functions in this collection. Furthermore, we require that for $j\neq 1$, the zero sets of $t_{\rm ini,j}$ and $t_{\rm fin,j}$ are disjoint from $\bar\Omega\cap\pa M$; and $t_{\rm ini,1}^{-1}(0)\cap t_{\rm fin,1}^{-1}(0)$ is disjoint from $\bar\Omega\cap M$ as well. We call the sets $t_{\rm ini,j}^{-1}(0)\subset\bar\Omega$ \emph{initial boundary hypersurfaces}, and the sets $t_{\rm fin,j}^{-1}(0)\subset\bar\Omega$ \emph{final boundary hypersurfaces}.
\end{definition}

See Figure~\ref{FigD}. The final condition rules out domains which have a wedge singularity at $\pa M$, an example being $\{\frac{r}{2}<|t|<1-\frac{r}{2}\}\subset\R_t\times[0,\infty)_r\times\Sph^{n-1}$. Control of waves on such domains can be accomplished using simple energy methods; see Proposition~\ref{PropSULoc}.

\begin{figure}[!ht]
\centering
\includegraphics{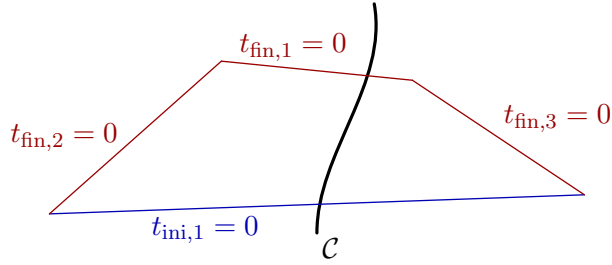}
\caption{A spacetime domain $\Omega$, with the fibers of $\pa M$ collapsed to points.}
\label{FigD}
\end{figure}

For all sufficiently small $\delta_{\rm ini,j},\delta_{\rm fin,j}\in\R$, the set
\[
  \Omega_{\delta_{\rm ini,1},\ldots,\delta_{\rm ini,N_{\rm ini}},\delta_{\rm fin,1},\ldots,\delta_{\rm fin,N_{\rm fin}}} = \bigcap_{j=1}^{N_{\rm ini}} \{ t_{\rm ini,j}>\delta_{\rm ini,j} \} \cap \bigcap_{j=1}^{N_{\rm fin}} \{ t_{\rm fin,j}<\delta_{\rm fin,j} \}
\]
is a spacetime domain as well. In the special case $\delta_{\rm ini,j}=\delta_{\rm ini}$, $\delta_{\rm fin,j}=\delta_{\rm fin}$ for all $j$, we denote this domain simply by
\begin{equation}
\label{EqDEnlarged}
  \Omega_{\delta_{\rm ini},\delta_{\rm fin}}.
\end{equation}
Thus, $\Omega_{\delta,-\delta}\subsetneq\Omega\subsetneq\Omega_{-\delta,\delta}$.

The plan for the remainder of this section is as follows.
\begin{enumerate}
\item In~\S\ref{SsDE}, we describe the null-geodesic flow, lifted to the edge cotangent bundle, on $(M,g)$.
\item In~\S\ref{SsDNrf}, we introduce a subclass of spacetime domains for which the null-geodesic flow does not connect two distinct points along $\cC\subset\cM$, and show that general spacetime domains can be decomposed into finite unions of such \emph{non-refocusing domains}.
\item In~\S\ref{SsDFn}, we construct certain monotone functions on non-refocusing domains, which will be used as order functions for the microlocal edge regularity theory in~\S\ref{SsSUPr}.
\end{enumerate}

\emph{Throughout, we work on a fixed spacetime with a timelike curve of cone points $(M,g)$.}

\subsection{Null-geodesics near the curve of cone points}
\label{SsDE}

Recalling $g_\eop=r^{-2}g$, denote by $G_\eop\colon\Te^*M\to\R$, $G_\eop(\zeta)=g_\eop^{-1}(\zeta,\zeta)$, the dual metric function. In the collar neighborhood~\eqref{EqDCollar}, we write edge-covectors as
\begin{equation}
\label{EqDECovec}
  {-}\sigma\,\frac{\dd t}{r} + \xi\,\frac{\dd r}{r} + \eta,\qquad \eta\in T^*\Sph^{n-1}.
\end{equation}
Using the summation convention, we thus obtain from~\eqref{EqDMetric}
\[
  G_\eop(t,r,\omega;\sigma,\xi,\eta) = -\sigma^2 + \xi^2 + h(t,\omega)^{a b}\eta_a\eta_b + \tilde G_\eop,\qquad \tilde G_\eop\in r P^{[2]}(\Te^*M),
\]
i.e.\ $\tilde G_\eop$ is quadratic in the fibers (and vanishes at $r=0$); here $h(t,\omega)^{a b}$ denotes the coefficients of the inverse metric of $h(t,\omega)$ in local coordinates $\omega^a$ on $\Sph^{n-1}$. We write
\[
  \Sigma = \{ \zeta\in \Te^*M\setminus o \colon G_\eop(\zeta)=0 \} = \Sigma^- \sqcup \Sigma^+
\]
for the characteristic set and its past (`$-$') and future (`$+$') components; we use the same symbols for the closures in $\ol{\Te^*}M\setminus o$. A change of variables computation shows that the Hamiltonian vector field of $G_\eop$ is
\begin{align}
  H_{G_\eop} &= -(\pa_\sigma G_\eop)r\pa_t + (\pa_\xi G_\eop)(r\pa_r+\sigma\pa_\sigma) + (\pa_{\eta_a}G_\eop)\pa_{\omega^a} \nonumber\\
    &\qquad + r(\pa_t G_\eop)\pa_\sigma - \bigl((r\pa_r+\sigma\pa_\sigma)G_\eop\bigr)\pa_\xi - (\pa_{\omega^a}G_\eop)\pa_{\eta_a} \nonumber\\
\label{EqDEHamOrig}
    &= 2\sigma r\pa_t + 2\xi(r\pa_r+\sigma\pa_\sigma) + 2 h^{a b}\eta_a\pa_{\omega^b} + r\pa_t h^{a b}\eta_a\eta_b\pa_\sigma + 2\sigma^2\pa_\xi - \pa_{\omega^c}h^{a b}\eta_a\eta_b\pa_{\eta_c} + H_{\tilde G_\eop},
\end{align}
where $H_{\tilde G_\eop}\in r\Ve(\Te^*M)$ is homogeneous of degree $1$ in the fibers (and indeed has fiber-linear coefficients relative to $r\pa_t,r\pa_r,\pa_{\omega^a},\zeta_\mu\zeta_\nu\pa_{\zeta_\lambda}$ where $\zeta=(\sigma,\xi,\eta_a)$). With $\pa_t$ being future timelike, we have $\Sigma^\pm=\Sigma\cap\{\pm\sigma>0\}$ near $\pa M$. Consider a point $\zeta\in\Sigma$ over $r=0$ where $H_{G_\eop}$ is fiber-radial; then $\eta=0$ and $\xi\sigma=c\sigma$, $\sigma^2=c\xi$ for some $c\in\R$. Since $\sigma$ cannot vanish, we get $c=\xi$ from the first equation, and the second equation gives $\xi^2=\sigma^2$. We define
\begin{subequations}
\begin{align}
\label{EqDERadIn}
  \cR_{\rm in}^\pm &= \{ (t,r,\omega;\sigma,\xi,\eta) = (t,0,\omega;\sigma,-\sigma,0) \colon \pm\sigma>0 \}, \\
\label{EqDERadOut}
  \cR_{\rm out}^\pm &= \{ (t,r,\omega;\sigma,\xi,\eta) = (t,0,\omega;\sigma,\sigma,0) \colon \pm\sigma>0 \},
\end{align}
\end{subequations}
and use the same symbols also for their closures in $\ol{\Te^*}M\setminus o$; we moreover set $\cR_{\rm in}=\cR_{\rm in}^\pm$ and $\cR_{\rm out}=\cR_{\rm out}^\pm$.

Let us define $\tilde H=H_{\tilde G_\eop}+r\pa_t h^{a b}\eta_a\eta_b\pa_\sigma\in r\Ve(\Te^*M)$. In the projective coordinates
\begin{equation}
\label{EqDECoordProj}
  \rho_\infty = \frac{1}{|\sigma|},\quad
  \hat\xi=\frac{\xi}{\sigma},\quad
  \hat\eta=\frac{\eta}{\sigma}
\end{equation}
near $\Sigma\cap\{\pm\sigma>0\}\subset\ol{\Te^*}M$, we then compute the rescaled Hamiltonian vector field
\begin{equation}
\label{EqDEHam}
\begin{split}
  \sigma^{-1}H_{G_\eop} &= 2 r\pa_t + 2\hat\xi(r\pa_r-\hat\xi\pa_{\hat\xi} - \hat\eta\pa_{\hat\eta} - \rho_\infty\pa_{\rho_\infty}) + 2\pa_{\hat\xi} \\
    &\quad\qquad + 2 h^{a b}\hat\eta_a\pa_{\omega^b} - \pa_{\omega^c}h^{a b}\hat\eta_a\hat\eta_b\pa_{\hat\eta_c} + \sigma^{-1}\tilde H.
\end{split}
\end{equation}
Restricted to $\Sigma^\pm$, its linearization at $\pa\cR_{\rm in}^\pm\subset\Se^*M$ (where $\hat\xi=-1$ and $\rho_\infty=0$) is $2 r\pa_t-2 r\pa_r + 2\hat\eta\pa_{\hat\eta} + 2\rho_\infty\pa_{\rho_\infty}$, while at $\pa\cR_{\rm out}^\pm$ (where $\hat\xi=1$) it is $2r\pa_t+2 r\pa_r - 2\hat\eta\pa_{\hat\eta} - 2\rho_\infty\pa_{\rho_\infty}$. Inside of $r=0$, the $\sigma^{-1}H_{G_\eop}$-flow describes distance $\pi$ propagation on a fixed fiber $\phi^{-1}(t)=\Sph^{n-1}$ (with metric $h(t)$) from $\cR_{\rm in}$ to $\cR_{\rm out}$; indeed, at $\rho_\infty=0=r$, we have
\[
  \frac12(\sigma^{-1}H_{G_\eop})|_{\Se^*_{\pa M}M} = (1-\hat\xi^2)\pa_{\hat\xi} - \hat\xi\hat\eta\pa_{\hat\eta} + h^{a b}\hat\eta_a\pa_{\omega^b} - \frac12\pa_{\omega^c}h^{a b}\hat\eta_a\hat\eta_b\pa_{\hat\eta_c}.
\]
Its integral curves satisfy $\hat\xi(s)=\tanh(s-s_0)$ for some $s_0\in\R$, and thus, on $\Sigma$, where $|\hat\eta|_{h^{-1}}^2=1-\hat\xi(s)^2=\cosh(s-s_0)^{-2}$, the curve $\omega(s)$ describes a geodesic on $(\Sph^{n-1},h(t))$ (with $t$ fixed) of length $\int_{-\infty}^\infty\frac{\dd s}{\cosh(s-s_0)}=\pi$. Thus, $\cR_{\rm in}^\pm$ and $\cR_{\rm out}^\pm$ are saddle points for the $\sigma^{-1}H_{G_\eop}$-flow, with integral curves tending from $r>0$ to $\cR_{\rm in}^\pm$, propagating distance $\pi$ along a fiber, and emanating from $\cR_{\rm out}^\pm$ into $r>0$.

Fix now a positive function $\rho_\infty\in S^{-1}_{\rm hom}(\Te^*M\setminus o)$ which equals $|\sigma|^{-1}$ near the characteristic set over a neighborhood of $\pa M$. If $\gamma\colon I\subseteq\R\to\Sigma^\pm$ is an integral curve of
\begin{equation}
\label{EqDRescHam}
  \sfH_{G_\eop} := \rho_\infty H_{G_\eop}
\end{equation}
in $r>0$, with image lying over a compact subset of $M$, so that $\gamma$ is not contained in $\cR_{\rm in}^\pm\cup\cR_{\rm out}^\pm$ and so that $\liminf_{s\searrow\inf I} r(\gamma(s))=0$ (resp.\ $\liminf_{s\nearrow\sup I} r(\gamma(s))=0$), then $\gamma(s)$ necessarily tends to $\cR_{\rm out}^\pm$ as $s\searrow\inf I=-\infty$ (resp.\ $\cR_{\rm in}^\pm$ as $s\nearrow\sup I=+\infty$); this follows from the source and sink nature of $\cR_{\rm out}^\pm$ and $\cR_{\rm in}^\pm$ in the radial direction.

\begin{rmk}[Flow-in and flow-out]
\label{RmkDFlowInOut}
  One can show that $\gamma(s)$ tends to a single point in $\cR_{\rm out}^\pm$ and $\cR_{\rm in}^\pm$ in the backward and forward direction. More precisely, the stable manifold of $\pa\cR_{\rm in}^\pm$ and the unstable manifold of $\pa\cR_{\rm out}^\pm$ are smooth $(n+1)$-dimensional submanifolds of $\Se^*M$ (whereas the unstable manifold of $\pa\cR_{\rm in}^\pm$ and the stable manifold of $\pa\cR_{\rm out}^\pm$ are both equal to the characteristic set over $\pa M$ minus the other radial set), and they correspond to conic coisotropic submanifolds of $\Te^*M\setminus o$. This follows from the stable/unstable manifold theorem in the form given in \cite[Theorem~4.1]{HirschPughShubInvariantManifolds} in a manner similar to the proof of \cite[Theorem~1.2]{MelroseWunschConic}. We leave the details to the reader.
\end{rmk}

\begin{lemma}[Re-focusing onto the curve of cone points]
\label{LemmaDEAway}
  For each $t_0\in I$, there exist $\delta>0$ and conic neighborhoods $U_{\rm out}$ and $U_{\rm in}\subset\Sigma^\pm\cap t^{-1}([t_0-\delta,t_0+\delta])$ of $\cR_{\rm out}^\pm\cap t^{-1}([t_0-\delta,t_0+\delta])$ and $\cR_{\rm in}^\pm\cap t^{-1}([t_0-\delta,t_0+\delta])$, respectively, inside of $\Sigma^\pm\cap t^{-1}([t_0-\delta,t_0+\delta])$ so that the following holds: if $\gamma\colon I\subseteq\R\to\Sigma^\pm$ is an integral curve of $\pm H_{G_\eop}$ with $t(\gamma(s))\in[t_0-\delta,t_0+\delta]$ for all $s\in I$, and if $\gamma(s_0)\in U_{\rm out}$ (resp.\ $\gamma(s_0)\in U_{\rm in}$) for some $s_0\in I$, then $\gamma(s)\notin U_{\rm in}$ (resp.\ $\gamma(s)\notin U_{\rm out}$) for all $s>s_0$ (resp.\ $s<s_0$).

  Furthermore, if $g$ is an invariant metric, i.e.\ $g=-\dd t^2+\dd r^2+h(t_0,\omega;\dd\omega)$, one can take $\delta=\infty$ and the neighborhoods $U_{\rm out},U_{\rm in}$ to be invariant under translations $(t,r,\omega;\sigma,\xi,\eta)\mapsto(t+c,r,\omega;\sigma,\xi,\eta)$, $c\in\R$, and dilations $(t,r,\omega;\sigma,\xi,\eta)\mapsto(\lambda t,\lambda r,\omega;\sigma,\xi,\eta)$, $\lambda>0$.\footnote{These are the lifts to $\Te^*M$ of translations $(t,r,\omega)\mapsto(t+c,r,\omega)$ and dilations $(t,r,\omega)\mapsto(\lambda t,\lambda r,\omega)$.}
\end{lemma}
\begin{proof}
  We give a qualitative proof which sidesteps the use of the stable/unstable manifold theorem (cf.\ Remark~\ref{RmkDFlowInOut}). It suffices to consider the `$+$' sign. Suppose $H_{G_\eop}r=2\xi r+H_{\tilde G_\eop}r=0$ (using~\eqref{EqDEHamOrig}) and $r>0$, and note that $H_{\tilde G_\eop}r\in r^2 P^{[1]}(\Te^*M)$. Then $H_{G_\eop}^2 r=2 r H_{G_\eop}\xi+H_{G_\eop}H_{\tilde G_\eop}r=4\sigma^2 r\bmod r^2 P^{[2]}(\Te^*M)$. On $\Sigma^+$, we have $\sigma^2\geq\xi^2+|\eta|^2_{h^{-1}}-\cO(r)P^{[2]}(\Te^*M)$, i.e.\ $|\sigma|$ dominates $|\xi|,|\eta|_{h^{-1}}$. Therefore, for any compact interval $J\subset\R$ there exists $r_0>0$ so that $H_{G_\eop}r=0$ at a point in $\Sigma^+\cap t^{-1}(J)$ with $0<r\leq r_0$ implies $H_{G_\eop}^2 r>0$; that is, $r$-level sets are strictly convex. Upon shrinking $r_0$ further if necessary, we moreover have
  \begin{equation}
  \label{EqDEAwayT}
    \sigma^{-1}H_{G_\eop}t\geq r\quad\text{on}\ \Sigma^+\cap t^{-1}(J)\cap r^{-1}([0,r_0)).
  \end{equation}
  (When $g$ is an invariant metric, the error terms in these calculations are absent, and one can take $r_0=\infty$.)

  Now by~\eqref{EqDEHam}, there exist conic neighborhoods $\tilde U_{\rm out}$ and $\tilde U_{\rm in}\subset\Sigma^+$ of $\cR_{\rm out}^+$ and $\cR_{\rm in}^+$, respectively, which we may take to be contained in $r<\frac{r_0}{e}$, so that
  \begin{equation}
  \label{EqDEAwayR}
  \begin{alignedat}{2}
    r&\leq \sigma^{-1}H_{G_\eop}r &&\leq 3 r\quad\text{on}\ \tilde U_{\rm out}, \\
    r&\leq -\sigma^{-1}H_{G_\eop}r &&\leq 3 r\quad\text{on}\ \tilde U_{\rm in}.
  \end{alignedat}
  \end{equation}
  Let $\gamma\subset\Sigma^+\cap r^{-1}(0)$ be an integral curve of $\sigma^{-1}H_{G_\eop}$. Upon quotienting out by the dilation action in the fibers of $\Te^*M\setminus o$, $\gamma(s)$ tends to $\cR_{\rm in}$, resp.\ $\cR_{\rm out}$ as $s\to-\infty$, resp.\ $s\to+\infty$, and indeed $\hat\xi(\gamma(s))$ is monotonically increasing; requiring $\tilde U_{\rm out}$, resp.\ $\tilde U_{\rm in}$ to be contained in a small neighborhood of $\hat\xi=1$, resp.\ $\hat\xi=-1$, the desired conclusion thus holds for such $\gamma$. For integral curves $\gamma\subset\Sigma^+\setminus r^{-1}(0)$ of $\sigma^{-1}H_{G_\eop}$, note that if $\gamma(s_0)\in\tilde U_{\rm out}$ (so in particular $r(\gamma(s_0))\leq\frac{r_0}{e}$), then $\frac{\dd}{\dd s}r(\gamma(s))|_{s=s_0}>0$ and therefore $r\circ\gamma(s)$ is monotonically increasing (and in particular $\gamma(s)$ does not enter $\tilde U_{\rm in}$) as $s\geq s_0$ increases as long as $r(\gamma(s))\leq r_0$; this latter bound is guaranteed to be satisfied if $r(\gamma(s_0))e^{3(s-s_0)}\leq r_0$ by~\eqref{EqDEAwayR}. For $s_1>s_0$ with $r(\gamma(s_1))=\frac{r_0}{e}$, the function $r\circ\gamma$ thus keeps increasing for $s\leq s_1+\frac{1}{3}$. But then $t(\gamma(s))\geq t(\gamma(s_1))+\frac{r_0}{e}(s-s_1)$ by~\eqref{EqDEAwayT}, so for $s=s_1+\frac{1}{3}$ we get $t(\gamma(s))\geq(t_0-\delta)+\frac{r_0}{3 e}$. Taking $\delta<\frac{r_0}{6 e}$, this exceeds $t_0+\delta$. We may then take $U_{\rm out}=\tilde U_{\rm out}\cap\Sigma^+\cap t^{-1}([t_0-\delta,t_0+\delta])$, similarly for $U_{\rm in}$.

  The second part follows from the \emph{global} null-bicharacteristic convexity of $r$ and the fact that the Hamiltonian vector field $H_{G_{\eop,t_0}}$ of the dual metric function of an invariant metric $g_{\eop,t_0}$ is invariant under translations in $t$ (since $G_{\eop,t_0}$ is) and homogeneous of degree $1$ with respect to dilations (since $G_{\eop,t_0}$ is homogeneous of degree $2$), and thus translations and dilations of null-bicharacteristics are again null-bicharacteristics.
\end{proof}

\subsection{Non-refocusing domains}
\label{SsDNrf}

We continue working on $(M,g)$, $g=r^2 g_\eop$, and recall $G_\eop(\zeta)=g_\eop^{-1}(\zeta,\zeta)$ and the notation $\sfH_{G_\eop}$ introduced in~\eqref{EqDRescHam}.

\begin{definition}[Non-refocusing domains]
\label{DefDNrf}
  We call a spacetime domain $\Omega\subset M$ (see Definition~\usref{DefD}) \emph{non-refocusing} (or \emph{$g$-non-refocusing} in order to make the metric explicit) if there does not exist a future null-geodesic $\gamma\colon I\to\bar\Omega\cap M^\circ$, with $I\subseteq\R$ an interval, so that $r(\gamma(s))\to 0$ both as $s\searrow\inf I$ and as $s\nearrow\sup I$.
\end{definition}

Equivalently, there does not exist an integral curve $\gamma\colon\R\to\Te^*_{\bar\Omega}M\cap \Sigma^\pm$ of $\pm\sfH_{G_\eop}$ so that $\gamma(s)$ tends to $\cR_{\rm out}^\pm$ as $s\searrow-\infty$ and to $\cR_{\rm in}^\pm$ as $s\nearrow+\infty$. Note that such curves $\gamma$ must be contained in $r>0$ due to the source-to-sink dynamics of the flow in $\Sigma^\pm\cap r^{-1}(0)$.

\begin{example}[Non-refocusing domains for invariant metrics]
\label{ExDNrfInv}
  If $g$ is an invariant metric (see~\eqref{EqDMetricModel}), then every spacetime domain inside $I_t\times[0,\infty)_r\times\Sph^{n-1}$ is non-refocusing; this follows from the second part of Lemma~\ref{LemmaDEAway}.
\end{example}

\begin{lemma}[Quantitative version of non-refocusing]
\label{LemmaDNrf}
  If $\Omega$ is a non-refocusing spacetime domain, then there exist conic neighborhoods $U_{\rm out}$ and $U_{\rm in}\subset\Sigma^\pm\cap\bar\Omega$ of $\cR_{\rm out}^\pm$ and $\cR_{\rm in}^\pm$, respectively, so that for every integral curve $\gamma=\gamma(s)$ of $\pm H_{G_\eop}$, $\gamma(s_0)\in U_{\rm out}$ (resp.\ $\gamma(s_0)\in U_{\rm in}$) implies $\gamma(s)\notin U_{\rm in}$ for all $s\geq s_0$ (resp.\ $\gamma(s)\notin U_{\rm out}$ for all $s\leq s_0$). One can moreover choose the sets $U_{\rm out},U_{\rm in}$ so that the estimates~\eqref{EqDEAwayR} hold on them.
\end{lemma}

For invariant metrics, we can take $U_{\rm out},U_{\rm in}$ to be the sets of Lemma~\ref{LemmaDEAway}.

\begin{proof}[Proof of Lemma~\usref{LemmaDNrf}]
  The estimates~\eqref{EqDEAwayR} hold as soon as $U_{\rm out}\subset\tilde U_{\rm out}$ and $U_{\rm in}\subset\tilde U_{\rm in}$, i.e.\ they hold for all sufficiently small neighborhoods $U_{\rm out}$ and $U_{\rm in}$ of $\cR_{\rm out}^\pm$ and $\cR_{\rm in}^\pm$, respectively.

  We work at fiber infinity $\Se^*M$, identify conic subsets of $\Te^*M\setminus o$ with subsets of $\Se^*M$, and consider integral curves of $\sfH_{G_\eop}$ in the (compact) set
  \[
    K:=\pa\Sigma^+\cap\Se^*_{\bar\Omega}M.
  \]
  Suppose that no open neighborhoods $U_{\rm out}$ and $U_{\rm in}\subset\Se^*M$ of $\pa\cR_{\rm out}^+$ and $\pa\cR_{\rm in}^+$ satisfy the desired condition. Then there exist sequences $\zeta_{\rm out,j}\in K\setminus\pa\cR_{\rm out}^+$ and $\zeta_{\rm in,j}\in K\setminus\pa\cR_{\rm in}^+$ converging to $\pa\cR_{\rm out}^+$ and $\pa\cR_{\rm in}^+$, respectively, and integral curves $\gamma_j\subset\pa\Sigma^+$ starting at $\zeta_{\rm out,j}$ and ending at $\zeta_{\rm in,j}$; thus, $r>0$ along $\gamma_j$. By the convexity of $r\circ\gamma$ near $\pa M$ established in the proof of Lemma~\ref{LemmaDEAway}, $r\circ\gamma_j$ is monotonically increasing at least until it reaches some value $r_0>0$ (depending on $\bar\Omega$). For each $j$, we may thus choose a point $\zeta_j\in\gamma_j\cap K\cap\{r\geq r_0\}$. Upon shifting the argument of $\gamma_j$, we may assume that $\zeta_j=\gamma_j(0)$; due to~\eqref{EqDEAwayR}, the domain of definition of $\gamma_j$ is then an interval $(a_j,b_j)$ with $a_j\searrow-\infty$ and $b_j\nearrow+\infty$.

  By compactness of $K\cap\{r\geq r_0\}$, we may assume that $\zeta_j\to\zeta\in K$. We claim that the maximal integral curve $\gamma\colon I\subseteq\R\to K$ of $\sfH_{G_\eop}$ with $\gamma(0)=\zeta$ tends to $\pa\cR_{\rm out}^+$ in the past and to $\pa\cR_{\rm in}^+$ in the future direction. If, say, $\gamma(s)$ did not tend to $\pa\cR_{\rm out}^+$ as $s\searrow\inf I$, then (by the discussion before Lemma~\ref{LemmaDEAway}) we would have $\inf_{s\geq 0}r(\gamma(s))=:r_1>0$; but then $\frac{\dd}{\dd s}t(\gamma(s))\geq\eps>0$ for some $\eps>0$ and all $s\geq 0$ by the timelike nature of $t$, which implies that $s_+:=\sup I\leq T/\eps$ where $T=\sup_{\bar\Omega}t-\inf_{\bar\Omega}t$. Consider now $\gamma_j([0,s_+])$ for large $j$ (so that $b_j>s_+$); by the continuous dependence of integral curves on initial conditions, this tends to $\gamma([0,s_+])$, and thus $\gamma_j(s_+)\geq\sup_{\bar\Omega}t-\eps$ and $\gamma_j(s_+)\geq\frac{r_1}{2}$ for any fixed $\eps>0$ when $j$ is sufficiently large. Exploiting the strict monotonicity of $t\circ\gamma_j$ in $r\geq\frac{r_1}{4}$, this implies that $t(\gamma_j(s))\geq\sup_{\bar\Omega}t$ for large enough $j$ and $s$, contradicting the choice of $\gamma_j$. Similarly, one shows that $\gamma(s)$ tends to $\pa\cR_{\rm in}^+$ as $s\nearrow\sup I$. The existence of $\gamma$ contradicts the non-refocusing assumption on $\Omega$. The proof is complete. 
\end{proof}

\begin{lemma}[Openness in the metric]
\label{LemmaDNrfOpen}
  Suppose that $\Omega\subset M$ is a $g$-non-refocusing spacetime domain. Then for all metrics $g'$ of the form~\eqref{EqDMetric} (for possibly different $h$ and $\tilde g$) which are sufficiently close to $g$ in $r^2\cC^1(\bar\Omega;S^2\,\Te^*M)$, the domain $\Omega$ is a spacetime domain in $(M,g')$, and it is $g'$-non-refocusing.
\end{lemma}
\begin{proof}
  In the notation of Definition~\ref{DefD}, the timelike nature of $\dd t_{\rm ini}$ and $\dd t_{\rm fin}$ persists for metrics that are close in $r^2\cC^0(\bar\Omega;S^2\,\Te^*M)$. The non-refocusing condition can be established via a proof by contradiction as in the proof of Lemma~\ref{LemmaDNrf}, using now the continuous dependence of integral curves on the metric (in $r^2\cC^1$) and the fact (following by inspection of the first step in the proof of Lemma~\ref{LemmaDEAway}) that the value $r_0>0$ so that the level sets $r=r'\in(0,r_0]$ are null-geodesically convex can be chosen uniformly for $g'$ in an $r^2\cC^1$-neighborhood of $g$.
\end{proof}

\begin{cor}[Small domains are non-refocusing]
\label{CorDNrfSmall}
  Let $p\in\cC$. Then there exists an open neighborhood $U\subset\cM$ of $p$ so that all spacetime domains $\Omega\subset M$ contained in (the preimage in $M$ of) $U$ are non-refocusing.
\end{cor}
\begin{proof}
  We have $p=\phi^{-1}(t_0)$ for some $t_0\in I$. Let $\delta>0$ be such that $(t_0-\delta,t_0+\delta)\subset I$. One possible choice of $U$ is $U=(t_0-\delta,t_0+\delta)\times[0,r_0)\times\Sph^{n-1}$ where $r_0>0$ is small enough such that $r=r_1\in(0,r_0]$ is null-geodesically convex in $t^{-1}([t_0-\delta,t_0+\delta])$.

  We present an alternative construction, which gives us the opportunity to introduce an important rescaling idea. In the coordinates~\eqref{EqDCollar}, consider the scaling map
  \begin{equation}
  \label{EqDNrfSmallScaling}
    S_\lambda \colon M':=\R\times[0,\infty)\times\Sph^{n-1} \ni (t',r',\omega') \mapsto (t,r,\omega)=(t_0+\lambda t',\lambda r',\omega')
  \end{equation}
  for $\lambda>0$. Fixing a precompact subset $U'\subset M'$, the map $S_\lambda|_{U'}\colon U'\to M$ is well-defined for $\lambda\in(0,\lambda_0)$ when $\lambda_0>0$ is sufficiently small, and since $S_\lambda^*\colon\frac{\dd t}{r}\mapsto\frac{\dd t'}{r'}$, $\frac{\dd r}{r}\mapsto\frac{\dd r'}{r'}$, $\dd\omega\mapsto\dd\omega'$, and $S_\lambda^*r=\lambda r'$, $S_\lambda^*(t-t_0)=\lambda t'$, we conclude that
  \[
    S_\lambda^*g_\eop - g_{\eop,t_0} \in \lambda \CI\bigl([0,\lambda_0);\CI(U';S^2\,\Te^*_{U'}M')\bigr).
  \]
  In particular, $r'{}^2 S_\lambda^*g_\eop\to r'{}^2 g_{\eop,t_0}$ in $r'{}^2\cC^1(U';S^2\,\Te^*_{U'}M')$, and the claim now follows from Lemma~\ref{LemmaDNrfOpen} and Example~\ref{ExDNrfInv} for $U$ equal to (the image in $\cM$ of) $S_\lambda(U')$ where $\lambda>0$ is chosen sufficiently small.
\end{proof}

\begin{cor}[Non-refocusing property upon enlarging domains]
\label{CorDNrfEnlarge}
  If $\Omega\subset M$ is non-refocusing, then there exists a non-refocusing spacetime domain $\Omega'\subset M$ containing $\bar\Omega$; and all spacetime domains $\Omega''\subset\Omega'$ are non-refocusing.
\end{cor}
\begin{proof}
  The final statement is an immediate consequence of the first. This in turn is a simple consequence of the existence of the open neighborhoods $U_{\rm out},U_{\rm in}$ in Lemma~\ref{LemmaDNrf}. One can also deduce this from Lemma~\ref{LemmaDNrfOpen} by considering the time $s$ flow $\Phi_s$ along a vector field $V\in\Vb(M)$ which is outward pointing at $\pa\bar\Omega$ and is a multiple of $\pa_t$ near $\pa M$; then $\Phi_s^*g\to g$ in $r^2\cC^1(\bar\Omega;S^2\,\Te^*M)$, so $\Phi_s^*g$ is non-refocusing for all small $s>0$.
\end{proof}

The null-geodesic flow on a non-refocusing domain has the following global behavior; see also Figure~\ref{FigDNrfFlow}.

\begin{lemma}[Null-geodesic flow on a non-refocusing domain]
\label{LemmaDNrfFlow}
  Let $\Omega\subset M$ be a non-refocusing domain, and let $\gamma\colon I\subseteq\R\to\pa\Sigma^+\cap\Se^*_{\bar\Omega}M$ be a maximal integral curve of $\sfH_{G_\eop}$. If $\gamma$ lies over $\Se^*_{\pa M}M$, then either $\gamma(I)\subset\pa\cR_{\rm out}^+\cup\pa\cR_{\rm in}^+$, or $\gamma(s)$ tends to $\pa\cR_{\rm in}^+$ as $s\searrow-\infty$ and to $\pa\cR_{\rm out}^+$ as $s\nearrow\infty$. Otherwise, $\gamma(I)\subset\Se^*_{M^\circ}M$, in which case either $\sup I=\infty$ and $\gamma(s)\to\pa\cR_{\rm in}^+$ as $s\to\infty$, or $\sup I<\infty$, $\sup I\in I$, and $\gamma(\sup I)$ lies in a final boundary hypersurface of $\Omega$ (but not on $\pa M$). Similarly, either $\inf I=-\infty$ and $\gamma(s)\to\pa\cR_{\rm out}^+$ as $s\to-\infty$, or $\inf I>-\infty$, $\inf I\in I$, and $\gamma(\inf I)$ lies in an initial boundary hypersurface of $\Omega$ (but not on $\pa M$). The same statements hold, mutatis mutandis, for maximal integral curves of $-\sfH_{G_\eop}$ in $\pa\Sigma^-$.
\end{lemma}
\begin{proof}
  In view of the (explicit) description of the $\sfH_{G_\eop}$-flow over $\pa M$, it suffices to consider the case that $\gamma(I)\subset\{r>0\}$. Now, for all $\eps>0$, there exists a constant $c_\eps>0$ so that for $s\in I$ with $r(\gamma(s))\geq\eps$ we have $\frac{\dd}{\dd s}t(\gamma(s))\geq c_\eps$; this follows from the timelike nature of $t$. So if $\sup I=\infty$, we must have $\liminf_{s\to\infty} r(\gamma(s))=0$ since $\sup_{\bar\Omega}t<\infty$ due to the compactness of $\bar\Omega$. As discussed before Lemma~\ref{LemmaDEAway}, this implies that $\gamma(s)\to\pa\cR_{\rm in}^+$. If $\sup I<\infty$, then $\gamma(s)$ cannot tend to $\pa\cR_{\rm in}^+$ as $s\nearrow\sup I$, and $\gamma$ also cannot lie over $\pa M$; it follows that $\liminf_{s\to\sup I} r(\gamma(s))>0$. Therefore, $\gamma$ is extendible as an integral curve on $M$, and the limiting point $p_+=\lim_{s\to\sup I}\gamma(s)$ must lie over $\pa\bar\Omega\setminus\pa M$ (and therefore $\sup I\in I$). Since $\gamma$ is a future null-bicharacteristic, $p_+$ must lie on a final boundary hypersurface.
\end{proof}

\begin{figure}[!ht]
\centering
\includegraphics{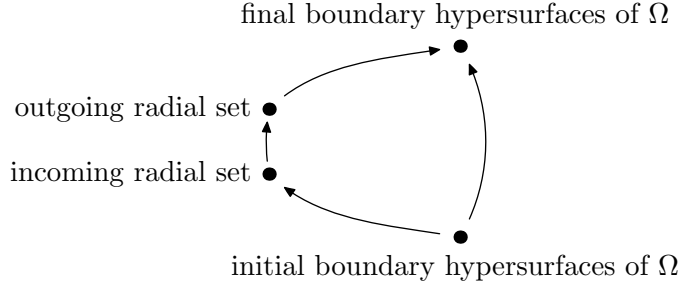}
\caption{Structure of the null-geodesic flow, lifted to $\Se^*M$, on a non-refocusing spacetime domain $\Omega$.}
\label{FigDNrfFlow}
\end{figure}

\subsection{Monotone functions on non-refocusing domains}
\label{SsDFn}

For the microlocal propagation of edge regularity into and out of the curve $\cC\subset\cM$ of conic points, certain threshold and monotonicity requirements must be satisfied for the edge regularity order $\sfs\in\CI(\Se^*M)$. Using the notation of~\S\ref{SsDNrf}, we thus show:

\begin{prop}[Existence of order functions]
\label{PropDFn}
  Let $\Omega\subset M$ be a non-refocusing spacetime domain. Let $\sfs_{\rm in},\sfs_{\rm out}\in\R$. Then there exists a function $\sfs\in\CI(\Se^*_{\bar\Omega}M)$ with the following properties:
  \begin{enumerate}
  \item $\sfs$ is constant near $\pa\cR_{\rm in}^\pm$ and near $\pa\cR_{\rm out}^\pm$;
  \item $\sfs>\sfs_{\rm in}$ at $\pa\cR_{\rm in}^\pm$ and $\sfs<\sfs_{\rm out}$ at $\pa\cR_{\rm out}^\pm$;
  \item $\pm\sfH_{G_\eop}\sfs\leq 0$ on $\pa\Sigma^\pm$.
  \end{enumerate}
\end{prop}

\begin{rmk}[Necessity of the non-refocusing assumption]
\label{RmkDFnNec}
  Proposition~\ref{PropDFn} has the following converse. \emph{If, for $\sfs_{\rm in}=0=\sfs_{\rm out}$, a function $\sfs\in\CI(\Se^*M)$ with the stated properties exists, then $\Omega$ is non-refocusing.} Indeed, if $\gamma$ is an integral curve of $\sfH_{G_\eop}$ in $\pa\Sigma^+\cap\Se^*_{\bar\Omega}M$, and $\gamma(s)$ tends to $\pa\cR_{\rm out}^+$ as $s\searrow-\infty$, then $\sfs(\gamma(s))$ is negative for $s\ll -1$; but since $\sfs\circ\gamma$ is monotonically decreasing, $\gamma(s)$ cannot tend to $\pa\cR_{\rm in}^+$ as $s\nearrow\infty$ since this would force $\sfs(\gamma(s))$ to be positive for $s\gg 1$.
\end{rmk}

\begin{rmk}[Explicit choices for invariant metrics]
\label{RmkDFnExplInv}
  For invariant metrics, so $\sigma^{-2}G_\eop=-1+\hat\xi^2+|\hat\eta|_{h(\omega)}^2$ in the notation~\eqref{EqDECoordProj}, one computes using~\eqref{EqDEHam} that $\frac12\sigma^{-1}H_{G_\eop}\hat\xi=1-\hat\xi^2=|\hat\eta|_{h(\omega)}^2$ on $\pa\Sigma^\pm$, so $\hat\xi$ is monotonically increasing along the $\pm\sfH_{G_\eop}$-flow in $\pa\Sigma^\pm$. We may thus take $\sfs=f(\hat\xi)$ where $f\in\CI([-1,1])$ is constant near $\pm 1$, satisfies $f'\leq 0$, and $f(-1)>\sfs_{\rm in}$ as well as $f(1)<\sfs_{\rm out}$.
\end{rmk}

\begin{proof}[Proof of Proposition~\usref{PropDFn}]
  We construct a function $\chi\in\CI(\Se^*_{\bar\Omega}M)$ which equals $0$ near $\pa\cR_{\rm in}^\pm$ and $1$ near $\pa\cR_{\rm out}^\pm$, and so that $\pm\sfH_{G_\eop}\chi\geq 0$ on $\pa\Sigma^\pm$. Then, for any $\eps>0$, the function
  \[
    \sfs := (\sfs_{\rm in}+\eps) - \max(\sfs_{\rm in}-\sfs_{\rm out}+2\eps,0)\chi
    = \begin{cases}
        \sfs_{\rm in}+\eps, & \sfs_{\rm in}\leq\sfs_{\rm out}, \\
        (\sfs_{\rm in}+\eps)(1-\chi) + (\sfs_{\rm out}-\eps)\chi, & \sfs_{\rm in}>\sfs_{\rm out},
      \end{cases}
  \]
  satisfies all requirements. We shall construct $\chi$ only on $\pa\Sigma^+$. The idea behind our construction is that $\chi(\zeta)$ should measure how close the integral curve $\gamma_\zeta(s)=e^{s\sfH_{G_\eop}}\zeta$ gets to $\pa\cR_{\rm out}^+$ for $s\leq 0$. Making this precise requires some care. First of all, we use Corollary~\ref{CorDNrfEnlarge} to pick non-refocusing spacetime domains $\Omega_1\supset\bar\Omega$ and $\Omega_2\supset\ol{\Omega_1}$; henceforth, we work on $\ol{\Omega_2}$. Thus, for example, $\pa\Sigma^+$ denotes the boundary at fiber infinity of $\{\zeta\in\Te^*_{\ol{\Omega_2}}M\setminus o\colon G_\eop(\zeta)=0\}$. We fix $U_{\rm out},U_{\rm in}\subset\Se^*_{\ol{\Omega_2}}M$ according to Lemma~\ref{LemmaDNrf}.

  \pfstep{Distance function; convexity property.} In local coordinates $t,r,\omega$ near $r=0$, and using the coordinates $\hat\xi,\hat\eta$ from~\eqref{EqDECoordProj} near $\pa\cR_{\rm out}^+$, set
  \[
    R^2 := r^2 + (\hat\xi-1)^2 + |\hat\eta|^2.
  \]
  Thus, $R^2$ is a local quadratic defining function of $\pa\cR_{\rm out}^+$. Let us denote by $\cO(R^3)$ smooth terms vanishing cubically at $\pa\cR_{\rm out}^+$; using~\eqref{EqDEHam}, we then find (with $\sfH_{G_\eop}=\sigma^{-1}H_{G_\eop}$ near $\pa\cR_{\rm out}^+$)
  \[
    \sfH_{G_\eop}(R^2) = 4 r^2 - 8(\hat\xi-1)^2 - 4|\hat\eta|^2 + \cO(R^3)
  \]
  Suppose this vanishes; then $|\hat\eta|^2\leq r^2+\cO(R^3)$, and thus
  \begin{align*}
    \sfH_{G_\eop}^2(R^2) &= \sfH_{G_\eop} \bigl(8 r^2-4(\hat\xi-1)^2 - 4 R^2 \bigr) + \cO(R^3) \\
      &= 32 r^2 + 32(\hat\xi-1)^2 + \cO(R^3) \\
      &\geq 16 R^2 + \cO(R^3).
  \end{align*}
  Therefore, there exists $R_0>0$ so that all level sets of $R^2\in(0,R_0^2]$ are strictly convex for the $\sfH_{G_\eop}$-flow. Upon decreasing $R_0$ if necessary, we can ensure that $\{R^2\leq R_0^2\}\subset U_{\rm out}$. For $R_1\in(0,R_0]$, define now
  \[
    d_{R_1} := \min\Bigl(\frac{R^2}{R_1^2},1\Bigr) \in\cC^0(\pa\Sigma^+)
  \]
  defined to be $1$ outside the coordinate chart. Note that $d_{R_1}$ is smooth on $\{d_{R_1}<1\}$.

  \pfstep{A non-smooth proxy $\chi_{R_1}$ for $\chi$.} Fix a smooth function $f\colon[0,1]\to[0,1]$ so that $f|_{[0,\frac14]}=1$, $f|_{[\frac12,1]}=0$, and $f'\leq 0$. Given $\zeta\in\pa\Sigma^+$ and $R_1\in(0,R_0]$, set
  \[
    d_{R_1,\rm min}(\zeta) := \inf_{s\leq 0} d_{R_1}(\gamma_\zeta(s)),\quad
    \chi_{R_1}(\zeta) := f(d_{R_1,\rm min}(\zeta)).
  \]
  Then $\chi_{R_1}(\zeta)=1$ when $d_{R_1}(\zeta)\leq\frac14$, which in particular holds in a neighborhood of $\pa\cR_{\rm out}^+$. If on the other hand $\zeta\in U_{\rm in}$, then since $\gamma_\zeta(s)\notin U_{\rm out}$ for all $s\leq 0$ by the non-refocusing property of $\Omega_2$, we have $d_{R_1,\rm min}(\zeta)=1$ and thus $\chi_{R_1}(\zeta)=0$. Note furthermore that $\chi_{R_1}(\gamma_s(\zeta))$ is an increasing function of $s$ since $d_{R_1,\rm min}(\gamma_s(\zeta))$ is a decreasing function. We proceed to analyze these functions in more detail.

  \pfsubstep{(i)}{Uniqueness of minimizers.} Let $R^2_{\rm in}=r^2+(\hat\xi+1)^2+|\hat\eta|^2$ and fix $\eps>0$ so that the annular region
  \begin{equation}
  \label{EqDFnAin}
    A_{\rm in}:=\{(t,r,\omega;\hat\xi,\hat\eta)\colon R^2_{\rm in}\in(\eps,3\eps)\} \subset \Se^*_{\ol{\Omega_2}}M
  \end{equation}
  is contained in $U_{\rm in}$. Since over $r=0$, the $\sfH_{G_\eop}$-flow in $\pa\Sigma^+$ flows from the source $\pa\cR_{\rm in}^+$ to the sink $\pa\cR_{\rm out}^+$ over a single fiber of $\pa M$, there exists $\delta>0$ so that all backwards null-bicharacteristics starting in the set
  \begin{equation}
  \label{EqDFnAout}
    A_{\rm out} := \bigl\{ d_{R_0} \in \bigl(\tfrac14,\tfrac34\bigr),\ r<\delta \bigr\} \cap \pa\Sigma^+ \cap \Se^*_{\ol{\Omega_1}}M
  \end{equation}
  enter $A_{\rm in}$, and thus $U_{\rm in}$, in finite time. See Figure~\ref{FigDFnBw}.%

  \begin{figure}[!ht]
  \centering
  \includegraphics{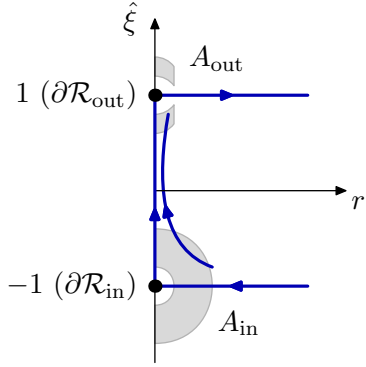}
  \caption{Illustration of the $\sfH_{G_\eop}$-flow near $r=0$. We only draw the variables $r,\hat\xi$; the coordinate $|\hat\eta|$ can be recovered over $r=0$ from $|\hat\eta|^2=1-\hat\xi^2$. Also indicated are the sets $A_{\rm in}$, $A_{\rm out}$ from~\eqref{EqDFnAin}--\eqref{EqDFnAout}.}
  \label{FigDFnBw}
  \end{figure}

  We now claim that if $R_1\in(0,R_0]$ is sufficiently small, then for all $\zeta\in\pa\Sigma^+\cap\Se^*_{\ol{\Omega_1}}M$ with $d_{R_1,\rm min}(\zeta)\in[\frac14,\frac34])$, the minimum of $0\geq s\mapsto d_{R_1}(\gamma_\zeta(s))$ is attained only once. If $d_{R_1}(\gamma_\zeta(s))=\min(\frac{R_0^2}{R_1^2}d_{R_0}(\gamma_\zeta(s)),1)$ attains a minimum value in $[\frac14,\frac34]$ at $s_0\leq 0$, then by the strict convexity of $d_{R_0}$, we have $-\frac{\dd}{\dd s}d_{R_0}(\gamma_\zeta(s))>0$ for $s\in(s_1,s_0)$ for some minimal $s_1<s_0$ and $d_{R_0}(\gamma_\zeta(s_1))=1$ (so in particular $\gamma_\zeta(s_1)\notin U_{\rm out}$)---unless $\gamma_\zeta(s_1)$ lies over an initial boundary hypersurface of $\Omega_2$, in which case there is nothing to prove. Now $r(\gamma_\zeta(s_0))\leq\sqrt{3 R_1/4}<\delta$ for small $R_1$, so by~\eqref{EqDEAwayR} we have $r(\gamma_\zeta(s))<\delta$ also for $s\in(s_1,s_0)$ since $\gamma_\zeta(s)$ remains in $U_{\rm out}$ for such $s$; by the intermediate value theorem and using that $d_{R_0}(\gamma_\zeta(s_0))=\frac{R_1^2}{R_0^2}d_{R_1}(\gamma_\zeta(s_0))\leq\frac{3 R_1^2}{4 R_0^2}<\frac14$ when $R_1$ is small, we have $d_{R_0}(\gamma_\zeta(s))\in(\frac14,\frac34)$ for some $s\in(s_1,s_0)$; so $\gamma_\zeta(s)\in A_{\rm out}$. Since this now implies that $\gamma_\zeta$ enters $U_{\rm in}$ in the past (unless it hits $\pa\ol{\Omega_2}$), the non-refocusing property implies that $\gamma_\zeta$ cannot enter $U_{\rm out}$ again, as desired.

  \pfsubstep{(ii)}{Smoothness of a minimizing argument.} Fix $\delta>0$ so that the maximal interval $(a(\zeta),b(\zeta))\subseteq\R$ of definition of all integral curves $\gamma_\zeta$, $\zeta\in\pa\Sigma^+\cap\Se^*_{\bar\Omega}M$, over $\bar\Omega$ can be extended at least to $(a(\zeta)-3\delta,b(\zeta)+3\delta)$ as integral curves over $\ol{\Omega_1}$, and similarly with $\Omega_2,\Omega_1$ in place of $\Omega_1,\Omega$. We may pick $\delta$ so small that moreover $\sfH_{G_\eop}d_{R_1}<\frac{1}{64\delta}$ on $\pa\Sigma^+\cap\Se^*_{\ol{\Omega_2}}M$, where $R_1$ is fixed according to the previous step.

  Let now $\zeta_0\in\pa\Sigma^+\cap\Se^*_{\ol{\Omega_1}}M$, and suppose that $(-3\delta,3\delta)\ni s\mapsto d_{R_1}(\gamma_{\zeta_0}(s))$ attains a local minimum at $s=0$, with $d_{R_1}(\gamma_{\zeta_0}(0))\in(\frac{1}{16},\frac{15}{16})$. The defining equation $\tilde d(\zeta,s):=\frac{\dd}{\dd s}d_{R_1}(\gamma_\zeta(s))=0$ for the minimum $s=s_{\rm min}(\zeta)$ satisfies $\pa_s\tilde d(\zeta_0,s)|_{s=s_{\rm min}(\zeta_0)}\neq 0$ due to the strict null-bicharacteristic convexity of $d_{R_1}$; by the implicit function theorem, $s_{\rm min}(\zeta)$ is therefore a smooth function of $\zeta$ near $\gamma_{\zeta_0}((-3\delta,3\delta))$. Since $s_{\rm min}(\gamma_\zeta(s))=s_{\rm min}(\zeta)-s$, the level set $s_{\rm min}=0$ is a smooth hypersurface in $\pa\Sigma^+$ transversal to $\gamma_{\zeta_0}$.

  Note that for $\zeta$ near $\gamma_{\zeta_0}([-2\delta,2\delta])$, we have
  \begin{equation}
  \label{EqDFnChiC11}
    \chi_{R_1}(\gamma_\zeta(z))=
      \begin{cases}
        f\bigl(d_{R_1}(\gamma_\zeta(s_{\rm min}(\zeta)))\bigr), & s_{\rm min}(\zeta)\leq z, \\
        f(d_{R_1}(\gamma_\zeta(z))), & s_{\rm min}(\zeta)>z.
      \end{cases}
  \end{equation}
  In particular, while this function is smooth for $s\neq s_{\rm min}(\zeta)$, it is only $\cC^{1,1}$ at $s=s_{\rm min}(\zeta)$.\footnote{As an illustrative example, the function $\chi_{R_1}(s)=\min_{s'\geq s} s'{}^2$, explicitly given by $\chi_{R_1}(s)=s_+^2$, is $\cC^{1,1}$ but no better.}

  \pfstep{Partial smoothing of $\chi_{R_1}$.} Fix $\phi\in\CI((-\delta,\delta);[0,\infty))$ with $\int\phi(z)\,\dd z=1$. For $\zeta\in\pa\Sigma^+\cap\Se^*_{\bar\Omega}M$, we now set
  \[
    \tilde\chi(\zeta) := \int_{-\delta}^\delta \phi(z)\chi_{R_1}(\gamma_\zeta(z))\,\dd z.
  \]

  In view of $\phi\geq 0$, and since $\chi_{R_1}(\gamma_\zeta(s))$ is increasing with $s$, the function $\tilde\chi(\gamma_\zeta(s))$ is an increasing smooth function of $s$. Furthermore, $\tilde\chi=0$ near $\pa\cR_{\rm in}^+$ due to the non-refocusing assumption. Due to $\int\phi(z)\,\dd z=1$ and since integral curves of $\sfH_{G_\eop}$ starting at a point with $d_{R_1}<\frac18$ remain in the set $d_{R_1}\leq\frac14$ for affine parameters in $(-\delta,\delta)$, we have $\tilde\chi=1$ near $\pa\cR_{\rm out}^+$.

  We proceed to analyze the regularity properties of $\chi$ on $\pa\Sigma^+\cap\Se^*_{\bar\Omega}M$. Consider $\zeta_0\in\pa\Sigma^+\cap\Se^*_{\bar\Omega}M$. We consider four (overlapping) cases. \emph{First}, suppose that there exists $s_-\leq 0$ so that $d_{R_1}(\gamma_{\zeta_0}(s_-))<\frac18$; then $\min_{s\leq z}d_{R_1}(\gamma_\zeta(s))<\frac14$ for $z\in(-\delta,\delta)$ and $\zeta=\zeta_0$, and thus for nearby $\zeta$, so $\tilde\chi(\zeta)=1$ for $\zeta$ near $\zeta_0$. \emph{Second}, if $d_{R_1}(\gamma_{\zeta_0}(s))>\frac78$ for all $s\leq 0$, then $d_{R_1}(\gamma_\zeta(s))\geq\frac34$ for $s\leq\delta$ and $\zeta$ near $\zeta_0$; thus $\tilde\chi(\zeta)=0$ for such $\zeta$. \emph{Third}, suppose the minimum of $d_{R_1}(\gamma_{\zeta_0}(s))$, $s\leq 0$, is attained at $s=s_-(\zeta_0)$, and $d_{R_1}(\gamma_{\zeta_0}(s_-(\zeta_0)))\in(\frac{1}{16},\frac{15}{16})$; suppose moreover that $s_-(\zeta_0)<-\delta$. If $\gamma_{\zeta_0}(s_-(\zeta_0))\in\Se^*_{\Omega_2}M$ (i.e.\ it does \emph{not} lie over an initial hypersurface of $\Omega_2$), then by an implicit function theorem argument as above, $s_-(\zeta)$ depends smoothly on $\zeta$ near $\zeta_0$, and in particular satisfies $s_-(\zeta)<-\delta$; therefore $\tilde\chi(\zeta)=f(d_{R_1}(\gamma_\zeta(s_-(\zeta))))$ is smooth for $\zeta$ near $\zeta_0$. We defer the study of the exceptional set
  \begin{equation}
  \label{EqDFnExc}
    E := \bigl\{ \zeta\in\pa\Sigma^+\cap\Se^*_{\bar\Omega}M \colon \gamma_\zeta(s_-(\zeta))\in\Se^*_{\pa\ol{\Omega_2}}M,\ d_{R_1}(\gamma_{\zeta}(s_-(\zeta)))\in\bigl(\tfrac{1}{16},\tfrac{15}{16}\bigr) \bigr\}
  \end{equation}
  to the last step of the construction. (At $E$, the function $\tilde\chi$ may fail to be smooth: $s_-(\zeta)$ may only be Lipschitz at $E$ since small modifications of $\zeta\in E$ to a nearby $\zeta'$ may lead to $\gamma_{\zeta'}(s_-(\zeta'))$ becoming an interior point, whereas other modifications remain in $E$.)

  The \emph{fourth} and final case is that the minimum of $d_{R_1}(\gamma_{\zeta_0}(s))$ is attained at $s=s_-(\zeta_0)\in[-\delta,0]$ and lies in the interval $(\frac{1}{16},\frac{15}{16})$. Denote by $s_{\rm min}(\zeta_0)$ the value of $s\in[-2\delta,2\delta]$ at which $d_{R_1}(\gamma_{\zeta_0}(s))$ is minimal. If $s_{\rm min}(\zeta_0)>\delta$, then $\min_{s\leq z}d_{R_1}(\gamma_{\zeta_0}(s))=d_{R_1}(\gamma_{\zeta_0}(z))$, so $\tilde\chi(\zeta)=\int \phi(z)f(d_{R_1}(\gamma_{\zeta}(z)))\,\dd z$ for $\zeta=\zeta_0$ and also nearby; this is smooth in $\zeta$. If, on the other hand, $s_{\rm min}(\zeta_0)\in[-\delta,\delta]$, then the function $\zeta\mapsto s_{\rm min}(\zeta)\in(-\frac52\delta,\frac52\delta)$ is well-defined and smooth for $\zeta$ in a sufficiently small neighborhood of $\gamma_{\zeta_0}([-\delta,\delta])$, and $\frac{\dd}{\dd s}s_{\rm min}(\gamma_\zeta(s))=-1$. We parameterize points $\zeta$ in such a neighborhood by means of the value $s_{\rm min}(\zeta)$ and the unique point $\zeta^\perp$ along $\gamma_\zeta$ lying on the level set $\{s_{\rm min}=s_{\rm min}(\zeta_0)\}$ (which is transversal to $\gamma_{\zeta_0}((-\frac32\delta,\frac32\delta))$). Define for $(\zeta^\perp,s_{\rm min})$ near $(\gamma_{\zeta_0}(s_{\rm min}(\zeta_0)),s_{\rm min}(\zeta_0))$ (i.e.\ for $\zeta$ near $\zeta_0$) and for $z\in(-\delta,\delta)$ the continuous function
  \[
    q(\zeta^\perp,s_{\rm min},z) := \chi_{R_1}(\gamma_{(\zeta^\perp,s_{\rm min})}(z)) = 
      \begin{cases} 
        f\bigl(d_{R_1}(\gamma_{(\zeta^\perp,s_{\rm min})}(s_{\rm min}))\bigr), & s_{\rm min}\leq z, \\
        f\bigl(d_{R_1}(\gamma_{(\zeta^\perp,s_{\rm min})}(z))\bigr), & s_{\rm min}>z;
      \end{cases}
  \]
  cf.\ \eqref{EqDFnChiC11}. Since $d_{R_1}<\frac{15}{16}+\frac{1}{64\delta}\cdot \frac52\delta<1$ for all arguments at which it is evaluated here, the function $q$ is smooth away from the `diagonal' $s_{\rm min}=z$. Moreover, the derivative of $q$ along any finite power of $\pa_{s_{\rm min}}+\pa_z$ remains bounded as well; therefore,
  \[
    \tilde\chi(\zeta^\perp,s_{\rm min}) = \int_{-\delta}^\delta q(\zeta^\perp,s_{\rm min},z)\phi(z)\,\dd z
  \]
  is smooth, as one can write
  \[
    \pa_{s_{\rm min}}\tilde\chi=\int \phi(z)(\pa_{s_{\rm min}}+\pa_z)q(\zeta^\perp,s_{\rm min},z)\,\dd z - \int \phi'(z)q(\zeta^\perp,s_{\rm min},z)\,\dd z,
  \]
  with both integrals defining continuous functions; similarly for higher derivatives. Smoothness in $\zeta^\perp$ is clear. (This can be viewed as a simple instance of \cite[Theorem~2.5.14]{HormanderFIO1}.)

  \pfstep{Cutting out the exceptional set.} We have shown that $\tilde\chi\in\CI(\pa\Sigma^+\cap\Se^*_{\bar\Omega}M\setminus E)$. Let $E_{\rm ini}=\{\gamma_\zeta(s_-(\zeta))\colon\zeta\in E\}\subset E$ in the notation used in~\eqref{EqDFnExc}. By definition of $E$, this set is disjoint from a neighborhood of $\pa\cR_{\rm out}^+$; but since $E_{\rm ini}\subset U_{\rm out}$, integral curves $\gamma_{\zeta_{\rm ini}}(s)$, $s\geq 0$, starting at points $\zeta_{\rm ini}\in E_{\rm ini}$ do not enter $U_{\rm in}$. Since $\sfH_{G_\eop}$ does not have any critical points in $\pa\Sigma^+\cap\Se^*_{\ol{\Omega_2}}M\setminus(\cR_{\rm out}^+\cup\cR_{\rm in}^+)$, there exists $\bar s<\infty$ so that $\gamma_{\zeta_{\rm ini}}(s)$ lies on a final boundary hypersurface of $\ol{\Omega_2}$ for some $s\leq\bar s$. Therefore, if $U_{\rm ini}\subset\pa\Sigma^+\cap\Se^*_{\pa\ol{\Omega_2}}M$ is a sufficiently small open neighborhood of $E_{\rm ini}$, we can use $(\zeta_{\rm ini},s)$ lying in an open subset of $U_{\rm ini}\times(0,\bar s)$ as smooth coordinates on an open neighborhood $U$ of $E\cap\Se^*_{\Omega_1}M$ inside of $\pa\Sigma^+\cap\Se^*_{\Omega_1}M$. Let $\tilde\psi\in\CIc(U_{\rm ini})$ be $1$ near $E_{\rm ini}$; then the function $\psi\in\CI(\pa\Sigma^+\cap\Se^*_{\bar\Omega})$ defined by $\psi(\zeta_{\rm ini},s):=1-\tilde\psi(\zeta_{\rm ini})$ on $U$ and $\psi=1$ outside of $U$ is smooth, invariant under the $\sfH_{G_\eop}$-flow over $\pa\Sigma^+\cap\Se^*_{\bar\Omega}M$, and therefore
  \[
    \chi := \psi\tilde\chi \in \CI(\pa\Sigma^+\cap \Se^*_{\bar\Omega}M)
  \]
  satisfies all requirements. This completes the construction.
\end{proof}

\section{Solvability and uniqueness theory}
\label{SSU}

Let $(M,g)$ denote a spacetime with a timelike curve $\cC=\pa M/{\sim}$ of cone points (Definition~\ref{DefDSpacetime}); write $r\in\CI(M)$ for a defining function of $\pa M$. Let $\cE\to M$ be a smooth complex vector bundle. We consider (principally scalar) wave operators
\begin{equation}
\label{EqSUOpMem}
  P \in r^{-2}\Diffe^2(M;\cE),\qquad \sigma^2(P)(z,\zeta)=g(z)^{-1}(\zeta,\zeta)\otimes\Id_\cE;
\end{equation}
here $\sigma^2(P)$ is the principal symbol and $z\in M^\circ$, $\zeta\in T_z^*M^\circ$; equivalently, $\sigmae^2(r^2 P)=G_\eop$ where $G_\eop(\zeta)=g_\eop^{-1}(\zeta,\zeta)$, $g_\eop=r^2 g$, is the (edge) dual metric function.

In the collar neighborhood~\eqref{EqDCollar} of $\pa M$, and fixing a bundle isomorphism of $\cE$ over a neighborhood of $r=0$ with the pullback of $\cE|_{\pa M}$ along $(t,r,\omega)\mapsto(t,\omega)$, every operator $P$ satisfying~\eqref{EqSUOpMem} can be written as
\begin{equation}
\label{EqSUOp}
\begin{split}
  P &= -D_t^2 + D_r^2 - \frac{n-1+b(t,\omega)}{r}i D_r + r^{-2}\Delta_{h(t)} \\
    &\quad\qquad + a(t,\omega)r^{-1}i D_t + r^{-2}W(t,\omega;D_\omega) + r^{-2}V(t,\omega) + \tilde P
\end{split}
\end{equation}
where $a,b,V\in\CI(\pa M;\End(\cE))$, and in local coordinates on $\Sph^{n-1}$ and trivializations of $E|_{\pa M}$, the operator $W$ takes the form $W(t,\omega;D_\omega)=\sum W^j(t,\omega)D_{\omega^j}$ where $W^j\in\CI(\pa M;\End(\cE))$, $j=1,\ldots,n-1$. Furthermore, in case $\cE$ is not the trivial bundle, we write $\Delta_{h(t)}$ for any principally scalar operator on $\Sph^{n-1}$ with principal symbol equal to the dual metric function of $h(t)$; different choices can be absorbed into changes of $W,V$. Examples of such operators include $P=\Box_g$ on functions\footnote{In this case, one has $a=b=W=V=0$ due to the chosen normalization of $b$ here.} or tensors; see~\S\ref{SEx}.

Unless specified otherwise, we use the metric volume density $|\dd g|$ and any fixed smooth positive definite fiber inner product\footnote{Different choices lead to the same Sobolev spaces of distributions with support in a fixed compact subset of $M$, up to equivalence of norms.} on $\cE$ to define $L^2$ and edge Sobolev spaces.

We define the normal operators of $P$ in~\S\ref{SsSUN}. In~\S\ref{SsSUPr} we prove microlocal propagation results near the radial sets over $\pa M$. Following the edge-local solvability of $P u=f$ established in~\S\ref{SsSULoc} and the inversion of the reduced normal operator in~\S\ref{SsSUI}, we assemble all ingredients in~\S\ref{SsSUe} to establish the central result of the paper, Theorem~\ref{ThmSUeNonrf}. Its improvement to higher b-regularity is given in~\S\ref{SsSUb}, and applications to initial value problems are presented in~\S\ref{SsSUIVP}.

\subsection{Normal operators and spectral data}
\label{SsSUN}

For $t_0\in I$, we define by
\begin{equation}
\label{EqSUNe}
\begin{split}
  N_{\eop,t_0}(r^2 P) &:= -(r' D_{t'})^2 + r'{}^2\Bigl(D_{r'}^2 - \frac{n-1+b(t_0,\omega)}{r'}i D_{r'}\Bigr) + \Delta_{h(t_0)} \\
    &\qquad + a(t_0,\omega)r' i D_{t'} + W(t_0,\omega;D_\omega) + V(t_0,\omega) \\
      &\in \Diff_{\eop,\rm I}^2({}^+N\phi^{-1}(t_0);\cE_{t_0})
\end{split}
\end{equation}
the edge normal operator of $r^2 P$ at the fiber $\phi^{-1}(t_0)\subset\pa M$; we recall here that ${}^+N\phi^{-1}(t_0)=\R_{t'}\times[0,\infty)_{r'}\times\Sph^{n-1}_\omega$ where $t'=\dd t,r'=\dd r$, and we write $\cE_{t_0}$ for the pullback of $\cE|_{\phi^{-1}(t_0)}$ (which is a vector bundle over the zero section of ${}^+N\phi^{-1}(t_0)\to\phi^{-1}(t_0)$) along the map $(t',r',\omega)\mapsto(0,0,\omega)$. The model of $P$ at $\phi^{-1}(t_0)$ is then
\[
  N_{\eop,t_0}(P) := r'{}^{-2}N_{\eop,t_0}(r^2 P).
\]
The spectral family of $N_{\eop,t_0}(r^2 P)$, formally obtained by acting on sections of $\cE_{t_0}$ of the form $e^{-i\sigma t'}u(r',\omega)$, is denoted
\begin{equation}
\label{EqSUNeFam}
\begin{split}
  N_{\eop,t_0}(r^2 P,\sigma) &:= -(r'\sigma)^2 + r'{}^2\Bigl(D_{r'}^2 - \frac{n-1+b(t_0,\omega)}{r'}i D_{r'}\Bigr) + \Delta_{h(t_0)} \\
    &\qquad - i a(t_0,\omega)r'\sigma + W(t_0,\omega;D_\omega) + V(t_0,\omega) \\
    &\in \Diff^2([0,\infty)_{r'}\times\Sph^{n-1};\cE_{t_0}).
\end{split}
\end{equation}
Exploiting the invariance of this operator under scaling $(r',\sigma)\mapsto(\lambda r',\sigma/\lambda)$ for $\lambda>0$, specifically for $\lambda=|\sigma|$, we introduce $\hat r:=r'|\sigma|$ and $\hat\sigma=\frac{\sigma}{|\sigma|}$; this gives rise to the \emph{reduced normal operator}
\begin{equation}
\label{EqSUNeFamRed}
\begin{split}
  \hat N_{\eop,t_0}(r^2 P,\hat\sigma) &= -(\hat r\hat\sigma)^2 + \hat r^2\Bigl(D_{\hat r}^2-\frac{n-1+b(t_0,\omega)}{\hat r}i D_{\hat r}\Bigr) + \Delta_{h(t_0)} \\
    &\qquad - i a(t_0,\omega)\hat r\hat\sigma + W(t_0,\omega;D_\omega) + V(t_0,\omega) \\
    &\in (1+\hat r)^2 \Diff_{\bop,\scop}^2(\hat X;\cE_{t_0}),\qquad \hat X:=[0,\infty]_{\hat r}\times\Sph^{n-1}.
\end{split}
\end{equation}
Conversely, as in~\eqref{EqEScale1}--\eqref{EqEScale2}, we have
\begin{equation}
\label{EqSUNRedScaleOp}
  N_{\eop,t_0}(r^2 P,\sigma) = (\hat M_{|\sigma|})_* \hat N_{\eop,t_0}\Bigl(r^2 P,\frac{\sigma}{|\sigma|}\Bigr),\qquad
  \hat M_\lambda \colon (\hat r,\omega) \mapsto (\hat r/\lambda,\omega),
\end{equation}
We further set
\begin{equation}
\label{EqSUNRedResc}
  N_{\eop,t_0}(P,\sigma) := r'{}^{-2}N_{\eop,t_0}(r^2 P,\sigma),\qquad
  \hat N_{\eop,t_0}(P,\hat\sigma) := \hat r^{-2}\hat N_{\eop,t_0}(r^2 P,\hat\sigma).
\end{equation}

The b-normal operator of $r^2 P$ is a family (in $t_0$) of dilation-invariant operators\footnote{The usual definition of the b-normal operator of a b-differential operator such as $r^2 P$ is that it is an $r'$-dilation-invariant operator on ${}^+N\pa M=I_t\times[0,\infty)_{r'}\times\Sph^{n-1}$; since $r^2 P$ is an edge operator, this does not involve any $t$-derivatives though.} on $[0,\infty)_{r'}\times\Sph^{n-1}$,
\begin{equation}
\label{EqSUNb}
\begin{split}
  N_{\bop,t_0}(r^2 P) &= r'{}^2\Bigl(D_{r'}^2 - \frac{n-1+b(t_0,\omega)}{r'}i D_{r'}\Bigr) + \Delta_{h(t_0)} + W(t_0,\omega;D_\omega) + V(t_0,\omega) \\
    &\in \Diff_{\bop,\rm I}^2([0,\infty)_{r'}\times\Sph^{n-1};\cE_{t_0});
\end{split}
\end{equation}
it is thus also equal to $N_{\eop,t_0}(r^2 P,0)$. Exploiting its dilation-invariance, we formally pass to the Mellin transform in $r'$, i.e.\ formally replacing $r'\pa_{r'}$ by multiplication by $\xi\in\C$, obtaining
\begin{equation}
\label{EqSUNbMellin}
  N_{\bop,t_0}(r^2 P,\xi) = \Delta_{h(t_0)} - \xi^2 - (n-2+b(t_0,\omega))\xi + W(t_0,\omega;D_\omega) + V(t_0,\omega) \in \Diff^2(\Sph^{n-1};\cE_{t_0}).
\end{equation}
This is a holomorphic family of elliptic operators, and it is invertible for bounded $|\Re\xi|$ when $|\Im\xi|$ is sufficiently large; the inverse (as a family of operators $L^2(\Sph^{n-1};\cE_{t_0})\to H^2(\Sph^{n-1};\cE_{t_0})$, say) is thus meromorphic. The set of its poles in any precompact open subset of $\C$ whose boundary does not contain any poles of $N_{\bop,t_0}(r^2 P,\xi)^{-1}$ varies continuously with $t_0$.

\begin{definition}[Non-indicial weights]
\label{DefSUNInd}
  A number $\ell\in\R$ is called a \emph{$P$-non-indicial weight at $t_0\in I$} if for all $\xi\in\C$ with $\Re\xi=\ell$ the operator $N_{\bop,t_0}(r^2 P,\xi)$ is invertible. If $\Omega\subset M$ is a spacetime domain, we call $\ell$ a \emph{$P$-non-indicial weight on $\Omega$} if it is a $P$-non-indicial weight at all $t_0\in I$ with $\phi^{-1}(t_0)\subset\bar\Omega$.
\end{definition}

Since $\bar\Omega$ is compact, the set of $P$-non-indicial weights is open (though it may be empty) in view of the observations following~\eqref{EqSUNbMellin}. In many applications, the operator $N_{\eop,t_0}(r^2 P)$ is $t_0$-independent, and thus so are the operators~\eqref{EqSUNeFam}, \eqref{EqSUNeFamRed}, \eqref{EqSUNb}, \eqref{EqSUNbMellin}; in this case, the set of $P$-non-indicial weights is the complement of a discrete subset of $\R$.

The poles of $N_{\bop,t_0}(r^2 P,\xi)^{-1}$ are related to weights and asymptotic expansions at $\hat r=0$ or $r=0$ for solutions of equations involving $\hat N_{\eop,t_0}(r^2 P,\hat\sigma)$ or $P$ itself. We also need to capture the decay rates of outgoing/incoming solutions of $\hat N_{\eop,t_0}(r^2 P,\hat\sigma)$ for $\hat\sigma=\pm 1$ at $\hat r=\infty$, as they will be related to scattering decay and edge regularity orders. As motivation, we compute that the action of the $\hat r$-derivatives of $\hat N_{\eop,t_0}(P,\hat\sigma)$, i.e.\ $-1-\pa_{\hat r}^2-\frac{n-1+b}{\hat r}\pa_{\hat r}-\frac{i a\hat\sigma}{\hat r}$, on $e^{\pm i\hat\sigma\hat r}\hat r^{-\alpha}$ (with `$+$' for outgoing and `$-$' for ingoing spherical waves) produces $e^{\pm i\hat\sigma \hat r}(\pm i\hat\sigma C_\pm(t_0,\omega)\hat r^{-\alpha-1}+\cO(\hat r^{-\alpha-2}))$ as $\hat r\to\infty$ where $C_\pm=2\alpha-(n-1+b\pm a)$ vanishes for $\alpha:=\frac{n-1}{2}+\frac{b\pm a}{2}$.

\begin{definition}[Threshold quantities]
\label{DefSUNThr}
  For $t_0\in I$, we set\footnote{These quantities are related to those in \cite[Definition~4.7]{HintzConicProp}, but unlike in the reference we do not require the choice of a fiber inner product here; see Lemma~\ref{LemmaSUPrThr} below.}
  \begin{align*}
    \vartheta_{\rm out}(t_0) &:= \frac12\min_{\omega\in\Sph^{n-1}} \min\bigl(\Re\spec(b(t_0,\omega) + a(t_0,\omega))\bigr), \\
    \vartheta_{\rm in}(t_0) &:=  \frac12\max_{\omega\in\Sph^{n-1}} \max\bigl(\Re\spec(b(t_0,\omega) - a(t_0,\omega))\bigr).
  \end{align*}
  For a spacetime domain $\Omega\subset M$, set $T=\{t_0\in I\colon\phi^{-1}(t_0)\subset\bar\Omega\}$ and
  \begin{equation}
  \label{EqSUNThrDom}
    \vartheta_{\rm out}(\Omega) := \vartheta_{\rm out}(T) := \min_{t_0\in T} \vartheta_{\rm out}(t_0),\qquad
    \vartheta_{\rm in}(\Omega) := \vartheta_{\rm in}(T) := \max_{t_0\in T} \vartheta_{\rm in}(t_0).
  \end{equation}
\end{definition}

\subsection{Microlocal edge propagation estimate}
\label{SsSUPr}

Near $\pa M$, we work in the base coordinates~\eqref{EqDCollar} and the fiber coordinates~\eqref{EqDECovec} on $\Te^*M$; we set $\sfH_{G_\eop}=|\sigma|^{-1}H_{G_\eop}$. We recall the radial sets $\cR_{\rm in}^\pm$, $\cR_{\rm out}^\pm$ from~\eqref{EqDERadIn}--\eqref{EqDERadOut}.

\begin{lemma}[Threshold quantities and symbols of imaginary parts]
\label{LemmaSUPrThr}
  Let $\eps>0$. Then there exists a smooth positive definite fiber inner product on $\cE$ so that
  \begin{subequations}
  \begin{equation}
  \label{EqSUPrThrOut}
    \sigmae^1\Bigl(r^2\frac{P-P^*}{2 i\sigma}\Bigr)\Big|_{\cR_{\rm out}^\pm\cap t^{-1}(t_0)} < -2 \vartheta_{\rm out}(t_0) + \eps
  \end{equation}
  for all $t_0\in I$ as self-adjoint bundle endomorphisms of (the pullback to $\Te^*_{\pa M}M$ of) of $\cE$, where we use the metric volume density $|\dd g|$ and the fiber inner product on $\cE$ to define $P^*$. Similarly, there exists a smooth positive definite fiber inner product on $\cE$ so that
  \begin{equation}
  \label{EqSUPrThrIn}
    \sigmae^1\Bigl(r^2\frac{P-P^*}{2 i\sigma}\Bigr)\Big|_{\cR_{\rm in}^\pm\cap t^{-1}(t_0)} < 2 \vartheta_{\rm in}(t_0) + \eps.
  \end{equation}
  \end{subequations}
\end{lemma}
\begin{proof}
  The symbol of $r^2\frac{P-P^*}{2 i\sigma}$ over $\phi^{-1}(t_0)$ can be computed entirely in terms of $N_{\eop,t_0}(r^2 P)$, cf.\ \eqref{EqSUNe}. Since $\Delta_{h(t_0)}-\Delta_{h(t_0)}^*$ is a first order differential operator on $\Sph^{n-1}$, its principal symbol vanishes at the zero section of $T^*\Sph^{n-1}$, and thus at $\cR_{\rm out}^\pm$ and $\cR_{\rm in}^\pm$. The same argument applies to $W-W^*$; and also $V$ does not contribute to~\eqref{EqSUPrThrOut}--\eqref{EqSUPrThrIn}, being a zeroth order operator. Since $-D_{t'}^2+D_{r'}^2-\frac{n-1}{r'}i D_{r'}$ is symmetric with respect to the volume density $|r'{}^{n-1}\,\dd t'\,\dd r'\,\dd h(t_0)|$, we compute at a point $(t,r,\omega;\sigma,\xi,\eta)$ with $r=0$, $\eta=0$,
  \[
    \sigmae^1\Bigl(r^2\frac{P-P^*}{2 i\sigma}\Bigr) = -\frac{b(t_0,\omega)+b(t_0,\omega)^*}{2}\frac{\xi}{\sigma} - \frac{a(t_0,\omega)+a(t_0,\omega)^*}{2}.
  \]
  At $\cR_{\rm out}^\pm$, resp.\ $\cR_{\rm in}^\pm$, this equals $-\Re(a+b)$, resp.\ $\Re(b-a)$. By \cite[Proposition~B.1]{HintzNonstat}, there exists a positive definite inner product on the pullback of $\cE$ along $\Se^*M\to M$ so that $-\Re(a+b)$, at each point of $\pa\cR_{\rm out}^\pm$, lies within distance $\eps$ of the convex hull of the real part of $\spec(a+b)$. Since the fibers of $\pa\cR_{\rm out}^\pm\to M$ are points, this inner product can be taken to be a pullback of an inner product on $M$. By definition of $\vartheta_{\rm out}(t_0)$, this implies~\eqref{EqSUPrThrOut}. The same arguments apply to~\eqref{EqSUPrThrIn}.
\end{proof}

\begin{prop}[Edge propagation near the incoming radial set]
\label{PropSUPrIn}
  Let $t_-<t_+$ and $-1<\tau_1<\tau_2<\tau_3$; in the notation~\eqref{EqSUNThrDom}, set $\vartheta_{\rm in}:=\vartheta_{\rm in}([t_-,t_+])$. Let $\psi\in\CI(\R)$ be identically $0$ on $(-\infty,0]$, positive on $(0,1)$, and equal to $1$ on $[1,\infty)$. Write $\psi_1,\psi_2\in\CIc([0,\infty);[0,1])$ for functions which equal $1$ near $0$, with $\psi_2=1$ near $\supp\psi_1$. Put
  \begin{equation}
  \label{EqUPrInCutoff}
    \chi_j(t,r) = \psi\Bigl[\Bigl(\tau_{j+1}-\frac{t-t_+}{r}\Bigr)/(\tau_{j+1}-\tau_j)\Bigr]\psi_j(r),\qquad j=1,2,
  \end{equation}
  For any open neighborhood $U$ of $\pa\cR_{\rm in}^\pm\cap t^{-1}([t_-,t_+])$, one can choose $\psi_1,\psi_2$, resp.\ operators $B,G,E\in\Psie^0(M)$ with supports, resp.\ operator wave front sets contained in (the projection to $M$ of) $U$ so that $B$ is elliptic at $U\cap\pa\cR_{\rm in}^\pm$, furthermore $\pm\sfH_{G_\eop}r<0$ on $\WFe'(E)$, and so that for all $s,s_0,\ell\in\R$ with $s>s_0>-\frac12+\ell+\vartheta_{\rm in}$ there exists a constant $C>0$ so that
  \begin{equation}
  \label{EqSUPrInEst}
    \| B \chi_1 u \|_{\He^{s,\ell}} \leq C\Bigl( \| G \chi_2 P u \|_{\He^{s-1,\ell-2}} + \| E \chi_2 u \|_{\He^{s,\ell}} + \| \chi_2 u \|_{\He^{s_0,\ell}} \Bigr)
  \end{equation}
  for all $u$ which vanish on $\supp\chi_2\cap\{t<t_-\}$. This estimate holds in the strong sense that if for some $u$ (satisfying this support condition) the right hand side is finite, then so is the left hand side, and the estimate holds.
\end{prop}

The condition on $\WFe'(E)$ means that $\pm\sfH_{G_\eop}$ is ingoing there. In particular, the radius function $r$, due to its null-bicharacteristic convexity (cf.\ the proof of Lemma~\ref{LemmaDEAway}), is monotonically increasing along backwards null-bicharacteristics starting at a point in $\WFe'(E)$ until they enter $r\geq r_0$ for some $r_0>0$ depending only on $t_-,t_+$. The cutoffs $\chi_j$ serve to localize sharply near the fiber $\phi^{-1}(t_+)$; note that they are not smooth or even conormal functions on $M$, though they \emph{are} bounded together with all their derivatives along \emph{edge}-vector fields. (In particular, multiplication by $\chi_j$ defines a bounded operator on all weighted edge Sobolev spaces.) Thus, $B\chi_1,G\chi_2,E\chi_2$ are edge ps.d.o.s with edge regular symbols. See Figure~\ref{FigSUPrIn}.

\begin{figure}[!ht]
\centering
\includegraphics{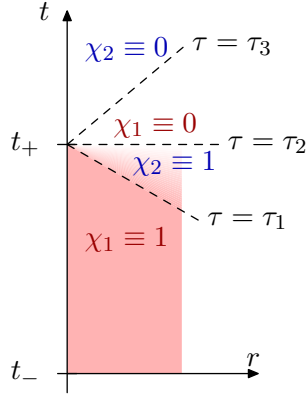}
\caption{Illustration of Proposition~\ref{PropSUPrIn}, including the sharp localization in $\frac{t-t_+}{r}$.}
\label{FigSUPrIn}
\end{figure}

Without the support assumption on $u$, one merely needs to add an additional term to the right hand side which controls $u$ in $\He^{s,\ell}$ near $\pa\cR_{\rm in}^\pm\cap\{t=t_-\}$; this follows by applying the stated estimate to a cutoff version, say $\psi(\frac{t-t_-'}{t_--t_-'})u$, of $u$, where $t_-'<t_-$.

\begin{proof}[Proof of Proposition~\usref{PropSUPrIn}]
  Except for the presence of the sharp localizers $\chi_1,\chi_2$, this is a standard radial point estimate; see \cite[Theorem~8.1]{MelroseWunschConic} and \cite[Theorem~11.1]{MelroseVasyWunschEdge} for closely related results. The sharp localization can be accommodated by working in the edge algebra with edge regular symbols. Thus, considering propagation near $\pa\cR_{\rm in}^+$ (where $\sigma>0$) and using the projective coordinates $\rho_\infty=\sigma^{-1}$, $\hat\xi=\frac{\xi}{\sigma}$, $\hat\eta=\frac{\eta}{\sigma}$ from~\eqref{EqDECoordProj}, we consider a commutant
  \begin{equation}
  \label{EqSUPrInComm}
  \begin{split}
    a &= r^{-2\ell+2} \rho_\infty^{-2 s+1}\chi_r\chi_{\hat\eta}\chi_{\hat\xi}\chi_{t_+}\chi_{t_-}, \\
    &\quad \chi_r=\chi(\digamma r),\quad
           \chi_{\hat\eta}=\chi(|\hat\eta|^2),\quad
           \chi_{\hat\xi}=\chi\bigl((\hat\xi+1)^2\bigr), \\
    &\quad
           \chi_{t_+}=\chi\Bigl(\frac{t-t_+}{r}-\tau_2\Bigr),\quad
           \chi_{t_-}=\chi(t_--t).
  \end{split}
  \end{equation}
  Here $\chi(x)=\chi_0(x/c)$ where $\chi_0\in\CI(\R;[0,1])$ equals $1$ on $(-\infty,\frac12]$, satisfies $\chi'\leq 0$ and $\sqrt\chi,\sqrt{-\chi'}\in\CI$, and equals $0$ on $[1,\infty)$; the constants $0<c<1$ and $\digamma>1$ will be specified below. We shall in particular require $c<\tau_3-\tau_2$ (so that the cutoff $\chi_{t_+}$ vanishes near $\frac{t-t_+}{r}\geq\tau_3$) and $c<\frac12$ (so that the derivatives of the cutoffs in $\hat\eta$ and $\hat\xi+1$ below have the right sign). Using~\eqref{EqDEHam}, and writing $\chi'_{\hat\eta}=\chi'(|\hat\eta|^2)$ etc., we then compute
  \begin{subequations}
  \begin{align}
    r^{-2} H_{G_\eop}a &= \rho_\infty^{-1} r^{-2}\sigma^{-1}H_{G_\eop}a \nonumber\\
    \label{EqSUPrInTMain}
      &= r^{-2\ell}\rho_\infty^{-2 s} \Bigl( (2\hat\xi+\cO(r))\bigl((-2\ell+2)-(-2 s+1)\bigr) \chi_r\chi_{\hat\eta}\chi_{\hat\xi}\chi_{t_+}\chi_{t_-} \\
    \label{EqSUPrInTr}
      &\hspace{6em} + (2\hat\xi+\cO(r))\digamma\chi'_r\chi_{\hat\eta}\chi_{\hat\xi}\chi_{t_+}\chi_{t_-} \\
    \label{EqSUPrInTeta}
      &\hspace{6em} - (4\hat\xi+\cO(r))|\hat\eta|^2\chi_r\chi'_{\hat\eta}\chi_{\hat\xi}\chi_{t_+}\chi_{t_-} \\
    \label{EqSUPrInTxi}
      &\hspace{6em} + (2(1-\hat\xi^2)+\cO(r))2(\hat\xi+1)\chi_r\chi_{\hat\eta}\chi'_{\hat\xi}\chi_{t_+}\chi_{t_-} \\
    \label{EqSUPrIntp}
      &\hspace{6em} + \Bigl(2-2\hat\xi\frac{t-t_+}{r}+\cO(r)\Bigr)\chi_r\chi_{\hat\eta}\chi_{\hat\xi}\chi'_{t_+}\chi_{t_-} \\
    \label{EqSUPrIntm}
      &\hspace{6em} - (2 r+\cO(r^2))\chi_r\chi_{\hat\eta}\chi_{\hat\xi}\chi_{t_+}\chi'_{t_-} \Bigr),
  \end{align}
  \end{subequations}
  where the $\cO(r)$ terms arise from $\sigma^{-1}\tilde H$ in~\eqref{EqDEHam}. In the positive commutator argument involving a quantization $A=A^*\in\Psie^{2 s-1,2\ell+2}(M;\cE)$ of $a$, i.e.\ the computation of the formal $L^2$-pairing
  \[
    2\Im\la P u,A u\ra = \la i(P^*A-A P)u,u\ra,
  \]
  the principal symbol of the operator on the right is $\sigmae^{2 s,2\ell}(i(P^*A-A P))$, which in view of $\sigmae^{2,2}(P)=r^{-2}G_\eop$ is equal to
  \begin{equation}
  \label{EqSUPrInSymb}
    H_{r^{-2}G_\eop}a + 2\cdot\sigmae^{1,2}\Bigl(\frac{P-P^*}{2 i}\Bigr)a = r^{-2}\Bigl(H_{G_\eop}a + 2\cdot\sigmae^1\Bigl(r^2\frac{P-P^*}{2 i\sigma}\Bigr)\rho_\infty^{-1}a + G_\eop H_{r^{-2}}a \Bigr).
  \end{equation}
  By Lemma~\ref{LemmaSUPrThr}, for any fixed $\eps>0$, we can choose a positive definite fiber inner product on $\cE$ so that at $\cR_{\rm in}^+$, the subprincipal symbol appearing in~\eqref{EqSUPrInSymb} is $<2\vartheta_{\rm in}([t_-,t_+])+\eps$, and so the sum of its contribution with that of~\eqref{EqSUPrInTMain} is $r^{-2\ell}\rho_\infty^{-2 s}b\chi_r\chi_{\hat\eta}\chi_{\hat\xi}\chi_{t_+}\chi_{t_-}$ where $b<(4\ell-4-4 s+2)+4 \vartheta_{\rm in}+2\eps$ at $\cR_{\rm in}^+$ (recalling that $\hat\xi=-1$ there); this is negative definite (as an endomorphism of $\cE$) if we fix $\eps>0$ sufficiently small, and thus $b$ equals minus the square of a smooth positive definite bundle endomorphism of $\cE$ over $\supp a$ if we moreover choose the localization parameter $c>0$ sufficiently small. The quantization of $r^{-\ell}\rho_\infty^{-s}\sqrt{-b}(\chi_r\chi_{\hat\eta}\chi_{\hat\xi}\chi_{t_-})^{1/2}$ thus provides control near $\pa\cR_{\rm in}\cap t^{-1}([t_-,t_+])$; it is the operator $B$ in~\eqref{EqSUPrInEst} (up to shifting orders), while the additional factor $\sqrt{\chi_{t_+}}$ further localizes in $\frac{t-t_+}{r}$. The term~\eqref{EqSUPrInTeta}, localizing near the flow-out of $\cR_{\rm in}^+$ over $\pa M$, is likewise the negative of a square, i.e.\ has the same sign as this main term, and can thus be dropped in the final estimate.

  The term~\eqref{EqSUPrInTr} is a nonnegative square; on its support, we have $H_{G_\eop}r<0$. This gives rise to the operator $E\chi_2$ in~\eqref{EqSUPrInEst}. The term~\eqref{EqSUPrIntm} likewise is a nonnegative square (of a symbol of order $s,\ell-\frac12$); but since we assume $u=0$ on $\supp\chi'_{t_-}$, it can be dropped (up to lower order terms captured by the final term in~\eqref{EqSUPrInEst}) in the positive commutator estimate. The term~\eqref{EqSUPrIntp} is a nonpositive square as well (i.e.\ it has the same sign as the main term) since on its support $\frac{t-t_+}{r}-\tau_2\in[c/2,c]$, so $\frac{t-t_+}{r}\geq\tau_2+\frac{c}{2}>-1$ and thus $2-2\hat\xi\frac{t-t_+}{r}$ is bounded from below; we stress again that this is an edge regular symbol (so its quantization and principal symbol are well-defined). This term can thus be dropped in the final estimate.

  We claim that the remaining terms (without the factors of $\chi_{t_+}$) are smooth multiplies of $G_\eop$, i.e.\ the action of their quantization on $u$ is controlled by $P u$. For the last term on the right in~\eqref{EqSUPrInSymb}, this is clear. Consider next the term~\eqref{EqSUPrInTxi}: on $\supp\chi_{\hat\eta}$ over $r=0$, we have $|\hat\eta|\leq\sqrt{c}$ and thus, on the characteristic set, $\hat\xi+1=-(1-|\hat\eta|^2)^{1/2}+1\leq|\hat\eta|^2\leq c$; but on $\supp\chi'_{\hat\xi}$ we have $\hat\xi+1\geq\sqrt{c/2}$, which cannot happen at the same time when $c$ is sufficiently small. For such $c$, we may choose $\digamma$ sufficiently large, thus localizing closely to $r=0$, so that $\supp\chi_{\hat\eta}\chi'_{\hat\xi}$ remains disjoint from $\Sigma^+$ on $\supp a$.

  Having analyzed~\eqref{EqSUPrInSymb}, the proof now proceeds in the standard manner by regularization and quantization of the symbolic calculation; see e.g.\ \cite[\S5.4.7]{VasyMinicourse}.
\end{proof}

\begin{rmk}[Resolution of the edge algebra]
\label{RmkSUPrSystem}
  Recalling Remark~\ref{RmkEBlowup}, one can phrase Proposition~\ref{PropSUPrIn} in a more systematic manner by working on the resolved space $[M;\phi^{-1}(t_+)]$. On this space, the cutoffs $\chi_j$ are smooth; moreover, one can quantize (weighted) edge symbols which are conormal on the pullback of $\Te^*M$ to $[M;\phi^{-1}(t_+)]$ (as discussed at the end of~\S\ref{SsEPsdo}), and one can define elliptic sets and operator wave front sets as subsets of the pullback of $\Se^*M$. The proof of Proposition~\ref{PropSUPrIn} can be rephrased as a standard positive commutator argument in this resolved pseudodifferential algebra. (For example, the operator $B\chi_1$ is an element, with smooth coefficients, of this algebra). See also Remark~\ref{RmkSULocb}.
\end{rmk}

\begin{prop}[Edge propagation near the outgoing radial set]
\label{PropSUPrOut}
  Let $t_-<t_+$ and $\tau_1<\tau_2<\tau_3<1$; in the notation~\eqref{EqSUNThrDom}, set $\vartheta_{\rm out}:=\vartheta_{\rm out}([t_-,t_+])$. Let $\psi\in\CI(\R)$ be identically $0$ on $(-\infty,0]$, positive on $(0,1)$, and equal to $1$ on $[1,\infty)$. For $\psi_j=\psi_j(r)$, $j=1,2$, define $\chi_j$ as in~\eqref{EqUPrInCutoff}. For any open neighborhood $U$ of $\pa\cR_{\rm out}^\pm\cap t^{-1}([t_-,t_+])$, one can choose $\psi_1,\psi_2$, resp.\ operators $B,G,E\in\Psie^0(M)$ with supports, resp.\ operator wave front sets contained in (the projection to $M$ of) $U$ so that $B$ is elliptic at $U\cap\pa\cR_{\rm out}^\pm$, furthermore all backwards null-bicharacteristics starting in $\WFe'(E)$ enter any fixed neighborhood of $\pa\cR_{\rm in}^\pm$, and so that for all $s,s_0,\ell\in\R$ with $s_0<s<-\frac12+\ell+\vartheta_{\rm out}$ there exists a constant $C>0$ so that
  \begin{equation}
  \label{EqSUPrOutEst}
    \| B \chi_1 u \|_{\He^{s,\ell}} \leq C\Bigl( \| G \chi_2 P u \|_{\He^{s-1,\ell-2}} + \| E \chi_2 u \|_{\He^{s,\ell}} + \| \chi_2 u \|_{\He^{s_0,\ell}} \Bigr)
  \end{equation}
  for all $u$ which vanish on $\supp\chi_2\cap\{t<t_-\}$. This estimate holds in the strong sense that if for some $u$ (satisfying this support condition) the right hand side is finite, then so is the left hand side, and the estimate holds.
\end{prop}
\begin{proof}
  We again consider the commutant~\eqref{EqSUPrInComm}, except we now set $\chi_{\hat\xi}=\chi((\hat\xi-1)^2)$ in order to localize near $\cR_{\rm out}^+$. We again have~\eqref{EqSUPrInTMain}--\eqref{EqSUPrIntm}, except $\hat\xi+1$ in~\eqref{EqSUPrInTxi} is replaced by $\hat\xi-1$. Since $\hat\xi=1$ at $\cR_{\rm out}^+$, the various terms comprising~\eqref{EqSUPrInSymb} have different relative signs. The main term, arising from~\eqref{EqSUPrInTMain} and the subprincipal symbol in~\eqref{EqSUPrInSymb}, is $r^{-2\ell}\rho_\infty^{-2 s}b\chi_r\chi_{\hat\eta}\chi_{\hat\xi}\chi_{t_+}\chi_{t_-}$, where now, for sufficiently small $\eps>0$, we have $b<(-4\ell+4+4 s-2)-4\vartheta_{\rm out}+2\eps<0$ by~\eqref{EqSUPrThrOut} if we choose an appropriate fiber inner product on $\cE$ and $a$ is sufficiently localized near $\cR_{\rm out}^+$. The term~\eqref{EqSUPrInTr} has the same sign as this main term, as does the term~\eqref{EqSUPrIntp} since $\frac{t-t_+}{r}\leq\tau_2+c<1$ (for small $c>0$).

  The a priori control term $E u$ in~\eqref{EqSUPrOutEst} arises from~\eqref{EqSUPrInTeta}. Note that, passing to $\Se^*M$, it is supported in the annular neighborhood $|\hat\eta|^2\in[\frac{c}{2},c]$ of $\pa\cR_{\rm out}^+$; if, with $c$ fixed, we localize sufficiently closely to $r=0$ by choosing $\digamma$ large enough in~\eqref{EqSUPrInComm}, then all backwards null-bicharacteristics starting from the support of~\eqref{EqSUPrInTeta} enter any fixed neighborhood of $\pa\cR_{\rm in}^+$ due to the source-sink dynamics over $\pa M$; this is the same argument as already used in~\eqref{EqDFnAin}--\eqref{EqDFnAout}.

  The term~\eqref{EqSUPrIntm} can be dropped by the support assumption on $u$, and~\eqref{EqSUPrInTxi} as well as the last term in~\eqref{EqSUPrInSymb} are controlled by $P u$.
\end{proof}

\subsection{Edge-local solvability near cone points; semi-global regularity estimate}
\label{SsSULoc}

In order to mitigate the fact that microlocal estimates for a distributional solution $u$ of $P u=f$ on some set require (mild) a priori control on $u$ on a larger set, we need to complement microlocal estimates with energy estimates which do not have this deficit; this is delicate mainly near $\pa M$. Since we do not yet have solvability and uniqueness for solutions of $P u=f$ on domains with nonempty intersection with $\pa M$ (which we will only be able to prove in~\S\ref{SsSUe}), these energy estimates must be sharply localized near a \emph{single} fiber of $\pa M$ (that is, near a single cone point in $\cM$), i.e.\ they will take place on domains which near $r=0$ are of the form $\frac{t-t_0}{r}\in[\tau_0,\tau_1]\subset(-1,1)$ (which we refer to as \emph{edge-local}); this is the reason for kepping track of such sharp localizations in Propositions~\ref{PropSUPrIn} and \ref{PropSUPrOut}.

\begin{prop}[Edge-local solvability]
\label{PropSULoc}
  Let $t_0\in I$. Fix $-1<\tau_0<\tau_1<1$ and $\tau_-\in(-1,\tau_0)$, and fix $r_0>0$ with the property that for $r\leq r_1:=\frac{r_0}{\tau_0-\tau_-}$, the functions $\tau:=\frac{t-t_0}{r}$ and $\frac{t-(t_0+r_0)}{\tau_-}$ are timelike when $\tau\in[\tau_0,\tau_1]$.\footnote{In view of the form of the metric~\eqref{EqDMetric}, this is true for all sufficiently small $r_0$.} Set
  \begin{equation}
  \label{EqSULocDom}
    \Omega_{\tau_0,\tau_1,\tau_-,r_0} := \Bigl\{ (t,r,\omega) \in I\times[0,\bar r)\times\Sph^{n-1} \colon \tau\in(\tau_0,\tau_1),\ \frac{t-(t_0+r_0)}{\tau_-}<0,\ r<r_1 \Bigr\}.
  \end{equation}
  Let $\sfs\in\CI(\Se^*M)$ be such that $\pm\sfH_{G_\eop}\sfs\leq 0$ on $\pa\Sigma^\pm$ near $\ol{\Omega_{\tau_0,\tau_1,\tau_-,r_0}}$, and let $\ell\in\R$. Denote by $\He^{\sfs,\ell}(\Omega_{\tau_0,\tau_1,\tau_-,r_0})^{\bullet,-}$ the space of restrictions to $\Omega_{\tau_0,\tau_,\tau_-,r_0}$ of elements of $\He^{\sfs,\ell}(M)$ with supported contained in $\ol{\Omega_{\tau_0,\tau_1',\tau_-,r_0'}}$ where $\tau_1'\in(\tau_1,1)$ and $r_0'>r_0$ are fixed and close to $\tau_1$ and $r_0$, respectively. Then for all $f\in\He^{\sfs-1,\ell-2}(\Omega_{\tau_0,\tau_1,\tau_-,r_0})^{\bullet,-}$, there exists a unique distributional forward solution $u$ of $P u=f$ in $\Omega_{\tau'_0,\tau_1,\tau_-,r_0}\cap\{r>0\}$ (here $\tau'_0\in(-1,\tau_0)$ is close to $\tau_0$, and $u$ vanishes in $\tau<\tau_0$), and it satisfies $u\in\He^{\sfs,\ell}(\Omega_{\tau_0,\tau_1,\tau_-,r_0})^{\bullet,-}$ with an estimate
  \begin{equation}
  \label{EqSULocEst}
    \|u\|_{\He^{\sfs,\ell}(\Omega_{\tau_0,\tau_1,\tau_-,r_0})^{\bullet,-}} \leq C\|f\|_{\He^{\sfs-1,\ell-2}(\Omega_{\tau_0,\tau_1,\tau_-,r_0})^{\bullet,-}}.
  \end{equation}
\end{prop}

See Figure~\ref{FigSULoc} for the setup. Note that the domain of intersection of $\Omega_{\tau_0,\tau_1,\tau_-,r_0}$ with $t\leq t_0+\tau_1(r-\delta)$, where $\delta>0$, has spacelike boundary hypersurfaces and its closure is disjoint from $\pa M$, so existence and uniqueness of forward solutions in this domain are standard; taking $\delta\searrow 0$ proves existence and uniqueness of the forward solution $u$.

\begin{figure}[!ht]
\centering
\includegraphics{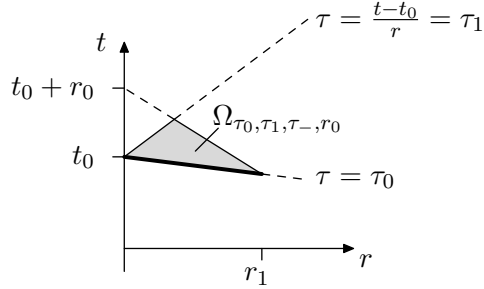}
\caption{Illustration of the domain on which Proposition~\ref{PropSULoc} takes place. The initial boundary hypersurface is drawn with a thick line; the quantity $\tau_-$ determines its slope.}
\label{FigSULoc}
\end{figure}

\begin{rmk}[Blow-up of the cone point; b-regularity]
\label{RmkSULocb}
  Upon blowing up the fiber $\phi^{-1}(t_0)$, one can use the coordinates $\tau\in\R$, $r\geq 0$, and $\omega\in\Sph^{n-1}$ as local coordinates near the interior of the front face. Note then that $r\pa_t$, $r\pa_r$ in $(t,r,\omega)$ coordinates read $\pa_\tau$, $r\pa_r-\tau\pa_\tau$ in $(\tau,r,\omega)$ coordinates, and thus edge vector fields on $M$ lift to b-vector fields on $[M;\phi^{-1}(t_0)]$. (Put differently, over the interior of the front face, $\Tb^*[M;\phi^{-1}(t_0)]$ is naturally isomorphic to the lift of $\Te^*M$ to $[M;\phi^{-1}(t_0)]$.) See also Remark~\ref{RmkEBlowup}. Therefore, elements of edge Sobolev spaces with orders $s\in\N_0$, $\ell\in\R$ with support in $\ol{\Omega_{\tau_0,\tau'_1,\tau_-,r'_0}}$ lift to elements of b-Sobolev spaces of the same orders; we leave it to the interested reader to show that this relationship remains valid also for variable order spaces. On the level of the operator $P$, we only record the principal terms of its leading order part,
  \begin{equation}
  \label{EqSULocb}
    r^2 P \equiv -D_\tau^2 + (r D_r-\tau D_\tau)^2  + \Delta_{h(t_0)} \equiv -(1-\tau^2)D_\tau^2 + (r D_r)^2 - 2 r D_r \tau D_\tau + \Delta_{h(t_0)}.
  \end{equation}
  In conclusion, Proposition~\ref{PropSULoc} is a local-in-time solvability and regularity statement for a wave operator $r^2 P$ associated with a metric for which $r=0$ lies at infinity (cf.\ the appearance of $r\pa_r$) and $\tau$-level sets (for $|\tau|$ bounded away from $1$) are spacelike. See Figure~\ref{FigSULocb}. We also remark that when $h_0$ is the standard metric on $\Sph^{n-1}$, the operator~\eqref{EqSULocb} is the same as the wave operator on Minkowski space $(\R_t\times(0,\infty)_r\times\Sph^{n-1},-\dd t^2+\dd r^2+r^2 g_{\Sph^{n-1}})$ (expressed using $\tau=\frac{t}{r}$), which was analyzed from this perspective in \cite{HintzVasyMink4}. The b-perspective breaks down however when working on nontrivial segments $|t-t_0|<\eps$ of the curve of cone points.
\end{rmk}

\begin{figure}[!ht]
\centering
\includegraphics{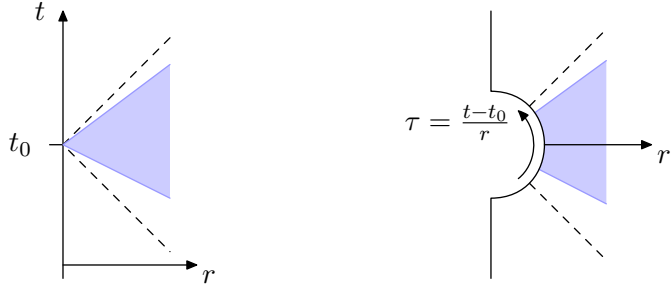}
\caption{Localization in $\frac{t-t_0}{r}$ on $M$ (on the left) and on $[M;\phi^{-1}(t_0)]$ (on the right). The dashed lines are the light cones $|t-t_0|=r$ for the model metric $-\dd t^2+\dd r^2+r^2 h(t_0)$.}
\label{FigSULocb}
\end{figure}

\begin{proof}[Proof of Proposition~\usref{PropSULoc}]
  We work with $r^2 P\in\Diffe^2(M;\cE)$, which is an edge wave operator, i.e.\ its principal symbol is the dual metric function of the Lorentzian edge metric $g_\eop=r^{-2}g$. The structure of our argument is similar to that employed in \cite[\S2.1.3]{HintzVasySemilinear} and in the proof of \cite[Theorem~6.4]{HintzVasyScrieb}, though the present setting is geometrically and microlocally simpler; hence we shall be somewhat brief. The presence of the bundle $\cE$ only requires notational overhead, and hence we consider the case that $\cE$ is trivial, i.e.\ we are working with complex-valued $u,f$.

  We begin with a basic energy estimate using the vector field multiplier $V=e^{-\digamma\tau}V_0$, $V_0:=r^{-2\ell}r\pa_t$, where $\digamma$ is a large constant. Recall the stress-energy-momentum tensor $T[u](V,W)=(V u)(W u)-\frac12 g_\eop(V,W)|\nabla^{g_\eop}u|^2$. Using that $\cL_{e^{-\digamma\tau}V_0}g_\eop=e^{-\digamma\tau}(\cL_{V_0}g_\eop-2\digamma\dd\tau\otimes_s g_\eop(V_0,\cdot))$, we have
  \begin{equation}
  \label{EqSULocDiv}
    \dv_{g_\eop}\bigl(T^{g_\eop}[u](V,\cdot)\bigr) = e^{-\digamma\tau}\Bigl(-(\Box_{g_\eop}u)V_0 u + \frac12\bigl( \la\cL_{V_0}g_\eop,T^{g_\eop}[u]\ra + 2\digamma T^{g_\eop}[u](-(\dd\tau)^\sharp,V_0)\bigr)\Bigr).
  \end{equation}
  Now, $r\pa_t$ is a future timelike edge vector field and $-\dd\tau$ is a future timelike edge 1-form on $\ol{\Omega_{\tau_0,\tau_1,\tau_-,r_0}}$. Therefore, upon choosing $\digamma$ large enough, the second term on the right controls $c\digamma r^{-2\ell}|{}^\eop\nabla u|^2$ for some fixed $c>0$, where we write $|{}^\eop\nabla u|^2=|r\pa_t u|^2+|r\pa_r u|^2+|\pa_\omega u|^2$. Integrating~\eqref{EqSULocDiv} over $\Omega_{\tau_0,\tau_1,\tau_-,r_0}$ (and noting that there the boundary terms from Stokes' theorem have good signs or vanish) thus gives the estimate $\|{}^\eop\nabla u\|_{r^\ell L^2}\leq C(\|r^2 P u\|_{r^\ell L^2}+\|u\|_{r^\ell L^2})$ on the domain $\Omega_{\tau_0,\tau_1,\tau_-,r_0}$. To control $u$ itself, we write $t=t_0+\tau r$ and thus
  \[
    u(t,r,\omega) = \int_{t_0+\tau_- r}^{t_0+\tau r} \pa_t u(s,r,\omega)\,\dd s = \int_{\tau_-}^\tau r\pa_t u(t_0+\sigma r,r,\omega)\,\dd\sigma.
  \]
  This yields $\|u\|_{r^\ell\He^1}\leq C\|r^2 P u\|_{r^\ell L^2}$, i.e.\ \eqref{EqSULocEst} for $\sfs=1$.

  We improve this to general orders $\sfs\geq 1$ using microlocal regularity results. From the b-perspective on $[M;\phi^{-1}(t_0)]$ espoused in Remark~\ref{RmkSULocb}, this relies on standard elliptic estimates and real principal type propagation estimates (note that since $\tau$ is a time function on $\Omega_{\tau_0,\tau_1,\tau_-,r_0}$, this suffices to cover all of phase space starting with propagation at $\tau=\tau_0$); see \cite[Appendix~A]{BaskinVasyWunschRadMink} for the case of variable orders. Since the parametrix or commutator proofs of the relevant b-estimates on $[M;\phi^{-1}(t_0)]$ are, without modifications, the same as edge-estimates on $M$ (cf.\ the phase space relationship in Remark~\ref{RmkSULocb}), we omit the details here. (In the edge-perspective, one needs to work with operators having edge regular symbols---specifically, with functions of $\frac{t-t_0}{r}$, as in~\S\ref{SsSUPr}, to localize in $\tau$.) Now, microlocal estimates on a domain always come with symbolically trivial error terms on a larger domain; thus, given $f\in\He^{\sfs-1,\ell-2}(\Omega_{\tau_0,\tau_1,\tau_-,r_0})^{\bullet,-}$, we pick an extension $\tilde f\in\He^{\sfs-1,\ell-2}(\Omega_{\tau_0,\tau_1',\tau_-,r_0'})^{\bullet,-}$ of $f$ to a larger domain, with norm bounded by a constant times that of $f$, and obtain
  \begin{equation}
  \label{EqSULocImpr}
    \|u\|_{\He^{\sfs,\ell}(\Omega_{\tau_0,\tau_1,\tau_-,r_0})^{\bullet,-}} \lesssim \|\tilde f\|_{\He^{\sfs-1,\ell-2}(\Omega_{\tau_0,\tau_1',\tau_-,r_0'})^{\bullet,-}} + \|\tilde u\|_{\He^{1,\ell}(\Omega_{\tau_0,\tau_1,\tau_-,r_0})^{\bullet,-}}
  \end{equation}
  for the forward solution $\tilde u$ of $P\tilde u=\tilde f$ (which restricts to $\Omega_{\tau_0,\tau_1,\tau_-,r_0}$ to $u$). (For future use, we note that the regularity order $1$ in the norm on $\tilde u$ can be replaced by any real number.) But by the already established case $\sfs=1$ of~\eqref{EqSULocEst}, we can bound the norm on $\tilde u$ by $\|\tilde f\|_{\He^{0,\ell-2}(\Omega_{\tau_0,\tau_1,\tau_-,r_0})^{\bullet,-}}$; since $\sfs\geq 1$, this gives~\eqref{EqSULocEst}.

  Turning to low regularity orders, we use duality. To wit, the arguments thus far prove an analogue of the estimate~\eqref{EqSULocEst} for the operator $P^*$ on `$-,\bullet$'-spaces with orders $\sfs^*\geq 1$ (monotonically decreasing along lifts of \emph{backwards} null-geodesics, i.e.\ $\pm\sfH_{G_\eop}\sfs^*\leq 0$ on $\pa\Sigma^\pm$), $\ell\in\R$, where $\He^{\sfs^*,\ell}(\Omega_{\tau_0,\tau_1,\tau_-,r_0})^{-,\bullet}$ is the space of restrictions of elements of $\He^{s,\ell}(M)$ with support contained in $\ol{\Omega_{\tau_0',\tau_1,\tau_-,r_0}}$ where $\tau_0'\in(-1,\tau_0)$ is close to $\tau_0$. Let now $f\in\He^{-\sfs^*,-\ell}(\Omega_{\tau_0,\tau_1,\tau_-,r_0})^{\bullet,-}$. As in the proof of \cite[Lemma~2.7]{HintzVasySemilinear}, the pairing
  \[
    \He^{\sfs^*-1,\ell-2}(\Omega_{\tau_0,\tau_1,\tau_-,r_0})^{-,\bullet}\ni P^*u^*\mapsto\la f,u^*\ra_{L^2},
  \]
  where $u^*\in\He^{\sfs^*,\ell}(\Omega_{\tau_0,\tau_1,\tau_-,r_0})^{-,\bullet}$, is well-defined and continuous. Therefore, it is given by $\la u,P^*u^*\ra_{L^2}$ for some $u\in\He^{-\sfs^*+1,-\ell+2}(\Omega_{\tau_0,\tau_1,\tau_-,r_0})^{\bullet,-}$ (with norm bounded by a constant times that of $f$), which therefore solves $P u=f$. This proves~\eqref{EqSULocEst} for $\sfs=-\sfs^*+1$, and thus for all (monotone) orders $\sfs\leq 0$.

  Finally, given a monotone order $\sfs$, fix a constant order $s_-<0$ with the property that $s_-<\min_{\Se^*_{\ol{\Omega_{\tau_0,\tau_1,\tau_-,r_0}}}M}\sfs$. Extending $f$ to $\tilde f\in\He^{\sfs-1,\ell-2}(\Omega_{\tau_0,\tau_1',\tau_-,r_0'})^{\bullet,-}$, we solve $P\tilde u=\tilde f$ in $\He^{s_-,\ell}(\Omega_{\tau_0,\tau_1',\tau_-,r'_0})^{\bullet,-}$, and use microlocal regularity results to prove~\eqref{EqSULocImpr} with the norm on $\tilde u$ replaced by the $\He^{s_-,\ell}$-norm. But since $s_-<0$, this norm is in turn bounded by the $\He^{s_--1,\ell}$-norm of $\tilde f$, and thus by the $\He^{\sfs-1,\ell}$-norm of $f$. This completes the proof.
\end{proof}

For domains with closure contained in the smooth part $\cM\setminus\cC$ of $\cM$, the analogous result is classical for constant orders (see e.g.\ \cite[Theorem~23.2.4]{HormanderAnalysisPDE3}), with the extension to variable orders accomplished most easily via the same extension, microlocal regularity, and restriction argument as in the preceding proof. We thus only state the result here:

\begin{lemma}[Local solvability away from the curve of cone points]
\label{LemmaSULocAway}
  Let $\Omega$ be a spacetime domain with $\bar\Omega\cap\pa M=\emptyset$. Let $\sfs\in\CI(\Se^*M)$ be such that $H_G\sfs\leq 0$ on $\pa\Sigma^\pm$ near $\bar\Omega$; here $G\colon T^*M^\circ\to\R$ is the dual metric function of $g$. Write $H^\sfs(\Omega)^{\bullet,-}$ for distributions of Sobolev regularity $\sfs$ with supported, resp.\ extendible character at the initial, resp.\ final boundary hypersurfaces of $\Omega$ in the terminology of \cite[Appendix~B]{HormanderAnalysisPDE3}. Then for all $f\in H^{\sfs-1}(\Omega;\cE)^{\bullet,-}$, there exists a unique distributional forward solution $u$ of $P u=f$, and it satisfies $u\in H^\sfs(\Omega;\cE)^{\bullet,-}$ and an estimate
  \[
    \|u\|_{H^\sfs(\Omega;\cE)^{\bullet,-}} \leq C\|f\|_{H^{\sfs-1}(\Omega;\cE)^{\bullet,-}}.
  \]
\end{lemma}

We can now control edge regularity semi-globally on non-refocusing spacetime domains.

\begin{definition}[Admissible orders]
\label{DefSULocAdm}
  Let $\Omega\subset M$ be a non-refocusing spacetime domain. Then $\sfs\in\CI(\Se^*M)$ and $\ell\in\R$ are \emph{$P$-admissible orders on $\Omega$} if
  \begin{enumerate}
  \item $\sfs$ is constant near $\pa\cR_{\rm in}^\pm\cap\Se^*_{\bar\Omega}M$ and near $\pa\cR_{\rm out}^\pm\cap\Se^*_{\bar\Omega}M$;
  \item recalling $\vartheta_{\rm out}(\Omega),\vartheta_{\rm in}(\Omega)$ from Definition~\ref{DefSUNThr}, we have
    \[
      \sfs>-\frac12+\ell+\vartheta_{\rm in}(\Omega)\quad\text{at}\quad \pa\cR_{\rm in}^\pm,\qquad
      \sfs<-\frac12+\ell+\vartheta_{\rm out}(\Omega)\quad\text{at}\quad \pa\cR_{\rm out}^\pm;
    \]
  \item $\pm\sfH_{G_\eop}\sfs\leq 0$ on $\pa\Sigma^\pm\cap\Se^*_{\Omega'}M$ for a spacetime domain $\Omega'\supset\Omega$.
  \end{enumerate}
\end{definition}

Given any $\ell\in\R$, the existence of an order function $\sfs$ so that $\sfs,\ell$ are $P$-admissible is guaranteed by Proposition~\ref{PropDFn}. Moreover, by continuity of $\vartheta_{\rm in}$ and $\vartheta_{\rm out}$, there exists a non-refocusing spacetime domain $\Omega'\supset\bar\Omega$ (cf.\ Corollary~\ref{CorDNrfEnlarge}) so that $P$-admissible orders on $\Omega$ are also $P$-admissible on $\Omega'$, and thus on $\Omega''$ for all $\Omega''\subseteq\Omega'$. Recalling the notation~\eqref{EqDEnlarged}, we use function spaces
\begin{equation}
\label{EqSULocFn}
  \He^{\sfs,\ell}(\Omega)^{\bullet,-} = \bigl\{ \tilde u|_\Omega \colon \tilde u\in \He^{\sfs,\ell}(\Omega_{-2\delta,2\delta}),\ \supp\tilde u\subset\Omega_{-\delta,\delta},\ \tilde u=0\ \text{on}\ \Omega_{-\delta,\delta}\setminus\Omega_{0,\delta} \bigr\}
\end{equation}
where $\delta>0$ is small and fixed (and the space does not depend on $\delta$ up to equivalence of norms); similarly for distributional sections of vector bundles. Elements of $\He^{\sfs,\ell}(\Omega)^{\bullet,-}$ are thus of supported, resp.\ extendible character at the initial, resp.\ final boundary hypersurfaces of $\Omega$.

\begin{prop}[Edge regularity on non-refocusing domains]
\label{PropSULocReg}
  Let $\Omega\subset M$ be a non-refocusing spacetime domain. Let $\sfs\in\CI(\Se^*M)$, $\ell\in\R$ be $P$-admissible orders on $\Omega$, and let $\sfs_0\in\CI(\Se^*M)$ be such that $\sfs_0>-\frac12+\ell+\vartheta_{\rm in}(\Omega)$ at $\pa\cR_{\rm in}^\pm$. Then there exists a constant $C>0$ so that for all $u\in\He^{\sfs_0,\ell}(\Omega;\cE)^{\bullet,-}$ with $f=P u\in\He^{\sfs-1,\ell-2}(\Omega;\cE)^{\bullet,-}$, we have $u\in\He^{\sfs,\ell}(\Omega;\cE)^{\bullet,-}$ and
  \begin{equation}
  \label{EqSULocRegEst}
    \|u\|_{\He^{\sfs,\ell}(\Omega;\cE)^{\bullet,-}} \leq C\Bigl(\|f\|_{\He^{\sfs-1,\ell-2}(\Omega;\cE)^{\bullet,-}}+\|u\|_{\He^{\sfs_0,\ell}(\Omega;\cE)^{\bullet,-}}\Bigr).
  \end{equation}
\end{prop}
\begin{proof}
  We drop the bundle $\cE$ from the notation. We use the notation of Definition~\ref{DefD} and write $\Omega$ in terms of $t_{\rm ini,j}$, $t_{\rm fin,j}$ as in~\eqref{EqD}. Let $t_+$ be the (constant) value of $t$ at the (unique) future boundary hypersurface $t_{\rm fin,1}^{-1}(0)\cap\bar\Omega$ which intersects $\pa M$. Let
  \[
    -1<\tau_0<\tau_0^+<\tau_1<1
  \]
  be such that $\tau:=\frac{t-t_+}{r}\in(\tau_0^+,\tau_1)$ on $t_{\rm fin,1}^{-1}(0)$ in a neighborhood $r<r_f$ of $r=0$. 

  We first work on
  \[
    \bar\Omega_{<\rm fin} := \left(\bigcap \{ t_{\rm ini,j}\geq 0 \} \cap \bigcap \{ t_{\rm fin,j}<-\delta \}\right) \cup \{ r<r_f,\ \tau<\tau_0^+ \},
  \]
  where $\delta>0$ is small (so that also $\bigcap\{ t_{\rm ini,j}>0\}\cap\bigcap\{t_{\rm fin,j}<-\delta\}$ is a spacetime domain). Microlocal elliptic regularity in the edge calculus (near $\phi^{-1}(t_+)$ with edge regular symbols) shows that $u$ lies microlocally in $\He^{\sfs+1,\ell}$ away from $\pa\Sigma^\pm$. Working in $\pa\Sigma^+$ for notational clarity, we next use real principal type propagation of edge regularity, starting in the past of the initial boundary hypersurfaces of $\Omega$ where $u$ vanishes (and thus lies in $\He^{\infty,\ell}$): this gives microlocal $\He^{\sfs,\ell}$-regularity at those $\zeta\in\pa\Sigma^+$ so that the backwards null-bicharacteristic starting at $\zeta$ enters $\bigcup\{t_{\rm ini,j}<0\}$. In view of the non-refocusing assumption on $\Omega$, the set of those $\zeta$ in particular contains the operator wave front set of the a priori control operator $E$ in Proposition~\ref{PropSUPrIn}. In the notation of Proposition~\ref{PropSUPrIn}, we thus conclude that $\chi_1 u$ lies in $\He^{\sfs,\ell}$ microlocally near $\pa\cR_{\rm in}^+\cap\bar\Omega$, from where we can propagate $\He^{\sfs,\ell}$-regularity along the flow over $\pa M$ to a punctured neighborhood of $\pa\cR_{\rm out}^+$. We can then apply Proposition~\ref{PropSUPrOut} (for slightly smaller cutoffs $\chi_1^\flat,\chi_2^\flat$) to get control of $\chi_1^\flat u$ at $\pa\cR_{\rm out}^+\cap\bar\Omega$. All $\zeta\in\pa\Sigma^+$ over $\bar\Omega_{<\rm fin}$ not yet covered thus far lie over $r>0$, and the backwards null-bicharacteristics starting at such $\zeta$ must tend to $\pa\cR_{\rm out}^+$; thus, microlocal $\He^{\sfs,\ell}$-regularity of $u$ at $\zeta$ follows again by real principal type propagation. In summary, since $\delta>0$ can be taken to be arbitrarily small, we have shown that $\chi u\in\He^{\sfs,\ell}$ for all $\chi\in\CIc(M)$ with support in $\bigcap\{t_{\rm ini,j}\geq-\delta'\}\cap\bigcap\{t_{\rm fin,j}<0\}$ where $\delta'>0$ is small enough so that $\Omega_{-\delta',0}=\bigcap\{t_{\rm ini,j}>-\delta'\}\cap\bigcap\{t_{\rm fin,j}<0\}$ is a spacetime domain; and also $\chi_1 u\in\He^{\sfs,\ell}$ for cutoffs $\chi_1$ as in Proposition~\ref{PropSUPrIn} which localize near $r=0$ and are supported in the region $\tau<\tau_0^+$.

  In order to obtain uniform regularity down to the final boundary hypersurfaces of $\Omega$, we use an extension and local solvability argument which is local near the final boundary hypersurfaces of $\Omega$ and edge-local near $\{t_{\rm fin,1}^{-1}(0)\}$. Fix a non-refocusing spacetime domain $\Omega'\supset\bar\Omega$ so that $\sfs,\ell$ are $P$-admissible on $\Omega'$, and fix a continuous extension map $E\colon\He^{\sfs-1,\ell-2}(\Omega)^{\bullet,-}\to\He^{\sfs-1,\ell-2}(\Omega')^{\bullet,-}$, e.g.\ the inverse of the restriction of the map
  \[
    \He^{\sfs-1,\ell-2}(\Omega')^{\bullet,-}\ni u\mapsto u|_\Omega\in\He^{\sfs-1,\ell-2}(\Omega)^{\bullet,-}
  \]
  to the orthogonal complement of its kernel. Given $f\in\He^{\sfs-1,\ell-2}(\Omega)^{\bullet,-}$, set $\tilde f=E f$.

  For $\tau_-\in(-1,\tau_0)$ and sufficiently small $r_0>0$ then, the domain $\Omega_{\tau_0,\tau_1,\tau_-,r_0}$ defined in~\eqref{EqSULocDom} is contained in $\Omega'$, its initial hypersurface $\tau=\tau_0$ is contained in $\bar\Omega$. Furthermore, the final hypersurface $\tau=\tau_1$ is disjoint from $\Omega$. Let $\chi=\chi(\tau)$ be a cutoff which is $0$ near $\tau_0$ and $1$ for $\tau\geq\tau_0^+$, and let $\psi=\psi(r)$ be equal to $1$ near $r=0$ and supported in $r<r_1:=\frac{r_0}{\tau_0-\tau}$. Then
  \begin{equation}
  \label{EqSULocReg}
    P(\chi u)=\chi f+[P,\chi]u
  \end{equation}
  on $\Omega\cap\Omega_{\tau_0,\tau_1,\tau_-,r_0}$. But then $\chi u=\tilde u|_{\Omega\cap\Omega_{\tau_0,\tau_1,\tau_-,r_0}}$ where $\tilde u$ is the unique forward solution of $P\tilde u=\chi\tilde f+[P,\chi]u\in\He^{\sfs-1,\ell-2}$ on $\Omega_{\tau_0,\tau_1,\tau_-,r_0}$; the membership $[P,\chi]u\in\He^{\sfs-1,\ell-2}$ follows from the regularity of $u$ established previously. An application of Proposition~\ref{PropSULoc} gives $\tilde u\in\He^{\sfs,\ell}(\Omega_{\tau_0,\tau_1,\tau_-,r_0})^{\bullet,-}$, which upon restriction implies the $\He^{\sfs,\ell}(\Omega)^{\bullet,-}$-membership of the localization of $u$ near the final fiber $\phi^{-1}(t_0)$.

\begin{figure}[!ht]
\centering
\includegraphics{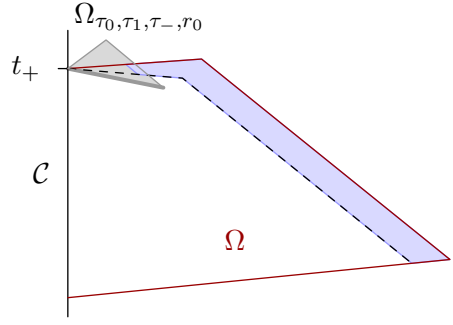}
\caption{Illustration of the local solvability and regularity argument near the final boundary hypersurfaces of $\Omega$: we solve an extended equation in $\Omega_{\tau_0,\tau_1,\tau_-,r_0}$, with control on $\tilde u$ near the thick initial boundary hypersurface provided by the first (microlocal propagation) part of the proof. We can subsequently also solve up to the final boundary hypersurfaces of $\Omega$ away from $\cC$ (drawn in blue) using Proposition~\ref{PropSULoc}.}
\label{FigSULocReg}
\end{figure}

  A similar extension argument gives uniform $H^\sfs$-membership of $u$ also near the complement of $\phi^{-1}(t_0)$ inside the union of the final boundary hypersurfaces of $\Omega$: this now relies on Lemma~\ref{LemmaSULocAway}, which we apply on the complement, inside of $\bigcap\{-\delta<t_{\rm fin,j}<\delta\}\cap\bigcap\{t_{\rm ini,j}>0\}$, of the region $t>t_++\tau_1(r-\delta)$, $r<r_1$, for small $\delta>0$, to a cut-off wave equation as in~\eqref{EqSULocReg}, where now $\chi$ equals $0$ near $t_{\rm fin,j}=-\delta$ and $1$ near $t_{\rm fin,j}\geq 0$. See Figure~\ref{FigSULocReg}.
\end{proof}

\subsection{Spectral admissibility; inversion of the edge normal operator}
\label{SsSUI}

Only using the principal symbol of $P$ (see~\eqref{EqSUOpMem}), or more specifically the structure of the null-geodesic flow on non-refocusing spacetime domains, and at the radial sets over $\pa M$ also subprincipal symbol information (from Definition~\ref{DefSUNThr}), Proposition~\ref{PropSULocReg} provides full control of solutions of $P u=f$ in a weighted edge Sobolev space, \emph{assuming} a priori membership of $u$ in an edge Sobolev space \emph{with the same weight}. In order to control $u$ also in the sense of decay at $\pa M$, we need to take the mapping properties of normal operators of $P$ at $\pa M$ into account. Recall $\hat N_{\eop,t_0}(r^2 P,\hat\sigma)$ and $\hat N_{\eop,t_0}(P,\hat\sigma)=\hat r^{-2}\hat N_{\eop,t_0}(r^2 P,\hat\sigma)$ from~\eqref{EqSUNeFamRed} and~\eqref{EqSUNRedResc}.

\begin{definition}[Spectral admissibility]
\label{DefSUIAdm}
  Let $\ell\in\R$ and $t_0\in I$. On $[0,\infty)_{\hat r}\times\Sph^{n-1}$, fix the volume density $|\hat r^{n-1}\,\dd\hat r\,\dd h(t_0)|$ and any non-degenerate fiber inner product on $\cE_{t_0}$ to define adjoints. We say that $P$ is \emph{spectrally admissible with weight $\ell$ at $t_0$} if $\ell-\frac{n}{2}$ is a non-indicial weight at $t_0$ (see Definition~\usref{DefSUNInd}), and if the following conditions hold for all $\hat\sigma\in\C$, $\Im\hat\sigma\geq 0$, $|\hat\sigma|=1$, and all smooth sections $u=u(\hat r,\omega)$ of $\cE_{t_0}$ on $(0,\infty)\times\Sph^{n-1}$.
  \begin{enumerate}
  \item\label{ItSUIAdmDir}{\rm (Injectivity.)} Assume that $\hat N_{\eop,t_0}(P,\hat\sigma)u=0$, $|(\hat r\pa_{\hat r})^j\pa_\omega^\alpha u|\lesssim\hat r^{\ell-\frac{n}{2}}$ for $\hat r\in(0,1]$, and $|(\hat r\pa_{\hat r})^j\pa_\omega^\alpha(e^{-i\hat\sigma\hat r}u)|\lesssim\hat r^C$ (for $\hat\sigma=\pm 1$), resp.\ $|\pa_{\hat r}^j\pa_\omega^\alpha u|\lesssim\hat r^{-N}$ (for $\Im\hat\sigma>0$) for $\hat r\in[1,\infty)$ and for all $j\in\N_0$, $\alpha\in\N_0^{n-1}$, $N\in\R$, and for some constant $C$. Then $u=0$.
  \item\label{ItSUIAdmAdj}{\rm (Injectivity of the adjoint.)} Assume that $\hat N_{\eop,t_0}(P,\hat\sigma)^*u=0$, $|(\hat r\pa_{\hat r})^j\pa_\omega^\alpha u|\lesssim\hat r^{-\frac{n-4}{2}-\ell}$ for $\hat r\in(0,1]$, and $|(\hat r\pa_{\hat r})^j\pa_\omega^\alpha(e^{i\hat\sigma\hat r}u)|\lesssim\hat r^C$ (for $\hat\sigma=\pm 1$), resp.\ $|\pa_{\hat r}^j\pa_\omega^\alpha u|\lesssim\hat r^{-N}$ (for $\Im\hat\sigma>0$) for $\hat r\in[1,\infty)$ and for all $j,\alpha,N$, for some constant $C$. Then $u=0$.
  \end{enumerate}
  We say that $P$ is \emph{spectrally admissible with weight $\ell$ on the spacetime domain $\Omega\subset M$} if it is spectrally admissible at all $t_0\in I$ with $\phi^{-1}(t_0)\subset\bar\Omega$.
\end{definition}

Recall $\hat X=[0,\infty]_{\hat r}\times\Sph_\omega^{n-1}$; we continue using the volume density $|\hat r^{n-1}\,\dd\hat r\,\dd h(t_0)|$ on $\hat X$. Writing scattering covectors near $\hat r=\infty$ as $\xi_\scop\,\dd\hat r+\eta_\scop\cdot\hat r\,\dd\omega$, set $G_\scop:=\xi_\scop^2+|\eta_\scop|_{h^{-1}(t_0)}^2$ and
\begin{equation}
\label{EqSUIRadsc}
\begin{split}
  {}^\scop\Sigma &:= \{(\hat r,\omega;\xi_\scop,\eta_\scop)\colon \hat r=\infty,\ G_\scop=1\}, \\
  {}^\scop\cR_{\rm in,\pm 1} &:= \{ (\hat r,\omega;\xi_\scop,\eta_\scop) \colon \hat r=\infty,\ \xi_\scop=\mp 1,\ \eta_\scop=0 \}, \\
  {}^\scop\cR_{\rm out,\pm 1} &:= \{ (\hat r,\omega;\xi_\scop,\eta_\scop) \colon \hat r=\infty,\ \xi_\scop=\pm 1,\ \eta_\scop=0 \}.
\end{split}
\end{equation}
Let moreover $\sfH:=\hat r H_{G_\scop}=(\pa_{\xi_\scop}G_\scop)(\hat r\pa_{\hat r}-\eta_\scop\pa_{\eta_\scop})+(\pa_{\eta_\scop}G_\scop)\pa_\omega-((\hat r\pa_{\hat r}-\eta_\scop\pa_{\eta_\scop})G_\scop)\pa_{\xi_\scop}-(\pa_\omega G_\scop)\pa_{\eta_\scop}$; thus, $\pm\sfH|_{\hat r^{-1}(\infty)}$ flows in ${}^\scop\Sigma$ from the source ${}^\scop\cR_{\rm in,\pm 1}$ to the sink ${}^\scop\cR_{\rm out,\pm 1}$. Recall the quantities $\vartheta_{\rm in}(t_0),\vartheta_{\rm out}(t_0)$ from Definition~\ref{DefSUNThr}. Note that ${}^\scop\Sigma$ is the preimage (over $\pa_\infty\hat X$) of $\Sigma\cap\Te^*_{\phi^{-1}(t_0)}M$ under the map $f_{\pm 1}$ in~\eqref{EqEInvPhaseSpace}; similarly for ${}^\scop\cR_{\rm in/out,\pm 1}$ and $\cR_{\rm in/out}\cap\Te^*_{\phi^{-1}(t_0)}M$.

\begin{lemma}[Spectral admissibility as an invertibility statement]
\label{LemmaSUIInv}
  Let $\ell\in\R$, and suppose that $P$ is spectrally admissible at $t_0$ with weight $\ell$. Let $\sfs\in\CI(\Ssc^*\hat X)$ and $\sfr\in\CI(\ol{\Tsc^*_{\pa_\infty\hat X}}\hat X)$; suppose that $\sfr$ is a constant $>-\frac12+\vartheta_{\rm in}(t_0)$ near ${}^\scop\cR_{\rm in,\pm 1}$ and a constant $<-\frac12+\vartheta_{\rm out}(t_0)$ near ${}^\scop\cR_{\rm out,\pm 1}$, as well as $\pm\sfH\sfr\leq 0$ on ${}^\scop\Sigma$. Then there exists a bounded right inverse
  \[
    \hat N_{\eop,t_0}(P,\hat\sigma)^{-1} \colon H_{\bop,\scop}^{\sfs-2,\ell-2,\sfr+1}(\hat X;\cE_{t_0}) \to H_{\bop,\scop}^{\sfs,\ell,\sfr}(\hat X;\cE_{t_0}),
  \]
  and moreover there exists a constant $C$ so that
  \begin{equation}
  \label{EqSUIInvEst}
    \| u \|_{H_{\bop,\scop}^{\sfs,\ell,\sfr}(\hat X;\cE_{t_0})} \leq C\| \hat N_{\eop,t_0}(P,\hat\sigma)u \|_{H_{\bop,\scop}^{\sfs-2,\ell-2,\sfr+1}(\hat X;\cE_{t_0})}
  \end{equation}
  for all $u$ for which both sides are finite, and for all $\hat\sigma=e^{i\theta}$ where $\theta\in[0,\frac{\pi}{4}]$ (for the `$+$' sign), resp.\ $\theta\in[\frac{3\pi}{4},\pi]$ (for the `$-$' sign). For $\hat\sigma=e^{i\theta}$, $\theta\in[\frac{\pi}{4},\frac{3\pi}{4}]$, this remains true for arbitrary orders $\sfs,\sfr$, and $\sfr+1$ on the right in~\eqref{EqSUIInvEst} can be replaced by $\sfr$.
\end{lemma}

For $\hat\sigma=\pm 1$, this result is analogous to \cite[Lemma~4.8]{HintzConicProp}. It can also be used to show that Definition~\ref{DefSUIAdm} is equivalent, up to shifts in orders due to shifts in weights arising from Sobolev embedding, to \cite[Definition~4.6]{HintzConicProp}.

\begin{proof}[Proof of Lemma~\usref{LemmaSUIInv}]
  On a conceptual level, the phase space relationship~\eqref{EqEInvPhaseSpace} means that the estimate~\eqref{EqSUIInvEst} follows in the cases $\hat\sigma=\pm 1$ near $\pa_\infty\hat X$ from the edge propagation result, Propositions~\ref{PropSUPrIn} and \ref{PropSUPrOut} (without localization in time), near the radial sets over $\pa M$; since the space $H_{\eop,I}^{\sfs,\ell}$ is related (via Lemma~\ref{LemmaEInvFT}) to $H_{\bop,\scop}^{f_{\hat\sigma}^*\sfs,\ell,\sfr}$ with $\sfr=f_{\hat\sigma}^*\sfs-\ell$, the threshold condition $\sfs>-\frac12+\ell+\vartheta_{\rm in}(t_0)$ in the edge setting is equivalent to $\sfr>-\frac12+\vartheta_{\rm in}(t_0)$ in the present scattering setting, similarly at the outgoing radial set. Allowing for $\Im\hat\sigma$ to be nonnegative corresponds to the usual limiting absorption principle estimate, and when $\Im\hat\sigma>0$ the operator $\hat N_{\eop,t_0}(P,\hat\sigma)$ is elliptic. The additional condition on $\ell$ ensures the applicability of b-normal operator estimates near $\pa_0\hat X$, i.e.\ estimates involving the inverse of the indicial family $N_{\bop,t_0}(r^2 P,\xi)$, $\Re\xi=\ell-\frac{n}{2}$, defined in~\eqref{EqSUNbMellin}; note here that $\hat r^\ell L^2(\hat X)=\hat r^{\ell-\frac{n}{2}}L^2(\hat X;|\frac{\dd\hat r}{\hat r}\,\dd h(t_0)|)$, with the latter space being mapped isometrically to $L^2(\{\Re\xi=\lambda-\frac{n}{2}\};L^2(\Sph^{n-1};|\dd h(t_0)|))$ by the Mellin transform.

  In some more detail, the ellipticity at fiber infinity of $\hat N_{\eop,t_0}(P,\hat\sigma)$ implies
  \[
    \|u\|_{H_{\bop,\scop}^{\sfs,\ell,\sfr}} \lesssim \|\hat N_{\eop,t_0}(P,\hat\sigma)u \|_{H_{\bop,\scop}^{\sfs-2,\ell-2,\sfr}} + \|u\|_{H_{\bop,\scop}^{s_0,\ell,\sfr}}
  \]
  for any fixed $s_0<\inf\sfs$. Next, let $\chi\in\CIc([0,2)_{\hat r})$ be identically $1$ on $[0,1]$, then the second term, with $u=\chi u+(1-\chi)u$, can be estimated using the triangle inequality by
  \[
    \|\chi u\|_{\Hb^{s_0,\ell-\frac{n}{2}}(|\frac{\dd\hat r}{\hat r}\,\dd h(t_0)|)} + \|u\|_{H_{\bop,\scop}^{s_0,-N,\sfr}}
  \]
  for any fixed $N<\ell$. Since $\ell-\frac{n}{2}$ is a non-indicial weight at $t_0$, we can estimate
  \[
    \|\chi u\|_{\Hb^{s_0,\ell}} \lesssim \| N_{\bop,t_0}(r^2 P)(\chi u) \|_{\Hb^{s_0-2,\ell}} = \| \hat r^{-2} N_{\bop,t_0}(r^2 P)(\chi u) \|_{\Hb^{s_0-2,\ell-2}}.
  \]
  Replacing $\hat r^{-2}N_{\bop,t_0}(r^2 P)$ by $\hat N_{\eop,t_0}(P,\hat\sigma)$ creates an error bounded by $\|\chi u\|_{\Hb^{s_0-2,\ell-1}}$, and similarly $\|[\hat N_{\eop,t_0}(P,\hat\sigma),\chi]u\|_{\Hb^{s_0-2,\ell-2}}\lesssim\|\tilde\chi u\|_{\Hb^{s_0-1,-N}}$ where $\tilde\chi\in\CIc([0,\infty))$ equals $1$ on $\supp\chi$. Thus, we have
  \[
    \|u\|_{H_{\bop,\scop}^{\sfs,\ell,\sfr}} \lesssim \|\hat N_{\eop,t_0}(P,\hat\sigma)u \|_{H_{\bop,\scop}^{\sfs-2,\ell-2,\sfr}} + \|u\|_{H_{\bop,\scop}^{s_0,\ell_0,\sfr}}
  \]
  where $\ell_0<\ell$. Finally, we weaken the scattering decay order $\sfr$ on the final term using elliptic estimates near $\pa_\infty\hat X$ away from ${}^\scop\Sigma$, and near ${}^\scop\Sigma$ using Melrose's radial point estimates \cite{MelroseEuclideanSpectralTheory} (with $\Im\hat\sigma$ acting as complex absorption when it is nonzero) and real principal type propagation. A detailed account is given in the proof of \cite[Proposition~5.28]{VasyMinicourse}. This gives
  \[
    \|u\|_{H_{\bop,\scop}^{\sfs,\ell,\sfr}} \lesssim \|\hat N_{\eop,t_0}(P,\hat\sigma)u \|_{H_{\bop,\scop}^{\sfs-2,\ell-2,\sfr}} + \|u\|_{H_{\bop,\scop}^{s_0,\ell_0,\sfr_0}}
  \]
  with $s_0<\sfs$, $\ell_0<\ell$, $\sfr_0<\sfr$; so the inclusion of the space on the left into the space on the right is compact. An analogous estimate for $\hat N_{\eop,t_0}(P,\hat\sigma)^*$ on the dual spaces implies the Fredholm property of
  \[
    \hat N_{\eop,t_0}(P,\hat\sigma) \colon \bigl\{ u\in H_{\bop,\scop}^{\sfs,\ell,\sfr} \colon \hat N_{\eop,t_0}(P,\hat\sigma)u \in H_{\bop,\scop}^{\sfs-2,\ell-2,\sfr+1} \bigr\} \to H_{\bop,\scop}^{\sfs-2,\ell-2,\sfr+1}.
  \]

  We proceed to prove the invertibility of this map. If $u\in H_{\bop,\scop}^{\sfs,\ell,\sfr}$ lies in the kernel of $\hat N_{\eop,t_0}(P,\hat\sigma)$, then $u\in H_{\bop,\scop}^{\infty,\ell,\sfr}$ by elliptic regularity; this implies smoothness for bounded $\hat r\in(0,\infty)$, and for $\hat r<1$ via Sobolev embedding the pointwise upper bound $\hat r^{\ell-\frac{n}{2}}$ for $u$ and all its b-derivatives. Consider now the case $\hat\sigma=\pm 1$. Near $\hat r=\infty$, we proceed as in \cite[Proof of Proposition~4.4]{HintzNonstat} and note that $u$ satisfies iterative regularity under any application of any number of operators $\hat r A$ where $A\in\Diffsc^1(\hat X)$ has principal symbol vanishing at ${}^\scop\cR_{\rm out,\hat\sigma}$; see e.g.\ \cite[\S2]{GellRedmanHassellShapiroZhangHelmholtz} for a detailed discussion extending \cite{HassellMelroseVasySymbolicOrderZero}. Taking as these operators for example $\hat r(\pa_{\hat r}-i\hat\sigma)$ and spherical vector fields $\pa_\omega$, one concludes (see \cite[\S3.3]{GellRedmanHassellShapiroZhangHelmholtz} and also \cite[Proposition~12]{MelroseEuclideanSpectralTheory}) that $u=e^{i\hat\sigma\hat r}u_0(\hat r,\omega)$ where $u_0$ is conormal at $\hat r=\infty$. According to Definition~\ref{DefSUIAdm}, the spectral admissibility of $P$ implies that $u$ must vanish. For the cokernel, we note that the $L^2$-dual of $H_{\bop,\scop}^{\sfs-2,\ell-2,\sfr+1}$ is $H_{\bop,\scop}^{\sfs',\ell',\sfr'}$ with $\sfs'=-\sfs+2$, $\ell'=-\ell+2$, $\sfr'=-\sfr-1$; and thus $\ell'-\frac{n}{2}=-\frac{n-4}{2}-\ell$. The triviality of $\ker\hat N_{\eop,t_0}(P,\hat\sigma)^*$ on this space then follows by similar arguments.

  In the case $\Im\hat\sigma>0$, elements of the kernel or cokernel of $\hat N_{\eop,t_0}(P,\hat\sigma)$ are automatically rapidly decaying as $\hat r\to\infty$, and hence invertibility follows directly from Definition~\ref{DefSUIAdm}.
\end{proof}

Using the (inverse) Fourier transform, we turn this into estimates for $N_{\eop,t_0}(r^2 P)$ which will take place on the invariant edge Sobolev spaces $H_{\eop,\rm I}^{\sfs,\ell}(M';\cE_{t_0})$ recalled in~\S\ref{SsEInv}; here $M'={}^+N\phi^{-1}(t_0)$, and we recall that ${}^+N\phi^{-1}(t_0)\cong\R_{t'}\times[0,\infty)_{r'}\times\Sph^{n-1}$. More generally, for $t'_0\in\R$, we consider
\begin{align*}
  M'_{\geq t'_0} &:= [t'_0,\infty) \times [0,\infty)_{r'}\times\Sph^{n-1}, \\
  H_{\eop,\rm I}^{\sfs,\ell}(M'_{\geq t'_0};\cE_{t_0})^\bullet &:= \bigl\{ u\in H_{\eop,\rm I}^{\sfs,\ell}(M';\cE_{t_0}) \colon \supp u\subset\{t'\geq t'_0\} \bigr\}.
\end{align*}

\begin{prop}[Forward solvability and uniqueness for edge normal operators]
\label{PropSUIFwd}
  Suppose that $P$ is spectrally admissible with weight $\ell\in\R$ at $t_0$. Let moreover $\sfs\in\CI(\Se^*M')$ be an invariant order function which is $P$-admissible at $t_0$, by which we mean that the restriction of $\sfs$ to $t'=r'=0$, regarded as an element of $\CI(\Se^*_{\phi^{-1}(t_0)}M)$, satisfies the conditions of Definition~\usref{DefSULocAdm} inside $\Se^*_{\phi^{-1}(t_0)}M$. Then
  \begin{align*}
    N_{\eop,t_0}(P) = r'{}^{-2}N_{\eop,t_0}(r^2 P) &\colon \bigl\{ u\in H_{\eop,\rm I}^{\sfs,\ell}(M';\cE_{t_0}) \colon N_{\eop,t_0}(P)u\in H_{\eop,\rm I}^{\sfs-1,\ell-2}(M';\cE_{t_0}) \bigr\} \\
      &\quad\hspace{13em} \to H_{\eop,\rm I}^{\sfs-1,\ell-2}(M';\cE_{t_0})
  \end{align*}
  is invertible. Moreover, this operator restricts to an invertible map
  \begin{equation}
  \label{EqSUIFwd}
    N_{\eop,t_0}(P) \colon \bigl\{ u\in H_{\eop,\rm I}^{\sfs,\ell}(M'_{\geq t_0'};\cE_{t_0})^\bullet \colon N_{\eop,t_0}(P)u\in H_{\eop,\rm I}^{\sfs-1,\ell-2}(M';\cE_{t_0})^\bullet \bigr\} \to H_{\eop,\rm I}^{\sfs-1,\ell-2}(M'_{\geq t_0'};\cE_{t_0})^\bullet.
  \end{equation}
\end{prop}
\begin{proof}
  We drop the bundle $\cE_{t_0}$ from the notation. We write $\cF u(\sigma;r',\omega')\equiv\hat u(\sigma;r',\omega'):=\int e^{i t'\sigma}u(t',r',\omega')\,\dd t'$ for the Fourier transform of $u$ in $t'$ (with the usual sign convention for Fourier transforms in time). For the first part, we use Lemma~\ref{LemmaEInvFT} with $\alpha=\ell$ and $w=n$ in conjunction with Lemma~\ref{LemmaSUIInv} to deduce, using $\hat M_{|\sigma|}(\hat r,\omega)=(\frac{\hat r}{|\sigma|},\omega)$, that
  \begin{align*}
    \|u\|_{H_{\eop,\rm I}^{\sfs,\ell}(M')} &\sim \int_{\R\setminus\{0\}} \|\hat M_{|\sigma|}^*(\hat u(\sigma;\cdot))\|_{H_{\bop,\scop}^{f_{\sgn\sigma}^*\sfs,\ell,f_{\sgn\sigma}^*\sfs-\ell}(\hat X)}^2\,|\sigma|^{2\ell-n}\,\dd\sigma \\
      &\lesssim \int_{\R\setminus\{0\}} \Bigl\|\hat N_{\eop,t_0}\Bigl(P,\frac{\sigma}{|\sigma|}\Bigr)(\hat M_{|\sigma|}^*\hat u(\sigma;\cdot))\Bigr\|_{H_{\bop,\scop}^{f_{\sgn\sigma}^*\sfs-2,\ell-2,f_{\sgn\sigma}^*\sfs-\ell+1}(\hat X)}^2\,|\sigma|^{2\ell-n}\,\dd\sigma.
  \end{align*}
  Recalling~\eqref{EqSUNRedScaleOp}, this is equal to
  \begin{align*}
    &\int_{\R\setminus\{0\}} \bigl\|\hat M_{|\sigma|}^*\bigl(N_{\eop,t_0}(r^2 P,\sigma)\hat u(\sigma;\cdot)\bigr) \bigr\|_{H_{\bop,\scop}^{f_{\sgn\sigma}^*\sfs-2,\ell,f_{\sgn\sigma}^*\sfs-\ell-1}(\hat X)}^2\,|\sigma|^{2\ell-n}\,\dd\sigma \\
    &\quad \lesssim \int_{\R\setminus\{0\}} \bigl\| \hat M_{|\sigma|}^*\bigl(\cF(N_{\eop,t_0}(r^2 P)u)(\sigma;\cdot)\bigr) \bigr\|_{H_{\bop,\scop}^{f_{\sgn\sigma}^*\sfs-1,\ell,f_{\sgn\sigma}^*\sfs-1-\ell}(\hat X)}^2\,|\sigma|^{2\ell-n}\,\dd\sigma \\
    &\quad \sim \|N_{\eop,t_0}(r^2 P)u\|_{H_{\eop,\rm I}^{\sfs-1,\ell}(M')} \\
    &\quad = \|N_{\eop,t_0}(P)u\|_{H_{\eop,\rm I}^{\sfs-1,\ell-2}(M')}.
  \end{align*}
  Solving $N_{\eop,t_0}(P)u=f$ via $\hat u(\sigma;\cdot)=N_{\eop,t_0}(P,\sigma)^{-1}\hat f(\sigma;\cdot)$ thus proves the first claim.

  We now turn to the proof of the forward solution property~\eqref{EqSUIFwd} by means of a Paley--Wiener argument. It suffices to consider the case $t'_0=0$ and source terms $f\in\CIc((0,\infty)_{t'}\times(0,\infty)_{r'}\times\Sph^{n-1})$. For all $N$, we then have $|\hat f(\sigma;r',\omega')|\lesssim\la|\sigma|\ra^{-N}$ uniformly in $\Im\sigma\geq 0$, and the same estimates are valid for all derivatives of $\hat f$. Concerning $\hat M_{|\sigma|}^*\hat f(\sigma;\cdot)$ then, i.e.\ passing to $\hat r=r'|\sigma|$, this implies
  \begin{equation}
  \label{EqSUIFwdf}
    |\hat r^j(\hat r\pa_{\hat r})^k\pa_{\omega'}^\alpha ( \hat M_{|\sigma|}^*\hat f(\sigma;\cdot) ) | \lesssim \la|\sigma|\ra^{-N},\qquad \Im\sigma\geq 0,
  \end{equation}
  for all $j\in\R$, $k\in\N_0$, and $\alpha$. Let $\sfs_{\hat\sigma}$ be the pullback of $\sfs$ under $f_{\sgn\Re\hat\sigma}$ when $\Im\hat\sigma<\frac12$, and an arbitrary constant otherwise. The bounds on $\hat N_{\eop,t_0}(P,\hat\sigma)^{-1}$ (with $\hat\sigma=\frac{\sigma}{|\sigma|}$ ranging over the closed upper half circle, which is compact) from Lemma~\ref{LemmaSUIInv} imply, upon noting that $\hat r^{-1}|\sigma|\sim 1$ on $\supp\hat f(\sigma;\cdot)$, that
  \begin{align}
    \|\hat M_{|\sigma|}^*(r'{}^{-\ell}N_{\eop,t_0}(P,\sigma)^{-1}\hat f(\sigma;\cdot))\|_{H_{\bop,\scop}^{\sfs_{\hat\sigma},0,\sfs_{\hat\sigma}}} &= \|\hat M_{|\sigma|}^*(N_{\eop,t_0}(P,\sigma)^{-1}\hat f(\sigma;\cdot))\|_{H_{\bop,\scop}^{\sfs_{\hat\sigma},\ell,\sfs_{\hat\sigma}-\ell}}\,|\sigma|^\ell \nonumber\\
      &\lesssim \| \hat M_{|\sigma|}^*\hat f(\sigma;\cdot)\|_{H_{\bop,\scop}^{N',\ell+2,N'}}\,|\sigma|^{\ell-2} \nonumber\\
      &= \|\hat M_{|\sigma|}^*\hat f(\sigma;\cdot)\|_{H_{\bop,\scop}^{N',0,N'+\ell+2}} \nonumber\\
  \label{EqSUIFwdMf}
      &\lesssim \la|\sigma|\ra^{-N}
  \end{align}
  for suitable $N'$ (which we may take to be arbitrarily large, exploiting~\eqref{EqSUIFwdf}). Consider now
  \begin{equation}
  \label{EqSUIFwdInvFT}
    u(t';\cdot) := \frac{1}{2\pi}\int_\R e^{-i t'\sigma} N_{\eop,t_0}(P,\sigma)^{-1}\hat f(\sigma;\cdot)\,\dd\sigma
  \end{equation}
  restricted to a bounded annular region $A_{R_0,R_1}=\{0<R_0\leq r'\leq R_1<\infty\}$. Fix $s_0\in\R$ to be smaller than $\sfs_{\hat\sigma}$ for all $\hat\sigma$; for $|\sigma|\lesssim 1$ then, we have $\hat r=r'|\sigma|\lesssim 1$ and thus
  \begin{equation}
  \label{EqSUIFwdUnif}
    \|N_{\eop,t_0}(P,\sigma)^{-1}\hat f(\sigma;\cdot)\|_{H^{s_0}(A_{R_0,R_1})} \lesssim \| \hat M_{|\sigma|}^*(r'{}^{-\ell}N_{\eop,t_0}(P,\sigma)^{-1}\hat f(\sigma;\cdot)) \|_{H_{\bop,\scop}^{\sfs_{\hat\sigma},0,\sfs_{\hat\sigma}}(\hat X)}
  \end{equation}
  (where the scattering decay order on the right is irrelevant since only the norm upon localization to $\hat r\lesssim 1$ enters here). For $|\sigma|\gtrsim 1$ on the other hand, so $\hat r\gtrsim 1$ and $\pa_{r'}=|\sigma|\pa_{\hat r}$, we have the same estimate except with an additional power $|\sigma|^{s_1}$ on the right where we can take $s_1=\lceil|s_0|\rceil$.

  For $\sigma\neq 0$, $\Im\sigma>0$, we have $N_{\eop,t_0}(P,\sigma)^{-1}\colon H_{\bop,\scop}^{\sfs-2,\ell-2,\sfr}\to H_{\bop,\scop}^{\sfs,\ell,\sfr}$ for arbitrary $\sfs,\sfr$. Since $\pa_\sigma N_{\eop,t_0}(P,\sigma)\in(\frac{r'}{r'+1})^{-1}\Diff_{\bop,\scop}^1(\hat X)$ (by~\eqref{EqSUNeFam} and \eqref{EqSUNRedResc}) maps $H_{\bop,\scop}^{\sfs,\ell,\sfr}\to H_{\bop,\scop}^{\sfs-1,\ell-1,\sfr}$ and thus back into the domain of $N_{\eop,t_0}(P,\sigma)^{-1}$, we conclude that $N_{\eop,t_0}(P,\sigma)^{-1}$ is continuous in the operator topology, and in fact holomorphic. For $\sigma_0\in\R\setminus\{0\}$, we note that $N_{\eop,t_0}(P,\sigma)^{-1}\colon H_{\bop,\scop}^{\sfs_{\hat\sigma_0}-2,\ell-2,\sfs_{\hat\sigma_0}-\ell+1}\to H_{\bop,\scop}^{\sfs_{\hat\sigma_0}-\eps,\ell,\sfs_{\hat\sigma_0}-\ell-\eps}$ is continuous for $\sigma\in\C$, $\Im\sigma\geq 0$, near $\sigma_0$ in the operator norm topology for any fixed $\eps>0$, as follows from the functional analytic arguments used in \cite[\S2.7]{VasyMicroKerrdS}. In view of~\eqref{EqSUIFwdMf} and the uniform bounds~\eqref{EqSUIFwdUnif} near $\sigma=0$, we may thus shift the integration contour in~\eqref{EqSUIFwdInvFT} and infer
  \[
    u(t';\cdot) = \frac{1}{2\pi}\int_{\Im\sigma=C} e^{-i t'\sigma}N_{\eop,t_0}(P,\sigma)^{-1}\hat f(\sigma;\cdot)\,\dd\sigma,\qquad C\geq 0;
  \]
  indeed, our above arguments imply the equality of both sides on any annular region $A_{R_0,R_1}$. Combining~\eqref{EqSUIFwdUnif} with~\eqref{EqSUIFwdMf}, the $H^{s_0}(A_{R_0,R_1})$-norm of the right hand side is bounded from above by $C' e^{C t'}$ for some $C'$ which is independent of $C$. For $t'<0$, this implies $u(t';\cdot)=0$ upon letting $C\to\infty$. The proof is complete.
\end{proof}

\begin{rmk}[Sharpness of the conditions on $\sfs,\ell$]
\label{RmkSUISharp}
  Suppose $\sfs$ violates the threshold condition at the outgoing radial set (i.e.\ $\sfs$ is too large); then also $\sfr=f_{\pm 1}^*\sfs-\ell$ is too large at ${}^\scop\cR_{\rm out,\pm 1}$. In explicit examples, one can then see that $\hat N_{\eop,t_0}(P,\pm 1)^{-1}$ does \emph{not} map $H_{\bop,\scop}^{f_{\pm 1}^*\sfs-2,\ell-2,f_{\pm 1}^*\sfs-\ell+1}(\hat X)$ into $H_{\bop,\scop}^{f_{\pm 1}^*\sfs,\ell,f_{\pm 1}^*\sfs-\ell}(\hat X)$ since the scattering decay rate of the target space is too strong. (For example, when $P=-D_t^2+D_r^2$ in $1+1$ dimensions, then the inverse of $\hat N_{\eop,t_0}(r^2 P,\hat\sigma)=-\pa_{\hat r}^2-\hat\sigma^2$ produces $e^{i\hat\sigma\hat r}$ behavior for large $\hat r$, which does not lie in a Sobolev space on $(1,\infty)_{\hat r}$ with weight $>-\frac12+\vartheta_{\rm out}=-\frac12$ at infinity.) Therefore, $N_{\eop,t_0}(P)^{-1}$ does not map $H_{\eop,\rm I}^{\sfs-1,\ell-2}$ into $H_{\eop,\rm I}^{\sfs,\ell}$, but merely into $\la r'D_{t'}\ra^\delta H_{\eop,\rm I}^{\sfs,\ell}$ where $\delta>0$ is any number larger than the difference of $\sfr$ and the outgoing radial point threshold. (On the Fourier transform side, this amounts to allowing extra growth $\la\hat r\ra^\delta$ at $\hat r=\infty$.) On the other hand, the nullspace of $N_{\eop,t_0}(P)$ on $H_{\eop,\rm I}^{\sfs,\ell}(M'_{\geq t'_0};\cE_{t_0})^\bullet$ is trivial for \emph{all} $\sfs$; this follows from the fact that the Fourier transform $\hat u(\sigma)$ of elements $u$ in this nullspace is holomorphic in $\Im\sigma>0$, and we have $N_{\eop,t_0}(r^2 P,\sigma)\hat u(\sigma)=0$; but since $N_{\eop,t_0}(r^2 P,\sigma)$ is elliptic at $r'=\infty$ for $\Im\sigma>0$, this implies that $\hat u(\sigma)$ must be Schwartz and thus $0$ due to the spectral admissibility assumption. In summary, when $P$ is spectrally admissible with weight $\ell$, the map~\eqref{EqSUIFwd} is injective for all $\sfs$, but its range in general does not even contain $\CIc((M'_{\geq t'_0})^\circ)$ unless $\sfs,\ell$ are $P$-admissible.
\end{rmk}

\begin{rmk}[Sharpness of the condition on $\ell$]
\label{RmkSUISharpl}
  For the scalar wave operator $\Box_g$ on Minkowski space $g=-\dd t^2+\dd r^2+r^2 g_{\Sph^{n-1}}$ in dimensions $n\neq 2$, Lemma~\ref{LemmaExCSpec} below computes the interval of weights $\ell$ for which $P$ is spectrally admissible as $(\ell_-,\ell_+)=(1-|\frac{n-2}{2}|,1+|\frac{n-2}{2}|)$. If $\ell<\ell_-$, then~\eqref{EqSUIFwd} is not injective; indeed, a convolution of the forward fundamental solution with a $\CIc(\R_{t'})$ function lies in its kernel. (Explicitly, for $n=3$, say, we have $\Box_g(r'{}^{-1}v(t'-r'))=0$ in $r'>0$ for any function $v\in\CIc(\R)$, and we then note that $r'{}^{-1}\in r'{}^\ell L^2((0,1);r'{}^2\,\dd r')$ for all $\ell<\ell_-=\frac12$. As an edge regularity order $\sfs$ for which $\sfs,\ell$ are $P$-admissible, one can take a suitable function $\sfs<-\frac12+\ell_-=0$, so indeed $r'{}^{-1}v(t'-r')\in H_{\eop,\rm I}^{\ell,\ell}\subset H_{\eop,\rm I}^{\sfs,\ell}$ in this case.) Dually, \eqref{EqSUIFwd} fails to be surjective for $\ell>\ell_+$.
\end{rmk}

\subsection{Solvability and uniqueness on non-refocusing domains}
\label{SsSUe}

By comparing $P$ with its edge normal operator at $\phi^{-1}(t_0)$, we shall now solve $P u=f$ on small domains; we use the coordinates $t,r,\omega$ from~\eqref{EqDCollar}.

\begin{prop}[Solvability and uniqueness on small domains]
\label{PropSUeSmall}
  Let $t_0\in\R$, and suppose that $\sfs\in\CI(\Se^*M)$ and $\ell\in\R$ are $P$-admissible orders on a spacetime domain containing $\phi^{-1}(t_0)$, and $P$ is spectrally admissible with weight $\ell$ at $t_0$. Define the map
  \[
    S_\lambda \colon M'=\R\times[0,\infty)\times\Sph^{n-1} \ni (t',r',\omega') \mapsto (t,r,\omega) = (t_0+\lambda t',\lambda r',\omega') \in M.
  \]
  Fix $t'_-<t'_+$, $r'_+\geq 0$, and $\kappa>1$, and set
  \begin{equation}
  \label{EqSUeSmallOmega}
    \Omega' = \Omega'_{t'_-,t'_+,r'_+,\kappa} := \{ (t',r',\omega') \in M' \colon t'_-<t'<t'_+,\ r'<r'_+ + \kappa(t'_+-t') \}.
  \end{equation}
  Then for small $\lambda>0$, the image $\Omega_\lambda:=S_\lambda(\Omega')\subset M$ is a non-refocusing spacetime domain, and for all $f\in\He^{\sfs-1,\ell-2}(\Omega_\lambda;\cE)^{\bullet,-}$, there exists a unique $u\in\He^{\sfs,\ell}(\Omega_\lambda;\cE)^{\bullet,-}$ solving $P u=f$ in $\Omega_\lambda$.
\end{prop}

See Figure~\ref{FigSUeSmall} for the setup.

\begin{figure}[!ht]
\centering
\includegraphics{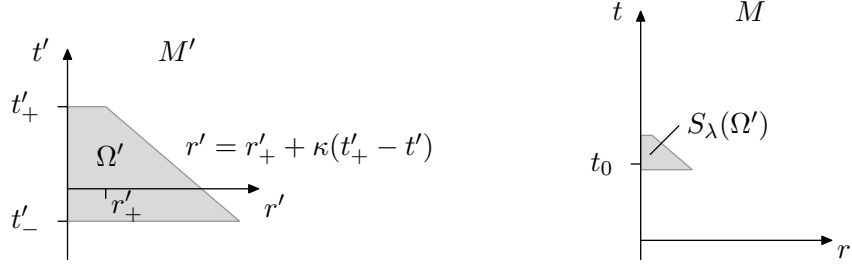}
\caption{\textit{On the left:} the domain $\Omega'=\Omega'_{t'_-,t'_+,r'_+,\kappa}$ inside of $M'$. \textit{On the right:} its image in $M$ under $S_\lambda$ with $0<\lambda\ll 1$.}
\label{FigSUeSmall}
\end{figure}

\begin{proof}[Proof of Proposition~\usref{PropSUeSmall}]
  We drop the bundle from the notation. Since $\Omega'$ is a non-refocusing spacetime domain in $M'$ for the model metric $r'{}^2 g_{\eop,t_0}$, this is true also with respect to the metric $S_\lambda^*g$ for sufficiently small $\lambda$ since
  \[
    \lambda^{-2}S_\lambda^*g-r'{}^2 g_{\eop,t_0}\in\lambda r'{}^{-2}\CI\bigl([0,\lambda_0);\CI(M';S^2\,\Te^*M')\bigr),
  \]
  cf.\ the proof of Corollary~\ref{CorDNrfSmall}. Therefore, $\Omega_\lambda$ is a non-refocusing domain inside of $(M,g)$ for small $\lambda$. For all such $\lambda$, we identify a neighborhood $U'$ of $\ol{\Omega'}$ in $M'$ with a subset of $M$ via $S_\lambda$, and correspondingly pull back $\sfs$ to a smooth function $\sfs_{(\lambda)}\in\CI([0,\lambda_0)_\lambda;\CI(\Se^*_{U'}M'))$ on $\Se^*_{U'}M'$; its limit $\sfs_{(0)}$ at $\lambda=0$ is the invariant extension of $\sfs|_{\Se^*_{\phi^{-1}(t_0)}M}$.

  \pfstep{Inversion of $N_{\eop,t_0}(P)$ on $(\bullet,-)$-spaces.} We apply Proposition~\ref{PropSUIFwd}, specifically~\eqref{EqSUIFwd}, with $\sfs_{(0)}$, $t'_-$ in place of $\sfs$, $t'_0$. Set $M'_{[t'_-,t'_+)}=M'_{\geq t'_-}\setminus M'_{\geq t'_+}$. Given $f\in H_{\eop,\rm I}^{\sfs_{(0)}-1,\ell-2}(M'_{[t'_-,t'_+)})^{\bullet,-}$, i.e.\ $f=\tilde f|_{M'_{[t'_-,t'_+)}}$ where $\tilde f\in H_{\eop,\rm I}^{\sfs_{(0)}-1,\ell-2}(M'_{\geq t'_-})$, we can then set $u=\tilde u|_{M'_{[t'_-,t'_+)}}\in H_{\eop,\rm I}^{\sfs_{(0)},\ell}(M'_{[t'_-,t'_+)})^{\bullet,-}$ where $\tilde u=N_{\eop,t_0}(P)^{-1}\tilde f$; in view of~\eqref{EqSUIFwd}, now with $t'_+$ in place of $t'_0$, this solution $u$ of $N_{\eop,t_0}(P)u=f$ does not depend on the choice of extension $\tilde f$.

  Next, given $f\in\He^{\sfs_{(0)}-1,\ell-2}(\Omega')^{\bullet,-}$ pick an extension $\tilde f\in H_{\eop,\rm I}^{\sfs_{(0)}-1,\ell-2}(M'_{[t'_-,t'_+)})^{\bullet,-}$ of $f$ with controlled norm. Setting $\tilde u:=N_{\eop,t_0}(P)^{-1}\tilde f\in H_{\eop,\rm I}^{\sfs_{(0)},\ell}(M'_{[t'_-,t'_+)})^{\bullet,-}$, we need to show that $u:=\tilde u|_{\Omega'}$ does not depend on the choice of extension $\tilde f$. By linearity, it suffices to consider the case that $N_{\eop,t_0}(P)\tilde u=\tilde f$ vanishes on $\Omega'$. Then so does $\tilde u$ in the region $t'-t'_-<r'<r'_++\kappa(t'_+-t')$ by finite speed of propagation (in the smooth part $r'>0$ of the spacetime) due to its vanishing in $t'<t'_-$. (This is region I in Figure~\ref{FigSUeSmallIt}.) Define by $\bar t'=(r'_++\kappa t'_++t'_-)/(1+\kappa)$ the unique time at which $t'-t'_-=r'_++\kappa(t'_+-t')$, and set $t'_0:=\min(\bar t',t'_+)$. Set $\tilde u_0:=\tilde u|_{M'_{[t'_-,t'_0)}}$, which solves $N_{\eop,t_0}(P)\tilde u_0=0$ for $r'<r'_++\kappa(t'_+-t')$; but since for all $t'\in[t'_-,t'_0)$, $\tilde u_0$ vanishes in the nonempty annular region $t'-t'_-<r'<r'_++\kappa(t'_+-t')$, we may define a new function $u_0$ to be equal to $\tilde u_0$ for $r'<r'_++\kappa(t'_+-t')$, and $u_0:=0$ for $r'\geq r'_++\kappa(t'_+-t')$, which then satisfies
  \begin{equation}
  \label{EqSUeSmallSlice}
    u_0 \in H_{\eop,\rm I}^{\sfs_{(0)},\ell}(M'_{[t'_-,t'_0)})^{\bullet,-},\qquad
    N_{\eop,t_0}(P)u_0 = 0.
  \end{equation}
  We can extend $u_0$ to a solution $\tilde u_0\in H_{\eop,\rm I}^{\sfs_{(0)},\ell}(M'_{\geq t'_-})^\bullet$ of $N_{\eop,t_0}(P)\tilde u_0=0$ by writing $\tilde u_0=\chi u_0-N_{\eop,t_0}(P)^{-1}([N_{\eop,t_0}(P),\chi]u_0)$ where $\chi\in\CI(\R)$ equals $1$ on $t\leq t'_0-\delta$ (where $\delta\in(0,t'_0-t'_-)$) and $0$ near $t\geq t'_0$. Proposition~\ref{PropSUIFwd} implies $\tilde u_0=0$ and thus $u_0=0$ on $t\leq t'_0-\delta$; since $\delta>0$ is arbitrary, this gives $u_0=0$ on $M'_{[t'_-,t'_0)}$. Therefore, $\tilde u=0$ for $t'\in[t'_-,t'_0)$, $r'<r'_++\kappa(t'_+-t')$. (This proves the vanishing of $\tilde u$ in region II in Figure~\ref{FigSUeSmallIt}.)

  Repeating this argument with $t'_0$ in place of $t'_-$ allows us to deduce the vanishing of $\tilde u$ on $\Omega'$ in increasingly large intervals of $t'$; and we obtain $\tilde u=0$ on $\Omega'$ in finitely many steps when $r'_+>0$, and $\tilde u=0$ on $\Omega'\cap\{t'<t'_+-\eps\}$ for any fixed $\eps>0$ in finitely many steps when $r'_+=0$ (and thus again $\tilde u=0$ on $\Omega'$). We have thus proved the invertibility of the operator
  \[
    N_{\eop,t_0}(P) \colon \bigl\{ u \in \He^{\sfs_{(0)},\ell}(\Omega')^{\bullet,-} \colon N_{\eop,t_0}(P)u\in\He^{\sfs_{(0)}-1,\ell-2}(\Omega')^{\bullet,-} \bigr\} \to \He^{\sfs_{(0)}-1,\ell-2}(\Omega')^{\bullet,-}.
  \]

  \begin{figure}[!ht]
  \centering
  \includegraphics{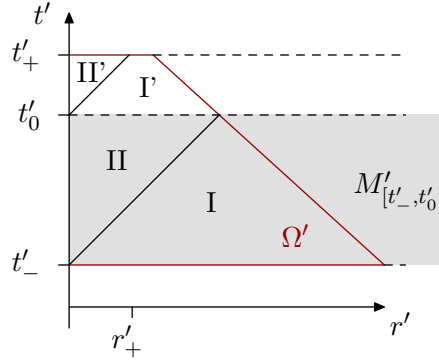}
  \caption{Illustration of the injectivity argument for $N_{\eop,t_0}(P)$ given around~\eqref{EqSUeSmallSlice}: if $N_{\eop,t_0}(P)\tilde u$ vanishes in $\Omega'$ and $\tilde u$ vanishes in $t'_-<0$, then $\tilde u$ vanishes in region I by finite speed of propagation, and then in II due to the uniqueness part of Proposition~\ref{PropSUIFwd} for the extension by $0$ of $\tilde u|_{M'_{[t'_-,t'_0)}}$ to the complement of regions I and II. ($M'_{[t_-',t_0')}$ is shaded in gray.) One then repeats this argument in regions I' and II'.}
  \label{FigSUeSmallIt}
  \end{figure}

  \pfstep{A priori estimate for $P$, restricted orders.} We next turn to the operator $P$ on $S_\lambda(\Omega')$, or equivalently to the operator $P_\lambda:=\lambda^2 S_\lambda^*P$ on $\Omega'$; note that
  \begin{equation}
  \label{EqSUePlambda}
    P_\lambda\in\CI([0,\lambda_0)_\lambda;r'{}^{-2}\Diffe^2(U')),\qquad
    P_0=N_{\eop,t_0}(P).
  \end{equation}
  \textit{We first assume that $\sfs,\ell$ and $\sfs-1,\ell$ are $P$-admissible}, so $\sfs$ exceeds the threshold value at the incoming radial set by more than $1$. Proposition~\ref{PropSULocReg} applies to $P_\lambda$ with uniform constants for all small $\lambda$; we take $\sfs$ there to be $\sfs_{(\lambda)}$, and $\sfs_0$ there to be $N_{\eop,t_0}(P)$-admissible and so that $\sfs_0<\sfs_{(0)}-1$ (and thus $\sfs_0<\sfs_{(\lambda)}-1$ on $U'$ for small $\lambda$). We may then estimate
  \begin{align*}
    &\|u\|_{\He^{\sfs_{(\lambda)},\ell}(\Omega')^{\bullet,-}} \\
    &\qquad \leq C\Bigl( \| P_\lambda u \|_{\He^{\sfs_{(\lambda)}-1,\ell-2}(\Omega')^{\bullet,-}} + \|u\|_{\He^{\sfs_0,\ell}(\Omega')^{\bullet,-}} \Bigr) \\
    &\qquad\leq C\|P_\lambda u\|_{\He^{\sfs_{(\lambda)}-1,\ell-2}(\Omega')^{\bullet,-}} + C C' \| N_{\eop,t_0}(P)u \|_{\He^{\sfs_0-1,\ell-2}(\Omega')^{\bullet,-}} \\
    &\qquad\leq (C+C C')\|P_\lambda u\|_{\He^{\sfs_{(\lambda)}-1,\ell-2}(\Omega')^{\bullet,-}} + C C'\|(P_\lambda-N_{\eop,t_0}(P))u\|_{\He^{\sfs_0-1,\ell-2}(\Omega')^{\bullet,-}}.
  \end{align*}
  The final term is bounded by $C C'\cdot C''\lambda\|u\|_{\He^{\sfs_0+1,\ell}(\Omega')^{\bullet,-}}\leq C'''\lambda\|u\|_{\He^{\sfs_{(\lambda)},\ell}(\Omega')^{\bullet,-}}$ for small $\lambda$ in view of~\eqref{EqSUePlambda}; if $\lambda<(2 C''')^{-1}$, this term can be absorbed into the left hand side, yielding
  \[
    \|u\|_{\He^{\sfs_{(\lambda)},\ell}(\Omega')^{\bullet,-}} \leq C\| P_\lambda u \|_{\He^{\sfs_{(\lambda)}-1,\ell-2}(\Omega')^{\bullet,-}}.
  \]
  Returning to the notation used in the statement of the Proposition, this implies the uniqueness of solutions of $P u=f$ on $\Omega_\lambda$ (and by duality the existence of solutions for the adjoint, i.e.\ backwards, problem on dual function spaces).

  \pfstep{Solvability for $P$, restricted orders.} Similar arguments, \textit{now only requiring that $\sfs,\ell$ and $\sfs+1,\ell$ be $P$-admissible} (so not only $-\sfs+1,-\ell+2$ are $P^*$-admissible for the backwards problem, but also $-(\sfs+1)+1=-\sfs,-\ell+2$, or equivalently $\sfs$ is more than $1$ below the threshold value at the outgoing radial set), imply the adjoint estimate
  \[
    \|u^*\|_{\He^{-\sfs_{(\lambda)}+1,-\ell+2}(\Omega')^{-,\bullet}} \leq C\| P_\lambda^* u^* \|_{\He^{-\sfs_{(\lambda)},-\ell}(\Omega')^{-,\bullet}}
  \]
  and thus the \emph{solvability} of $P u=f\in\He^{\sfs_{(\lambda)}-1,\ell-2}(\Omega_\lambda)^{\bullet,-}$ on $\Omega_\lambda$ with $u\in\He^{\sfs_{(\lambda)},\ell}(\Omega_\lambda)^{\bullet,-}$, as claimed.

  \pfstep{Solvability and uniqueness, general orders.} We now remove the additional assumptions, i.e.\ we only assume the $P$-admissibility of $\sfs,\ell$. We work on $\Omega_\lambda$. For uniqueness, we simply note that $P u=0$ for $u\in\He^{\sfs,\ell}(\Omega_\lambda)^{\bullet,-}$ implies $u\in\He^{\sfs',\ell}(\Omega_\lambda)^{\bullet,-}$ for all $P$-admissible order functions $\sfs'$; in particular, we may take $\sfs'$ to exceed the threshold value at the incoming radial set by $1$, thus $\sfs'-1,\ell$ are admissible, and therefore $u=0$ by what we have already shown. For existence, let $f\in\He^{\sfs-1,\ell-2}(\Omega_\lambda)^{\bullet,-}$. Fix an order function $\tilde\sfs<\sfs$ near $\ol{\Omega_\lambda}$ so that $\tilde\sfs,\ell$ and $\tilde\sfs+1,\ell$ are $P$-admissible. (Such a function $\tilde\sfs$ can easily be found when $P$ is replaced by its edge normal operator and $\tilde\sfs$ is invariant in the coordinates $(t-t_0,r,\omega)$, and then the same $\tilde\sfs$ works near $\ol{\Omega_\lambda}$ for small $\lambda$.) Then by what we have already shown, we can solve $P u=f$ on $\Omega_\lambda$ with $u\in\He^{\tilde\sfs,\ell}(\Omega)^{\bullet,-}$; but propagation of edge regularity (Proposition~\ref{PropSULocReg}) implies $u\in\He^{\sfs,\ell}(\Omega)^{\bullet,-}$, as desired. (Similar arguments are used also in \cite[Proof of Proposition~4.1]{VasyLowEnergy}.)
\end{proof}

We next piece these small pieces (near $\cC$) together with the standard solvability theory away from $\cC$ to obtain the main result of this paper:

\begin{thm}[Solvability and uniqueness on non-refocusing domains]
\label{ThmSUeNonrf}
  Let $\Omega\subset M$ be a non-refocusing spacetime domain (Definition~\usref{DefDNrf}). Suppose that $\sfs\in\CI(\Se^*M)$ and $\ell\in\R$ are $P$-admissible orders (Definition~\usref{DefSULocAdm}), and $P$ is spectrally admissible with weight $\ell$ on $\Omega$ (Definition~\usref{DefSUIAdm}). Then there exists a constant $C$ so that the following holds: for all\footnote{See~\eqref{EqSULocFn} for the definition of the function space.} $f\in\He^{\sfs-1,\ell-2}(\Omega;\cE)^{\bullet,-}$ there exists a solution $u\in\He^{\sfs,\ell}(\Omega;\cE)^{\bullet,-}$ of $P u=f$, and it satisfies the estimate
  \[
    \|u\|_{\He^{\sfs,\ell}(\Omega;\cE)^{\bullet,-}} \leq C\|f\|_{\He^{\sfs-1,\ell-2}(\Omega;\cE)^{\bullet,-}}.
  \]
  Moreover, $u$ is the unique solution of $P u=f$ in the space $\He^{-\infty,\ell}(\Omega;\cE)^{\bullet,-}$.
\end{thm}

The sharpness of the conditions on $\sfs,\ell$ on the scale of function spaces used here is discussed already for invariant operators in Remarks~\ref{RmkSUISharp} and \ref{RmkSUISharpl}.

\begin{proof}[Proof of Theorem~\usref{ThmSUeNonrf}]
  We begin with the proof of uniqueness in $\He^{\sfs,\ell}(\Omega)^{\bullet,-}$. Let $t_-=\min_{\bar\Omega}t$ and $t_+=\max_{\bar\Omega}t$. For $\kappa>1$ to be chosen momentarily, let $\lambda_0>0$ be so small that the local solvability and uniqueness statement of Proposition~\ref{PropSUeSmall} holds for all $t_0\in[t_-,t_+]$ and $0<\lambda\leq\lambda_0$, where we fix $\Omega'=\Omega'_{-1,1,1,\kappa}$ as in~\eqref{EqSUeSmallOmega}; let us write $\Omega_{t_0,\lambda}=S_{t_0,\lambda}(\Omega')$ where $S_{t_0,\lambda}(t',r',\omega')=(t_0+\lambda t',\lambda r',\omega)$. We choose $\kappa$ such that for all $\lambda<\lambda_0$, the initial hypersurface of $\Omega_{t_-,\lambda}$ is disjoint from the initial boundary hypersurfaces of $\Omega$. This holds when $\kappa-1$ is sufficiently small, since, in view of the timelike nature of $t_{\rm ini,1}$ in the notation of Definition~\ref{DefD} and the form~\eqref{EqDMetric} of the metric, we have $t\geq t_- - \frac{r}{\kappa}$ on $t_{\rm ini,1}^{-1}(0)$ near $\phi^{-1}(t_-)$.

  Consider now $u\in\He^{\sfs,\ell}(\Omega;\cE)^{\bullet,-}$ with $P u=0$. The region $\Omega\cap\{t<t_--\lambda_0\}$ is a spacetime domain whose closure is disjoint from the curve $\cC$ of cone points, and thus local uniqueness for wave equations (cf.\ Lemma~\ref{LemmaSULocAway}) implies the vanishing of $u$ there. Consider next $u_0:=u|_{\Omega_{t_-,\lambda_0}}$, defined to be $0$ on the complement of $\Omega$ inside of $\Omega_{t_-,\lambda_0}$; this thus solves $P u_0=0$ in $\Omega_{t_-,\lambda_0}$, and therefore $u_0=0$. We conclude that $u$ vanishes in $\Omega_{t_-,\lambda_0}\cup\{t\leq t_--\lambda_0\}$. From this we can infer the vanishing of $u$ in $\{t\leq t_-+\lambda_0\}$ using local uniqueness away from $\cC$. We can now repeat this argument with the domain $\Omega_{t_-^1,\lambda_0}$ where $t_-^1=t_-+\lambda_0$. After finitely many iterations, we conclude that $u=0$ in $\Omega\cap\{t<t_\flat\}$ where $t_+-\lambda_0<t_\flat<t_+$.

  Next, we consider the domain $\Omega'':=\Omega'_{-1,0,0,\kappa}$, where $\kappa>1$ is chosen so close to $1$ that $S_{t_+,\lambda_0}(\Omega'')\subset\Omega$. We may assume (by taking $\lambda_0$ smaller from the outset if necessary) that Proposition~\ref{PropSUeSmall} applies to $S_{t_+,\lambda_0}(\Omega'')$. The vanishing of $u$ near $\{t_+-\lambda_0\leq t\leq t_+\}$ then follows by an edge-local energy estimate (Proposition~\ref{PropSULoc}). Finally, we can conclude $u=0$ on $\Omega$ using standard local uniqueness away from $\cC$. See Figure~\ref{FigSUeNonrfUniq}.

  For the stronger uniqueness statement, we use an interpolation idea from \cite[\S8]{MelroseWunschConic}. Suppose that $u\in\He^{s_0,\ell}(\Omega;\cE)^{\bullet,-}$ (with arbitrary $s_0$), $P u=0$. Choosing $\ell_0\leq\ell$ so that $s_0>-\frac12+\ell_0+\vartheta_{\rm in}$, Proposition~\ref{PropSULocReg} applies to $u\in\He^{s_0,\ell_0}(\Omega;\cE)^{\bullet,-}$ and gives $u\in\He^{\sfs',\ell_0}(\Omega;\cE)^{\bullet,-}$ for admissible orders $\sfs'$ which, in particular, we can choose to be arbitrarily large near $\pa\cR_{\rm in}^\pm$. With $B\in\Psie^0(M)$ having Schwartz kernel with support disjoint from the final boundary hypersurfaces of $\Omega$ in both factors, and with elliptic and wave front set contained in a neighborhood of $\pa\cR_{\rm in}^\pm$, the statements $B u\in\He^{\infty,\ell_0}\cap\He^{s_0,\ell}$ imply, by interpolation, $B u\in\He^{\infty,\ell-\eps}$ for all $\eps>0$. This provides the a priori above-threshold membership for $u$ near $\pa\cR_{\rm in}^\pm$, and we conclude that $u\in\He^{\sfs,\ell-\eps}(\Omega';\cE)^{\bullet,-}$ for all $\Omega'\subset\Omega$ with closure disjoint from the final boundary hypersurfaces of $\Omega$. But since the assumptions on $\ell$ are open, we can appeal to the injectivity of $P$ on $\He^{\sfs,\ell-\eps}(\Omega';\cE)^{\bullet,-}$ to conclude that $u=0$ on $\Omega'$. Letting $\Omega'\to\Omega$ completes the argument.

  \begin{figure}[!ht]
  \centering
  \includegraphics{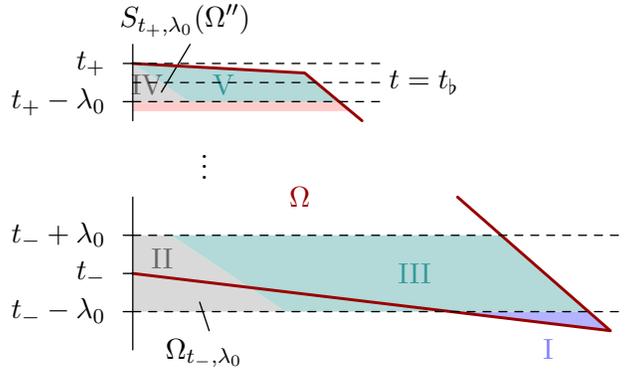}
  \caption{Illustration of the proof of uniqueness. \textit{On the bottom:} uniqueness in I by finite speed of propagation and in II using Proposition~\ref{PropSUeSmall}, and then in III by finite speed of propagation using to the now known vanishing on the bottom and left boundary hypersurfaces of region III. Subsequently, steps II and III are repeated around $t_-+\lambda_0$, $t_-+2\lambda$, etc. \textit{On the top:} Once uniqueness is known in the region $t<t_+-\lambda_0$, we obtain it in region IV from Proposition~\ref{PropSUeSmall} and then in region V by finite speed of propagation.}
  \label{FigSUeNonrfUniq}
  \end{figure}

  We can \emph{solve} $P u=f$ in a similar manner by constructing $u$ step by step in the same regions as those used in the uniqueness proof above. Thus, we first construct $u$ in $\Omega\cap\{t<t_--\lambda_0\}$ using Lemma~\ref{LemmaSULocAway} and in $\Omega_{t_-,\lambda_0}$ using Proposition~\ref{PropSUeSmall}; and to continue the solution further, we use a cutoff argument: if $\chi\in\CI(M)$, then $P(\chi u)=\chi f+[P,\chi]u$, so if we choose $\chi$ to be $1$ in the region into which we wish to extend $u$ next, and to transition to $0$ in the region where we have already constructed $u$, then by solving $P u'=\chi f+[P,\chi]u$ (with the right hand side already determined), we construct $u=u'$ on the set where $\chi=1$. This completes the proof.
\end{proof}

\begin{rmk}[Results on general spacetime domains]
\label{RmkSUeGeneral}
  Under suitable assumptions on $P$, one can construct solutions of $P u=f$ on a general spacetime domain $\Omega$ by partitioning $\Omega$ into a finite number of non-refocusing domains. As a simple example (and dropping the bundle $\cE$ from the notation), if $\vartheta_{\rm in}=\vartheta_{\rm out}$ are equal and constant along $\cC$, and if $P$ is spectrally admissible with weight $\ell$ at all points of $\cC$, then we can construct a unique solution
  \begin{subequations}
  \begin{equation}
  \label{EqSUeGeneral1}
    u \in \bigcap_{\eps>0} \He^{s-\eps,\ell-\eps}(\Omega)^{\bullet,-}
  \end{equation}
  of the equation
  \begin{equation}
  \label{EqSUeGeneral2}
    P u = f \in \He^{s-1,\ell-2}(\Omega)^{\bullet,-}
  \end{equation}
  \end{subequations}
  for suitable orders $s$ in the following manner: by writing $\Omega$ as a union of a finite number of non-refocusing domains $\Omega_1,\ldots,\Omega_N$, we can solve $P u=f$ on the first domain $\Omega_1$, with an admissible order function $\sfs_1$ which is $\delta$-close to the constant value $-\frac12+\ell+\vartheta$ (where $\vartheta=\vartheta_{\rm in}=\vartheta_{\rm out}$). To continue this solution to the next domain $\Omega_2$, note that $\min(\sfs_1-\delta)>-\frac12+(\ell-2\delta)+\vartheta$; we may thus take an order function $\sfs_2\leq\min(\sfs_1-\delta)$ so that $\sfs_2,\ell-2\delta$ are $P$-admissible on $\Omega$; and so on. Since $\delta>0$ can be taken to be arbitrarily small, we obtain~\eqref{EqSUeGeneral1}--\eqref{EqSUeGeneral2} for any $s>-\frac12+\ell+\vartheta$.
\end{rmk}

\subsection{Higher b-regularity; polyhomogeneity}
\label{SsSUb}

There is a fixed upper bound on the edge regularity order $\sfs$ for which Theorem~\ref{ThmSUeNonrf} applies, even for source terms $f\in\CIc(M^\circ;\cE)$, since $\ell$ is restricted to a bounded interval (via the requirement of spectral admissibility) and $\sfs$ is bounded above in terms of $\ell$ on the flow-out of the outgoing radial set via Definition~\ref{DefSULocAdm}. We surpass this limitation by proving additional b-regularity of solutions, provided the source terms have the same additional regularity.

Working momentarily on a compact manifold $M$ with fibered boundary $\pa M$, we define mixed edge-b-Sobolev spaces by
\[
  H_{\eop;\bop}^{(\sfs;k),\ell}(M) := \{ u\in \He^{\sfs,\ell}(M) \colon A u\in\He^{\sfs,\ell}(M)\ \forall\,A\in\Diffb^k(M) \};
\]
here $\sfs\in\CI(\Se^*M)$, $\ell\in\R$. These spaces can be given the structure of Hilbert spaces, with squared norm given by the sum of the squared $\He^{\sfs,\ell}(M)$-norms of $A u$ where $A$ lies in a fixed finite subset of $\Diffb^k(M)$ which spans $\Diffb^k(M)$ over $\CI(M)$. The mapping properties of edge-ps.d.o.s on such spaces can be analyzed using a mixed algebra of differential-pseudodifferential operators. (Such algebras appeared implicitly in \cite{MelroseEuclideanSpectralTheory} and explicitly in \cite{VasyPropagationCorners,VasyWaveOnAdS,MelroseVasyWunschDiffraction,HintzVasyScrieb}.) For $\sfs\in\CI(\Se^*M)$, $k\in\N_0$, we consider the space $\Diffb^k\Psie^\sfs(M)$ of all operators which are finite sums $\sum_i Q_i P_i$ where $Q_i\in\Diffb^k(M)$ and $P_i\in\Psie^\sfs(M)$ are operators \emph{with smooth coefficients}. (It suffices to require b-regular coefficients; but one \emph{cannot} allow $P_i$ to merely have edge-regular coefficients.) The key lemma states that one can write any such operator as $\sum_i P_i'Q_i'$ for some $Q_i'\in\Diffb^k(M)$, $P'_i\in\Psie^\sfs(M)$; see \cite[Lemma~5.8]{HintzVasyScrieb} for the (more general) constant order edge-b setting, and \cite[Lemma~2.9]{Hintz3b} for how to treat variable orders (the proof there carrying over \emph{mutatis mutandis} to the present setting). This relies on the iterated regularity of Schwartz kernels of edge-ps.d.o.s with smooth (or b-regular) symbols under application of b-vector fields on the edge double space which are tangent to the edge-diagonal.\footnote{A concrete instance demonstrating the necessity of b-regular coefficients is the following. If $V\in\Vb(M)$ and $W=a W_0$ where $W_0\in\Ve(M)$ and $a\in L^\infty(M)$ is edge regular, i.e.\ bounded together with all derivatives along edge vector fields, then $W V=V W-[V,W]=V W-a[V,W_0]-V(a)W_0$; the second term is an edge vector field with edge regular coefficients, but typically one only has $V(a)\in r^{-1}L^\infty$. One may for example take $V=\pa_t=W_0$ and $a=\chi((t-t_0)/r)$.} This then implies the algebra property $\Diffb^k\Psie^\sfs(M)\circ\Diffb^{k'}\Psie^{\sfs'}(M)\subset\Diffb^{k+k'}\Psie^{\sfs+\sfs'}(M)$; we also deduce that for $P\in\Diffb^k\Psie^\sfs(M)$ and $u\in H_{\eop;\bop}^{(\sfs';k'),\ell}(M)$ with $k'\geq k$ one has $P u\in H_{\eop;\bop}^{(\sfs'-\sfs;k'-k),\ell}(M)$.

One can microlocalize edge-regularity in such mixed edge-b-Sobolev spaces. A more general situation is discussed in \cite[\S5.2]{HintzVasyScrieb}, and thus we shall be brief. If $B,G\in\Psie^0(M)$ are such that $\WFe'(B)\subset\Elle(G)\cap\Elle(P)$ where $P\in r^{-\alpha}\Psie^m(M)$, then for all $\sfs,k,\ell,N$ we have the estimate
\[
  \|B u\|_{H_{\eop;\bop}^{(\sfs;k),\ell}} \leq C\Bigl( \| G P u \|_{H_{\eop;\bop}^{(\sfs-m;k),\ell-\alpha}} + \|u\|_{H_{\eop;\bop}^{(-N;k),\ell}} \Bigr);
\]
this follows from the microlocal elliptic parametrix construction (which uses only the principal symbol map on $\Psie(M)$).

Similarly, one has real principal type propagation estimates: if $B,E,G\in\Psie^0(M)$ are such that $\WFe'(B)\subset\Elle(G)$, and so that all backward null-bicharacteristics from $\WFe'(B)\cap\Char(P)$ reach $\Elle(E)$ in finite time while remaining in $\Elle(G)$, then we have
\begin{equation}
\label{EqSUbRealpr}
  \|B u\|_{H_{\eop;\bop}^{(\sfs;k),\ell}} \leq C\Bigl( \|G P u\|_{H_{\eop;\bop}^{(\sfs-m+1;k),\ell-\alpha}} + \|E u\|_{H_{\eop;\bop}^{(\sfs;k),\ell}} + \|u\|_{H_{\eop;\bop}^{(-N;k),\ell}} \Bigr).
\end{equation}
The proof proceeds by induction on $k$; one applies this estimate for $k-1$ in place of $k$ to $V u$ where $V\in\Vb(M)$ is a spanning set of the space of \emph{commutator b-vector fields}, i.e.\ all vector fields with $[V,W]\in\Ve(M)$ for all $W\in\Ve(M)$. (The space of such $V$ spans $\Vb(M)$ over $\CI(M)$; see~\cite[\S5.1]{HintzVasyScrieb}. In the setting~\eqref{EqDCollar}, one can take all $V\in\Ve(M)$ and in addition $V=\chi(r)\pa_t$ where $\chi\in\CIc([0,\bar r))$ equals $1$ near $0$; recall $\bar r$ from~\eqref{EqDCollar}.) In the term $G P(V u)=G V P u+G[P,V]u$, one must then estimate $\|G[P,V]u\|_{H_{\eop;\bop}^{(\sfs-m+1;k-1),\ell-\alpha}}\leq\|G'u\|_{H_{\eop;\bop}^{(\sfs+1;k-1),\ell}}+\|u\|_{H_{\eop;\bop}^{(-N;k-1),\ell}}$ where $G'\in\Psie^0(M)$ has $\Elle(G')\supset\WFe'(G)$; and the term $G' u$ is then estimated using~\eqref{EqSUbRealpr} again (if, as one may, one chooses $G'$ to have wave front set supported in a small neighborhood of $\WFe'(B)$), now with $k-1,\sfs+1$ in place of $k,\sfs$.

Edge-microlocal estimates at radial points are proved similarly. In the setting of the incoming radial point estimate (Proposition~\ref{PropSUPrIn}), we thus obtain the estimate~\eqref{EqSUbRealpr} when $s>s_0>-\frac12+\ell+\vartheta_{\rm in}$ and $-N=s_0$, further $\WFe'(B),\pa\cR_{\rm in}^\pm\subset\Elle(G)$, and all backward integral curves of $\pm\sfH_{G_\eop}$ from $\WFe'(B)\cap\Sigma^\pm$ either reach $\Elle(E)$ in finite time or tend to $\pa\cR_{\rm in}^\pm$, all while remaining in $\Elle(G)$. (Carefully note that the estimate~\eqref{EqSUbRealpr} only uses microlocalizers with smooth symbols; near $\pa M$, this means concretely that one can localize in $t$, but \emph{not} in $\frac{t-t_+}{r}$.)

An analogous generalization holds also in the outgoing radial point estimate (Proposition~\ref{PropSUPrOut}). The only caveat is that since there is an upper threshold $s_{\rm max}$ for the edge regularity order $s$ for propagation, as in Proposition~\ref{PropSUPrOut}, the threshold condition one obtains from the above argument for edge propagation in $H_{\eop;\bop}^{(s;k),\ell}$ is $s+k<s_{\rm max}$ (cf.\ the final step of the previous argument where the \emph{sum} $s+k=(s+1)+(k-1)$ is unchanged in the inductive argument). We argue that this can be relaxed to the $k$-independent condition $s<s_{\rm max}$ by adapting the idea of \emph{module regularity} from \cite{HassellMelroseVasySymbolicOrderZero} in a manner closely related to \cite[Lemma~5.31 and Proposition~5.32]{HintzNonstat}. We work in the collar neighborhood~\eqref{EqDCollar}.

\begin{lemma}[Commutators with $P$]
\label{LemmaSUbComm}
  Let $V_1=\chi\pa_t$, $V_2=\chi r(\pa_t+\pa_r)$, and $V_j=\chi\Omega_j$, $j=3,\ldots,N$, where the $\Omega_j$ span $\cV(\Sph^{n-1})$ over $\CI(\Sph^{n-1})$. Let $X_i\in\Diffb^1(M;\cE)$, $i=1,\ldots,N$, be an operator with scalar principal symbol equal to $V_i$. Then for $i=1,\ldots,N$, we can write
  \begin{equation}
  \label{EqSUbComm}
    [P,X_i] = a_i P + \sum_{j=1}^N Y_{i,j}X_j + R_i,
  \end{equation}
  where $a_i\in\CI(M)$, $Y_{i,j},R_i\in r^{-2}\Diffe^1(M;\cE)$, and $\sigmae^1(r^2 Y_{i,j})=0$ at $\cR_{\rm out}^\pm$.
\end{lemma}
\begin{proof}
  The addition of any zeroth order term to $X_i$ adds an element of $r^{-2}\Diffe^1$ to the commutator $[P,X_i]$, which can be absorbed into $R_i$. Thus,~\eqref{EqSUbComm} can be checked via a principal symbol calculation. We drop the bundle and work with $V_i$ instead of $X_i$. Since $V_i\in\Ve(M)$ for $i\geq 2$, we have $[P,V_i]\in r^{-2}\Diffe^2$; and also\footnote{One can check this by direct computation using~\eqref{EqSUOp}. More conceptually, we have $[\pa_t,W]\in\Ve(M)$ for all $W\in\Ve(M)$ due to the fact that $\pa_t$ pushes forward to a well-defined vector field on the base $\R_t$ of the fibration of $\pa M$; cf.\ \cite[\S5.1]{HintzVasyScrieb}.} $[V_1,P]\in r^{-2}\Diffe^2(M)$.

  Due to the presence of $a_i$ in~\eqref{EqSUbComm}, it suffices to check that one can write
  \begin{equation}
  \label{EqSUbCommCheck}
    [r^2 P,V_i] \equiv \tilde a_i P + \sum_{j=0}^N \tilde Y_{i,j}V_j \bmod \Diffe^1,
  \end{equation}
  where $\tilde Y_{i,j}\in\Diffe^1$ with $\sigmae^1(\tilde Y_{i,j})|_{\cR_{\rm out}}=0$. For $i=2$, we note that using the coordinates~\eqref{EqDECovec} on the fibers of $\Te^*M$, we have $\sigmae^1(i^{-1}V_2)=\xi-\sigma$. Thus, recalling the expression for $H_{G_\eop}$ (with $G_\eop=r^{-2}G$) from~\eqref{EqDEHamOrig}, we have
  \[
    \sigmae^2([r^2 P,V_2]) = H_{G_\eop}(\xi-\sigma) \equiv 2\sigma^2 - 2\sigma\xi = -G_\eop + (\sigma-\xi)^2 + |\eta|^2 \bmod r P^{[2]}(\Te^*M).
  \]
  Since $\eta=0$ and $\sigma-\xi=0$ at $\cR_{\rm out}$ by~\eqref{EqDERadOut}, and since every function vanishing at $\eta=0$ can be written as a linear combination of the symbols of $V_j$ (with smooth coefficients) near $r=0$, this guarantees~\eqref{EqSUbCommCheck} modulo $\Diffe^1+r\Diffe^2$. Note then that every element of $r\Diffe^2$ can be written in the form $\sum_{j=1}^N \tilde Y_j^\flat X_j+R^\flat$ where $\tilde Y_j^\flat,R^\flat\in r\Diffe^1$ since $r V_1$ and $V_j$, $j\geq 2$, span $\Ve(M)$ over $\CI(M)$; so in particular $\sigmae^1(\tilde Y_j^\flat)|_{\cR_{\rm out}}=0$.

  For $i=3,\ldots,N$, we note that $\sigmae^1(V_i)$ is a linear function in the spherical momentum variable $\eta$, and thus~\eqref{EqDEHamOrig} shows that $\sigmae^2([r^2 P,V_i])$ is a quadratic expression in $\eta$ modulo $r P^{[2]}(\Te^*M)$; arguing as before, $[r^2 P,V_i]$ is therefore of the required form. For $i=1$ finally, we note that the principal symbol of $[\chi\pa_t,r^2 P]\equiv\chi[\pa_t,\Delta_{h(t)}]\bmod\Diffe^1+r\Diffe^2$ is a quadratic expression in $\eta$ as well.
\end{proof}

We then have the following variant of Proposition~\ref{PropSUPrOut}, phrased here for better readability under more geometric assumptions on the microlocalizers.

\begin{prop}[Edge propagation near $\cR_{\rm out}$]
\label{PropSUbRout}
  Suppose that $\chi\in\CIc(M)$ and $B,E,G\in\Psie^0(M)$ are such that the Schwartz kernels of $B,E,G$ are supported in the interior of $\supp\chi\times\supp\chi$, moreover $\WFe'(B),\pa\cR_{\rm out}^\pm\subset\Elle(G)$, and suppose that all backward integral curves of $\pm\sfH_{G_\eop}$ from $\WFe'(B)\cap\pa\Sigma^\pm$ either reach $\Elle(E)$ in finite time or tend to $\pa\cR_{\rm out}^\pm$, all while remaining in $\Elle(G)$. Let $s,s_0,\ell\in\R$ be such that $s_0<s<-\frac12+\ell+\vartheta_{\rm out}$, and let $k\in\N_0$. Then there exists a constant $C>0$ so that
  \begin{equation}
  \label{EqSUbRout}
    \|B u\|_{H_{\eop;\bop}^{(\sfs;k),\ell}} \leq C\Bigl( \|G P u\|_{H_{\eop;\bop}^{(\sfs-1;k),\ell-2}} + \|E u\|_{H_{\eop;\bop}^{(\sfs;k),\ell}} + \|\chi u\|_{H_{\eop;\bop}^{(s_0;k),\ell}} \Bigr).
  \end{equation}
  This holds in the strong sense that if the right hand side is finite, then so is the left hand side and the estimate holds.
\end{prop}
\begin{proof}
  The case $k=0$ is a standard radial point estimate, proved as in Proposition~\ref{PropSUPrOut} (but with localization only in $t$ and not in $\frac{t-t_0}{r}$) for specific choices of $B,E,G$ with operator wave front sets contained in any fixed neighborhood of $\pa\cR_{\rm out}^\pm$; for general $B,E,G$ under the stated assumptions, one combines such a specific estimate with real principal type propagation away from $\pa\cR_{\rm out}^\pm$ and microlocal elliptic estimates away from $\pa\Sigma^\pm$.

  Consider the case $k=1$. Using the notation of Lemma~\ref{LemmaSUbComm}, set $U:=(X_1 u,\ldots,X_N u)$. Write $f=P u$. Then $P X_i u=X_i f+[P,X_i]u$ implies that
  \[
    \cP U=\cF:=(X_i f+a_i f+R_i u)_{i=1,\ldots,N},
  \]
  where $\cP=(\delta_{i j}\cP-Y_{i,j})_{i,j=1,\ldots,N}$ is a principally scalar operator (\emph{with the same principal symbol} as $P$). By Lemma~\ref{LemmaSUbComm}, the principal symbol of $\cP-\cP^*$ at $\cR_{\rm out}^\pm$ is equal to $\Id_{\C^N}$ tensored with that of $P-P^*$; therefore, we can apply the outgoing radial point estimate to $\cP$ \emph{with the same threshold condition} on the edge regularity order; the source terms $R_i u$ are controlled by the inductive hypothesis. Write now $\|B u\|_{H_{\eop;\bop}^{(s;k),\ell}}\lesssim\|B u\|_{\He^{s,\ell}}+\sum_{i=1}^N\|B X_i u\|_{\He^{s,\ell}}+\|[B,X_i]u\|_{\He^{s,\ell}}\sim\|B U\|_{\He^{s,\ell}}+\|\tilde B u\|_{\He^{s,\ell}}$ where $B$ on the right acts component-wise, and $\tilde B\in\Psie^0(M)$ has $\Ell(\tilde B)\supset\WFe'(B)$ and $\WFe'(\tilde B)$ contained in a small neighborhood of $\WFe'(B)$. One can then estimate the first term as in~\eqref{EqSUbRout} but with $\cP U$, $E U$, $\chi U$ on the right, and the second term can be estimated directly using~\eqref{EqSUbRout} with $\tilde B$ in place of $B$. Commuting in a similar manner the operators $X_i$ through $G P$, resp.\ $E$ when estimating the norm of $G P X_i u$, resp.\ $E X_i u$, we obtain~\eqref{EqSUbRout} for $k=1$. The case of general $k\geq 2$ can be handled inductively.
\end{proof}

On the spacetime $M$ and a spacetime domain $\Omega\subset M$ inside it, we shall now prove estimates for forward solutions of $P$ on the Hilbert spaces
\[
  H_{\eop;\bop}^{(\sfs;k),\ell}(\Omega)^{\bullet,-}
\]
which are defined analogously to~\eqref{EqSULocFn} or, equivalently, directly via iterated regularity of elements of $\He^{\sfs,\ell}(\Omega)^{\bullet,-}$ under application of up to $k$ b-vector fields.

\begin{thm}[Higher b-regularity]
\label{ThmSUb}
  Let $\Omega\subset M$ be a non-refocusing spacetime domain, and let $\sfs\in\CI(\Se^*M)$, $\ell\in\R$ be $P$-admissible orders so that also $\sfs-1,\ell$ are $P$-admissible, and so that $P$ is spectrally admissible with weight $\ell$ on $\Omega$. Let $k\in\N_0$. Then there exists a constant $C$ so that for all $f\in H_{\eop;\bop}^{(\sfs-1;k),\ell-2}(\Omega;\cE)^{\bullet,-}$ the unique solution $u\in\He^{-\infty,\ell}(\Omega;\cE)^{\bullet,-}$ of $P u=f$ from Theorem~\usref{ThmSUeNonrf} satisfies
  \[
    u\in H_{\eop;\bop}^{(\sfs;k),\ell}(\Omega;\cE)^{\bullet,-},\qquad
    \|u\|_{H_{\eop;\bop}^{(\sfs;k),\ell}(\Omega;\cE)^{\bullet,-}} \leq C\|f\|_{H_{\eop;\bop}^{(\sfs-1;k),\ell-2}(\Omega;\cE)^{\bullet,-}}.
  \]
\end{thm}
\begin{proof}
  Using Corollary~\ref{CorDNrfEnlarge}, fix a non-refocusing spacetime domain $\Omega'\subset M$ containing $\bar\Omega$. Given $f$, pick an extension $\tilde f\in H_{\eop;\bop}^{(\sfs-1;k),\ell-2}(\Omega';\cE)^{\bullet,-}$ of $f$ with controlled norm. Let $\Omega''\supset\bar\Omega$ be a spacetime domain with closure contained in $\Omega'$. Suppose the Theorem is proved for $k-1\geq 0$ b-derivatives. Denote by $\tilde u\in H_{\eop;\bop}^{(\sfs;k-1),\ell}(\Omega';\cE)^{\bullet,-}$ the solution of $P\tilde u=\tilde f$ on $\Omega'$. (In the case $k-1=0$, this solution is provided by Theorem~\ref{ThmSUeNonrf}.)

  \pfstep{Gaining one b-derivative.} We shall prove that $\tilde u\in H_{\eop;\bop}^{(\sfs-1-\delta;k),\ell}(\Omega'';\cE)^{\bullet,-}$ for all $\delta>0$, that is,
  \begin{equation}
  \label{EqSUbVu}
    V\tilde u\in \He^{(\sfs-1-\delta;k-1),\ell}(\Omega'';\cE)^{\bullet,-}
  \end{equation}
  for all $V\in\Diffb^1(M;\cE)$. For $V\in\Diffe^1(M;\cE)$, this membership is clear. Since $\Vb(M)$ is spanned over $\CI(M)$ by $\Ve(M)$ and $\chi(r)\pa_t$, where $\chi\in\CIc([0,\bar r))$ equals $1$ near $0$, we only need to prove~\eqref{EqSUbVu} for $V\in\Diffb^1(M;\cE)$ with scalar principal part $\chi(r)\pa_t$.

  \textit{From now on, we drop the bundle $\cE$ from the notation.} The basic idea for proving~\eqref{EqSUbVu} for $V=\chi(r)\pa_t$ and in the case $k=1$ is the following. Note that $P(V\tilde u)=\tilde f':=V\tilde f+[P,V]\tilde u$, with $V f\in\He^{\sfs-1,\ell-2}(\Omega')^{\bullet,-}$ and $[P,V]\in r^{-2}\Diffe^2$, so $\tilde f'\in\He^{\sfs-2,\ell-2}(\Omega')^{\bullet,-}$; thus, there exists a unique solution $\tilde u'\in\He^{\sfs-1,\ell}(\Omega')^{\bullet,-}$ (since $\sfs-1,\ell$ are $P$-admissible) of the equation $P\tilde u'=\tilde f'$. However, since a priori only $V\tilde u\in\He^{\sfs-1,\ell-1}(\Omega')^{\bullet,-}$, we cannot immediately conclude that $V\tilde u=\tilde u'$ (which would give~\eqref{EqSUbVu}).\footnote{This \emph{is} possible if the indicial gap, i.e.\ the interval of allowed $\ell$, has length exceeding $1$, and $\ell$ is more than $1$ above the lower endpoint.} To get around this issue, we use a regularization argument. Fix $\eps_0>0$ so that (with $\bar r>0$ from~\eqref{EqDCollar})
  \[
    \eps<\eps_0,\ (t,r,\omega)\in\Omega'',\ r<\bar r \implies T_\eps(t,r,\omega) := (t+\eps\chi(r),r,\omega)\in\Omega'.
  \]
  Then the function
  \[
    \tilde u'_\eps := \eps^{-1}(T_\eps^*\tilde u-\tilde u)
  \]
  is well-defined on $\Omega''$, and for $C>\|\pa_t\sfs\|_{L^\infty(\Se^*_{\Omega'}M)}$ we have $\tilde u'_\eps\in H_{\eop;\bop}^{(\sfs-C\eps_0;k-1),\ell}(\Omega'')^{\bullet,-}$ for all $\eps\in(0,\eps_0)$. Furthermore, on $\Omega'$ we compute
  \begin{equation}
  \label{EqSUbPup}
    P\tilde u'_\eps = \tilde f'_\eps := \eps^{-1}(T_\eps^*\tilde f-\tilde f) + \eps^{-1}\bigl(P(T_\eps^*\tilde u)-T_\eps^*(P\tilde u)\bigr) \in H_{\eop;\bop}^{(\sfs-2-C\eps_0;k-1),\ell-2}(\Omega')^{\bullet,-}.
  \end{equation}
  The first summand of $\tilde f'_\eps$ is uniformly bounded in $H_{\eop;\bop}^{(\sfs-C\eps_0-2;k-1),\ell-2}(\Omega')^{\bullet,-}$ and converges to $\chi\pa_t\tilde f'$. The second summand is equal to $\eps^{-1}(P-T_\eps^*P T_{-\eps}^*)(T_\eps^*\tilde u)=:P'_\eps(T_\eps^*\tilde u)$ where $P'_\eps\in\CI([0,\eps_0);r^{-2}\Diffe^2(M))$; since $T_\eps^*\tilde u$ is uniformly bounded in $H_{\eop;\bop}^{(\sfs-C\eps_0;k-1),\ell}(\Omega'')^{\bullet,-}$, we obtain the \emph{uniform} membership
  \[
    \tilde f'_\eps \in L^\infty\bigl((0,\eps_0)_\eps;H_{\eop;\bop}^{(\sfs-2-C\eps_0;k-1),\ell-2}(\Omega'')^{\bullet,-}\bigr).
  \]
  For sufficiently small $\eps_0>0$ (so that $\sfs-1-C\eps_0,\ell$ are $P$-admissible), the solution of~\eqref{EqSUbPup} therefore satisfies
  \[
    \tilde u'_\eps\in L^\infty((0,\eps_0)_\eps;H_{\eop;\bop}^{(\sfs-1-C\eps_0;k-1),\ell}(\Omega'')^{\bullet,-})
  \]
  by the inductive hypothesis (or Theorem~\ref{ThmSUeNonrf} in the case $k=1$). In view of the distributional convergence $\tilde u'_\eps\to\chi\pa_t\tilde u$, every weak subsequential limit $\tilde u'\in H_{\eop;\bop}^{(\sfs-1-C\eps_0;k-1),\ell}(\Omega'')^{\bullet,-}$ of $\tilde u_\eps'$ as $\eps\searrow 0$ must satisfy $\tilde u'=\chi\pa_t\tilde u$. Since $\eps_0>0$ is arbitrary, this establishes~\eqref{EqSUbVu} for $V=\chi\pa_t$ and thus, as argued before, in general.

  \pfstep{Recovery of edge regularity.} We shall now prove that
  \begin{equation}
  \label{EqSUbEdgeReg}
    \tilde u\in H_{\eop;\bop}^{(\sfs';k),\ell}(\Omega'')^{\bullet,-},\ P\tilde u=\tilde f\in H_{\eop;\bop}^{(\sfs-1;k),\ell-2}(\Omega'')^{\bullet,-} \implies \tilde u\in H_{\eop;\bop}^{(\sfs;k),\ell}(\Omega)^{\bullet,-}
  \end{equation}
  when $\sfs',\ell$ (with $\sfs'\leq\sfs$) and $\sfs,\ell$ are $P$-admissible. But this follows from the edge microlocal regularity estimates discussed above, i.e.\ microlocal elliptic edge regularity, propagation near radial points near the radial sets (using Proposition~\ref{PropSUbRout} near the outgoing radial set), and real principal type propagation in order to control $\tilde u$ microlocally in a neighborhood of $\bar\Omega$ inside $\Omega''$, similarly to (but simpler than) the proof of Proposition~\ref{PropSULocReg}. Applying~\eqref{EqSUbEdgeReg} with $\sfs'=\sfs-1-\delta$, this finishes the inductive step in view of~\eqref{EqSUbVu} and thus completes the proof.
\end{proof}

\begin{rmk}[Results on general spacetime domains]
\label{RmkSUbGeneral}
  The considerations of Remark~\usref{RmkSUeGeneral} on general spacetime domains $\Omega$ (which need not satisfy the non-refocusing condition) extend to the case of higher b-regularity, provided the lower bound required on $s$ there is increased by $1$ order. Importantly, rather than successively applying Theorem~\ref{ThmSUb}---which due to the requirement that also $\sfs-1,\ell$ be $P$-admissible would result in a loss of 1 order of edge regularity for each application of Theorem~\ref{ThmSUb}---one first applies the method of Remark~\ref{RmkSUeGeneral} to obtain an edge-regular solution on $\Omega$; and then an application of Theorem~\ref{ThmSUb} to a cut-off version of the equation $P u=f$ gives the desired improved b-regularity.
\end{rmk}

\begin{rmk}[Coisotropic regularity and diffractive improvements]
\label{RmkSUbCoisotropic}
  Using Remark~\ref{RmkDFlowInOut}, one can consider spaces encoding iterated (coisotropic) regularity, relative to $\He^{\sfs,\ell}$, under application of first order edge pseudodifferential operators with symbol vanishing on the flow-in of $\pa\cR_{\rm in}^\pm$ and on the flow-out of $\pa\cR_{\rm out}^\pm$ under $\pm\sfH_{G_\eop}$. One then has higher coisotropic regularity, relative to $\He^{\sfs,\ell}$, of solutions of $P u=f$. Sharpened to a microlocalized propagation estimate akin to \cite[Theorem~11.1]{MelroseVasyWunschEdge}, one can use this to prove a diffractive improvement for singularities (of waves satisfying a non-focusing condition) hitting the curve of cone points: null-geodesics which are \emph{not} geometric (distance $\pi$ propagation along the fibers) continuations of null-geodesics carrying incoming singularities carry only weaker singularities. We leave the details to the interested reader.
\end{rmk}

\begin{rmk}[Fundamental solutions]
  In order to describe the fundamental solution starting at a point sufficiently close to $\pa M$ (or other solutions with highly singular forcing), one needs to use a refined solution theory, e.g.\ using a non-focusing condition as mentioned in Remark~\ref{RmkSUbCoisotropic}. More simply, one can instead argue via duality using Theorem~\ref{ThmSUb}: the high regularity (by choosing $k$ large) existence and uniqueness statement for \emph{backwards} evolution, i.e.\ on spaces $H_{\eop;\bop}^{(-\sfs+1;k),-\ell+2}(\Omega)^{-,\bullet}$, gives existence and uniqueness for the forward problem on the dual low regularity spaces $\{u_0+\sum_j A_j u_j\colon u_0,u_j\in\He^{\sfs,\ell}(\Omega)^{\bullet,-},\ A_j\in\Diffb^k\}$.
\end{rmk}

In the special case of operators $P$ with $t$-independent b-normal operators~\eqref{EqSUNb}, one can extract asymptotics at $\pa M$:

\begin{thm}[Polyhomogeneity]
\label{ThmSUbPhg}
  Suppose the b-normal operators $N_{\bop,t}(r^2 P)$ of $r^2 P$ are $t$-independent. Let $\Omega\subset M$ be a non-refocusing spacetime domain, and let $\sfs\in\CI(\Se^*M)$, $\ell\in\R$ be $P$-admissible so that also $\sfs-1,\ell$ are $P$-admissible. Suppose $u\in\He^{\sfs,\ell}(\Omega)^{\bullet,-}$ satisfies $P u=f\in\CIdot(M)$, i.e.\ $f$ vanishes to infinite order at $\cC$. Then $u\in\cA_\phg^\cF(\Omega;\cE)^{\bullet,-}$ (i.e.\ $u$ is the restriction to $\Omega$ of an element of $\cA_\phg^\cF(M;\cE)$ vanishing in the past of the initial boundary hypersurfaces of $\Omega$); here, the index set $\cF$ satisfies $\cF\subset\{(z+j,k)\colon z\in\specb(r^2 P),\ j,k\in\N_0\}$ where $\specb(r^2 P)$ is the set of $\xi\in\C$ with $\Re\xi>\ell-\frac{n}{2}$ for which $N_{\bop,t}(r^2 P,\xi)$ is not invertible.
\end{thm}

More generally, if $f$ is polyhomogeneous and $r^{-\ell+1}f$ is square-integrable near $r=0$, then $u$ is polyhomogeneous; this follows from Theorem~\ref{ThmSUbPhg} by first solving away the polyhomogeneous expansion of $f$ to infinite order at $r=0$ in generalized Taylor series, and then applying Theorem~\ref{ThmSUbPhg} to solve away the remaining $\CIdot(M)$ error.

\begin{proof}[Proof of Theorem~\usref{ThmSUbPhg}]
  We extend $f$ to a larger non-refocusing spacetime domain $\Omega'\supset\bar\Omega$, and using Theorem~\ref{ThmSUeNonrf} we also extend $u$. Theorem~\ref{ThmSUb} implies that $u\in H_{\eop;\bop}^{(\sfs;k),\ell}(\Omega';\cE)^{\bullet,-}$ for all $k\in\N_0$, which means that $u\in\Hb^{\infty,\ell}\subset\cA^{\ell-\frac{n}{2}}$ by Sobolev embedding. Rewrite now the equation $P u=f$ as
  \begin{equation}
  \label{EqSUbPhg}
    N_\bop(r^2 P)(\chi u)=\chi f + [P,\chi]u-(P-N_\bop(r^2 P))(\chi u),
  \end{equation}
  where $\chi\in\CIc([0,\bar r))$ equals $1$ near $0$. The right hand side lies in $\cA^{\ell-\frac{n}{2}+1}$. Solving this equation using the (inverse) Mellin transform in $r$ shows that $\chi u$ is the sum of a term which is polyhomogeneous at $\pa M$ near $\bar\Omega$, and a term in $\cA^{\ell-\frac{n}{2}+1-\eps}$ for all $\eps>0$. Plugging this information back into~\eqref{EqSUbPhg} and repeating the argument gives polyhomogeneity of $u$ modulo $\cA^{\ell-\frac{n}{2}+2-\eps}$. Iterating yields the full polyhomogeneity of $u$. (The resulting index set $\cF$ is contained in the extended union of all shifts $\cF_0+j$, $j\in\N_0$, where $\cF_0$ is the smallest index set containing all $(z,k)$ in the divisor of $N_{\bop,t}(r^2 P,\xi)^{-1}$ which satisfy $\Re z>\ell-\frac{n}{2}$.)
\end{proof}

\subsection{Initial value problems}
\label{SsSUIVP}

Via a reduction of initial value problems to forcing problems, we now extend Theorems~\ref{ThmSUeNonrf} and \ref{ThmSUb} to show:

\begin{thm}[Initial value problems]
\label{ThmSUIVP}
  Let $\Omega\subset M$ be a non-refocusing spacetime domain with a single initial boundary hypersurface $X=\{t_{\rm ini,1}=0\}\cap\bar\Omega$. Let $\sfs\in\CI(\Se^*M)$, $\ell\in\R$ be admissible orders near $\bar\Omega$, and assume that $P$ is spectrally admissible in $\Omega$. Let $s_0\geq 1$ be such that $s_0\geq\sup_{\Se^*_\Omega M}\sfs$, and let $k\in\N_0$. If $k\geq 1$ assume moreover that also $\sfs-1,\ell$ are spectrally admissible. Then for all
  \[
    u_0 \in r^{\ell+k}\Hb^{s_0+k}(X;\cE), \quad u_1 \in r^{\ell+k-1}\Hb^{s_0+k-1}(X;\cE),
  \]
  there exists a unique solution\footnote{Here we write $\bar H_{\eop;\bop}^{(\sfs;k)}(\Omega)$ for the space of restrictions of elements of $H_{\eop;\bop}^{(\sfs;k)}(M)$ to $\Omega$. The space $\bar H_\bop^s(X)$ is similarly defined as the space of restrictions to $X$ of elements of $\Hb^s$ on hypersurface containing $X$; that is, away from $X\cap\pa\Omega$ this space is the same as $H^s$.} $u\in r^\ell\bar H_{\eop;\bop}^{(\sfs;k)}(\Omega;\cE)$ of the initial value problem
  \begin{equation}
  \label{EqSUIVP}
    P u = 0,\quad
    u|_X=u_0,\quad
    \pa_t u|_X=u_1.
  \end{equation}
\end{thm}
\begin{proof}
  We drop the bundle $\cE$ from the notation. We only consider the case that $u_0,u_1$ are supported in a small neighborhood of $r=0$. We assume that $t_{\rm ini,1}=t$; the general statement follows by minor (notational) modifications. We blow up the fiber $\phi^{-1}(0)\subset\pa M$ in $M$ and pass freely between edge-notions on $M$ and b-notions near the interior of the front face of $[M;\phi^{-1}(0)]$; cf.\ Remark~\ref{RmkSULocb}. Let $\tau=\frac{t}{r}$. Since $\pa_\tau$ is timelike near $r=0$ for $|\tau|\leq\frac12$, we have local solvability of $r^2 P u=0$, $u|_{\tau=0}=u_0$, $\pa_\tau u|_{\tau=0}=r u_1$ on weighted b-Sobolev spaces on $(-\frac12,\frac12)\times X$ (without any conditions on the weight at $r=0$) by a variant of the arguments in \cite[\S{23.2}]{HormanderAnalysisPDE3}. This produces a local solution $u'\in r^{\ell+k}\Hb^{s_0+k}((-\frac12,\frac12)\times X)$. We then solve~\eqref{EqSUIVP} via
  \[
    u = \chi u' - P^{-1} [P,\chi]u'
  \]
  where $\chi\in\CIc(\R\times X)$ equals $1$ for $\tau\leq\frac18$ and vanishes for $\tau\geq\frac14$. To analyze $u$, note that $\chi u'\in r^{\ell+k}\bar H_\eop^{s_0+k}(\Omega)\subset\bar H_{\eop;\bop}^{(s_0;k),\ell}(\Omega)$ (since $\Vb(M)$ is spanned by $\pa_t\in r^{-1}\Ve(M)$ and $\Ve(M)$) and $[P,\chi]u'\in r^{\ell+k-2}H_\eop^{s_0+k-1}(\Omega)^{\bullet,-}\subset H_{\eop;\bop}^{(s_0-1;k),\ell-2}(\Omega)^{\bullet,-}$. Theorem~\ref{ThmSUb} thus gives $P^{-1}[P,\chi]u'\in H_{\eop;\bop}^{(\sfs;k),\ell}(\Omega)^{\bullet,-}$ and finishes the proof.
\end{proof}

\section{Applications}
\label{SEx}

Our first set of examples in~\S\ref{SsExC} is geometric in nature and concerns wave equations on time-dependent backgrounds which feature conic singularities (with possibly time-dependent cone angles). In the special, well-studied, case of wave equations on ultrastatic spacetimes, we shall relate our solution theory to classical approaches via spectral theory; see~\S\ref{SssExCUStat}. As another special case, we study wave equations on \emph{smooth} spacetimes with artificial curves of cone points; see~\S\ref{SssExCGHyp}. These arise naturally in gluing problems (as discussed in~\S\ref{SI}).

The second set of examples in~\S\ref{SsExSc} concerns wave equations with scaling critical singularities. We discuss time-dependent inverse square potentials in~\S\ref{SssExScV} and the wave equation associated with the Dirac--Coulomb operator in~\S\ref{SssExScC}.

\subsection{Wave equations on spacetimes with timelike conic singularities}
\label{SsExC}

We consider spatial dimensions
\[
  n \in \N,\quad n\neq 2.
\]
Let $h=h(t,\omega;\dd\omega)$ be a smooth family (in $t\in\R$) of Riemannian metrics on $\Sph^{n-1}$. On the interior $M^\circ$ of a manifold with boundary $M$, we then consider a Lorentzian metric, with global time function $t\in\CI(M)$, which in a collar neighborhood $\R_t\times[0,\infty)_r\times\Sph^{n-1}$ of the boundary $\pa M$ takes the form
\begin{equation}
\label{EqExCMetric}
  g = -\dd t^2 + \dd r^2 + r^2 h(t,\omega;\dd\omega) + \tilde g(t,r,\omega;\dd t,\dd r,r\,\dd\omega)
\end{equation}
where $\tilde g$ is a linear combination of symmetric tensor products of $\dd t,\dd r,r\,\dd\omega$ with coefficients in $r\CI(\R_t\times[0,\infty)_r\times\Sph^{n-1})$. We shall study the scalar wave equation
\[
  \Box_g u = f.
\]

The edge normal operators of $t^2\Box_g$ (see~\eqref{EqSUNe}) are
\[
  N_{\eop,t_0}(r^2 \Box_g) = -(r'D_{t'})^2 + r'^2\Bigl(D_{r'}^2-\frac{n-1}{r'}i D_{r'}\Bigr) + \Delta_{h(t_0)},
\]
and correspondingly the threshold quantities introduced in Definition~\ref{DefSUNThr} are $\vartheta_{\rm in}(t_0)=\vartheta_{\rm out}(t_0)=0$. The reduced normal operator (see~\eqref{EqENeRed} and~\eqref{EqSUNRedResc}) is
\[
  \hat N_{\eop,t_0}(\Box_g,\hat\sigma) = D_{\hat r}^2 - \frac{n-1}{\hat r}i D_{\hat r} + \hat r^{-2}\Delta_{h(t_0)} - \hat\sigma^2.
\]

The following is a special case of Lemma~\ref{LemmaExScV} below.

\begin{lemma}[Spectral admissibility]
\label{LemmaExCSpec}
  Let $\ell\in(1-|\frac{n-2}{2}|,1+|\frac{n-2}{2}|)$. Then the operator $\Box_g$ is spectrally admissible with weight $\ell$ at all $t_0\in\R$ (see Definition~\usref{DefSUIAdm}).
\end{lemma}

The threshold conditions for two orders $\sfs\in\CI(\Se^*M)$, $\ell\in\R$ to be $\Box_g$-admissible (see Definition~\ref{DefSULocAdm}) on a non-refocusing spacetime domain $\Omega\subset M$ are
\begin{equation}
\label{EqExCThr}
  \sfs > -\frac12+\ell \ \text{at}\ \pa\cR_{\rm in}^\pm, \qquad
  \sfs < -\frac12+\ell \ \text{at}\ \pa\cR_{\rm out}^\pm.
\end{equation}
We thus obtain the solvability and uniqueness in $\He^{\sfs,\ell}(\Omega)^{\bullet,-}$ of solutions of
\begin{equation}
\label{EqExCBox}
  \Box_g u=f \in \He^{\sfs-1,\ell-2}(\Omega)^{\bullet,-}
\end{equation}
by Theorem~\ref{ThmSUeNonrf}. Furthermore, by Theorem~\ref{ThmSUb}, $u$ has $k$ orders of b-regularity (relative to this edge Sobolev space) if $f$ does. Since $\Box_g$ fits into the setting of Remarks~\ref{RmkSUeGeneral} and \ref{RmkSUbGeneral}, we also obtain a solvability result on general spacetime domains.

\begin{example}[A simple special class of metrics]
\label{ExExCDilInv}
  If $\tilde g=0$, $M=\R_t\times[0,\infty)_r\times\Sph^{n-1}$ in~\eqref{EqExCMetric}, examples of non-refocusing spacetime domains are
  \[
    \Omega=\{t_-<t<t_+,\ r<r_+ + \kappa(t_+-t)\}
  \]
  for arbitrary $t_-<t_+$, $r_+>0$, and sufficiently large $\kappa$ (to guarantee the timelike nature of the second final boundary hypersurface $r=r_++\kappa(t_+-t)$ in $\Omega$). Indeed, the curves $s\mapsto(t_0+s,s,\omega_0)$ are null-geodesics for all $t_0\in\R$, $\omega_0\in\Sph^{n-1}$ (as can be directly checked using~\eqref{EqDEHamOrig}) and thus, for fixed $t_0$ and varying $\omega_0$, are the projection to the base of the flow-out (in the causal future direction) of the outgoing radial set $\pa\cR_{\rm out}^\pm$; since $r$ is monotonically increasing along them, they do not intersect the curve of cone points $r=0$ at any later time.
\end{example}

\subsubsection{Ultrastatic metrics}
\label{SssExCUStat}

In order to relate of our solvability, uniqueness, and regularity theory with the standard approach using spectral theory for the spatial Laplacian, we must restrict to the domain of applicability of the spectral approach and thus consider \emph{ultrastatic metrics}
\[
  g=-\dd t^2+g_X(x,\dd x),
\]
where $g_X$ is a Riemannian metric with a conic singularity on a compact $n$-dimensional manifold $X$ (locally near the cone point given by $[0,\bar r)_r\times\Sph^{n-1}$). This is the special case of~\eqref{EqExCMetric} where $h$ and $\tilde g$ are independent of $t$, and $\tilde g=\tilde g(r,\omega;\dd r,r\,\dd\omega)$ (i.e.\ there are no cross terms involving $\dd t$). Examples of non-refocusing domains $\Omega\subset M=\R_t\times X$ are all domains of the form $[t_-,t_+]\times X$ for all $t_-<t_+$ for which $t_+-t_-$ is less than the length of the shortest geodesic in $(X,g_X)$ starting and ending at the cone point. Write
\begin{equation}
\label{EqExCUStatEll}
  (\ell_-,\ell_+):=\Bigl(1-\Bigl|\frac{n-2}{2}\Bigr|,1+\Bigl|\frac{n-2}{2}\Bigr|\Bigr)
\end{equation}
for the interval of spectrally admissible weights.

We write $\cD^{\tilde s}$ for the domain of the operator $\Delta_{g_X}^{\tilde s/2}$, defined using the functional calculus for the Friedrichs extension of $\Delta_{g_X}$. If $f\in\CI(\R;\cD^{\tilde s-1})$ (or $f\in\sD'(\R;\cD^{\tilde s-1})$) vanishes for $t\leq 0$, then $u\in\CI(\R;\cD^{\tilde s})$, $u|_{t\leq 0}=0$, given by Duhamel's formula
\begin{equation}
\label{EqCDUDuhamel}
  u(t)=\int_0^t \frac{\sin\bigl((t-s)\sqrt{\Delta_{g_X}}\,\bigr)}{\sqrt{\Delta_{g_X}}}f(s)\,\dd s,
\end{equation}
solves $(-D_t^2+\Delta_{g_X})u=f$ in $\CI(\R;\cD^{\tilde s-1})$. The microlocal study of the propagation of singularities for such $u$ in the case $n\geq 2$ is the subject of \cite{MelroseWunschConic}. 

\begin{prop}[Ultrastatic metrics and admissible solutions]
\label{PropCDUStat}
  Let $n\geq 3$ and consider a non-refocusing domain of the form $\Omega=[t_-,t_+]\times X$ with $t_-<0<t_+$. Then the solution $u\in\He^{\sfs,\ell}(\Omega)^{\bullet,-}$ of the scalar wave equation $\Box_g u=f$ produced by Theorem~\usref{ThmSUeNonrf} for a source term $f\in\He^{\sfs-1,\ell-2}(\Omega)^{\bullet,-}$ has a representative which is admissible in the sense of \cite[Definition~4.1]{MelroseWunschConic}. That is, for some $\tilde s$, there exist $u_D\in\dot\sD'([t_-,t_+);\cD^{\tilde s})$ (the space of $\cD^{\tilde s}$-valued distributions on $(-\infty,t_+)$ with support in $t\geq t_-$) and $f_D\in\dot\sD'([t_-,t_+);\cD^{\tilde s-1})$ so that\footnote{For $\tilde s<-\frac{n}{2}$, $\cD^{\tilde s}$ is not a space of distributions in that the map $\cD^{\tilde s}\to\sD'(X^\circ)$ is not injective. See \cite[Remark~6.5]{HintzConicPowers}.} $u_D|_{\Omega\cap M^\circ}=u$ and $f_D|_{\Omega\cap M^\circ}=f$, and $\Box_g u_D=f_D$ holds in $\dot\sD'([t_-,t_+);\cD^{\tilde s-1})$.
\end{prop}

Therefore, the results of \cite{MelroseWunschConic} apply to $u_D$ and yield, upon restriction to $M^\circ$, propagation results for $u$. This includes the statements \cite[Theorems~I.2 and 4.4]{MelroseWunschConic} on the propagation of singularities through $r=0$ which, unlike the results in the present paper (see also Remark~\ref{RmkSUbCoisotropic}) do not lose an (arbitrarily small) amount $\eps>0$ of regularity. Other results in \cite{MelroseWunschConic}, such as the description of the fundamental solution is concerned, lose arbitrarily small positive amounts of regularity just like the present paper (even though we do not give the details of a diffractive theorem here, and thus shall not give a more detailed comparison with \cite{MelroseWunschConic}); note here that there is an automatic loss when working with $L^2$-based Sobolev spaces since $\delta$ at a spacetime point lies in the spacetime Sobolev space $H_\cp^{-\frac{n+1}{2}-\eps}(M^\circ)$ for all $\eps>0$ but not in $H_\cp^{-\frac{n+1}{2}}(M^\circ)$.

\begin{proof}[Proof of Proposition~\usref{PropCDUStat}]
  Let $\psi\in\CIc([t_-,t_+))$ be equal to $1$ near $t_-$; note that $\Box_g(\psi u)=f':=\psi f+[\Box_g,\psi]u\in\dot H_\eop^{\sfs-1,\ell-2}(\Omega)$ and $\psi u\in\dot H_\eop^{\sfs,\ell}(\Omega)$ are supported in $\{t_-\leq t<t_+\}$ (i.e.\ they have supported character at both endpoints $t=t_\pm$, hence the dot on top of $\He$). It thus suffices, by finite speed of propagation, to consider the case that $u,f$ are supported in $\{t_-\leq t_+<t\}$.

  Let $\chi\in\CIc([0,\bar r))$ be equal to $1$ near $0$, and note that $\tilde u=(1-\chi)u$ and $\Box_g((1-\chi)u)=\tilde f:=(1-\chi)f-[\Box_g,\chi]u$ vanish near $r=0$ and thus lie in $H^s_\cp(\R\times X^\circ)$ where $s\leq\inf\sfs-1\in\R$. Integrating $\tilde u$ $N>-s$ times in $t$ from $t=t_-$ produces an element $\tilde u_N\in\cC^0(\R;H^s_\cp(X^\circ))$ vanishing for $t<t_-$. If $s\geq 0$, this lies in $\cD^0$ at all times; otherwise we choose $k\in\N_0$ so that $s+2 k\geq 0$ and solve, using standard b-theory for the elliptic b-operator $\Delta_{g_X}\in r^{-2}\Diffb^2(X)$, the equation $\Delta^k v(t,\cdot)=\tilde u_N(t,\cdot)$ parametrically in $t$ for $v\in\cC^0(\R;\Hb^{s+2 k,\ell_+-\eps}(X))$ where we can fix any $\eps\in(0,\ell_+)$. Thus $v\in\cC^0(\R;\cD^0)$, therefore $\tilde u_N\in\cC^0(\R;\cD^{-k})$, which in turn gives $\tilde u\in\dot\sD'([t_-,t_+];\cD^{-k})$. The equality $\Box_g\tilde u=\tilde f$ holds in distributions and thus (by the vanishing of $\tilde u,\tilde f$ near $r=0$) in $\sD'([t_-,t_+];\cD^{-k-1})$.

  It remains to consider $u,f$ with support in $[t_-,t_+]\times[0,\bar r)_r\times\Sph^{n-1}$. Integrating $N+1$ times in $t$ from $t=t_-$ is the same as convolving with $\frac{1}{N!}t^N H(t)$, where $H$ is the Heaviside function. If we fix $\psi\in\CIc(\R)$ to be equal to $1$ for $|t|\leq t_+-t_-$ and set $\chi_N(t):=\frac{1}{N!}t^N H(t)\psi(t)$, then $\tilde u:=\chi_N*u$ and $\tilde f:=\chi_N*f$ have compact support in $t$, satisfy
  \begin{equation}
  \label{EqCDUStatRecover}
    \Box_g \tilde u = \tilde f;\qquad
      u=\pa_t^N\tilde u,\quad
      f=\pa_t^N\tilde f\quad\text{near}\ \ [t_-,t_+]\times X.
  \end{equation}
   The Fourier transform of $\tilde u$ in $t$ is given by $\wh{\chi_N}(\sigma)\hat u(\sigma)$ where $|\wh{\chi_N}(\sigma)|\lesssim\la\sigma\ra^{-N-1}$ and $\hat u\in L^2(\R_\sigma;r^{\ell-\frac{n}{2}}(\hat M_{|\sigma|})_*(H_{\bop,\scop}^{s,0,s}))$ by Lemma~\ref{LemmaEInvFT}; here we work with the unweighted b-density $|\frac{\dd\hat r}{\hat r}\,\dd h|$ on $[0,\infty]_{\hat r}\times\Sph^{n-1}$, and we recall $\hat M_{|\sigma|}(\hat r,\omega)=(\frac{\hat r}{|\sigma|},\omega)$. But since $r=\frac{\hat r}{|\sigma|}\lesssim 1$, we have $\la\sigma\ra^{-N'}\lesssim\la\hat r\ra^{-N'}$ for $N'\geq 0$. Fix now $N\geq 2$ so that $s+N-2\geq 0$. Then $\la\sigma\ra^{-N-1}H_{\bop,\scop}^{s,0,s}\subset\la\sigma\ra^{-3}H_{\bop,\scop}^{s,0,s+N-2}\subset\la\sigma\ra^{-3}\Hb^s$. Note that $(\hat M_{|\sigma|})_*\Hb^{s'}=\Hb^{s'}$ if one uses the unweighted b-densities $|\frac{\dd\hat r}{\hat r}\,\dd h|$ on $[0,\infty]_{\hat r}\times\Sph^{n-1}$ and $|\frac{\dd r}{r}\,\dd h|$ on $[0,\infty)_r\times\Sph^{n-1}$ (which contains the collar neighborhood of $X$). Therefore,
  \[
    \wh{\chi_N}\hat u \in \la\sigma\ra^{-3}L^2\bigl(\R_\sigma; r^{\ell-\frac{n}{2}}\Hb^s\bigr).
  \]
  Upon taking the inverse Fourier transform and using Sobolev embedding in $\R_t$, and arguing similarly for $\tilde f$, this gives
  \[
    \tilde u\in\cC^2(\R;\Hb^{s,\ell}(X,|\dd g_X|)),\qquad
    \tilde f\in\cC^0(\R;\Hb^{s-1,\ell-2}(X,|\dd g_X|));
  \]
  and $\tilde u,\tilde f$ vanish for $t\leq t_-$ and have compact support in $t$. Choose $k\in\N_0$ so that $\min(s-1+2 k,\ell-2+2 k)\geq s_+:=\max(\ell_+,2)$. We solve $\Delta_{g_X}^k v=\tilde u$, $\Delta_{g_X}^k h=\tilde f$ parametrically in $t$ and obtain
  \[
    v \in \cC^2(\R;\Hb^{s_+,\ell_+-\eps}),\qquad
    h \in \cC^0(\R;\Hb^{s_+,\ell_+-\eps})\quad\forall\,\eps>0;
  \]
  these solve $\Box_g v=h$, as follows from $k$-fold application of $\Delta_{g_X}^{-1}$ to $D_t^2\tilde u=\tilde f+\Delta_{g_X}\tilde u$. But by \cite[Lemma~3.2]{MelroseWunschConic} or \cite[Theorem~6.3(2)]{HintzConicPowers}, we have
  \begin{equation}
  \label{EqCDUStatDom}
    \cD^{\tilde s/2}=\Hb^{\tilde s,\tilde s}(X),\qquad |\tilde s|<\frac{n}{2},
  \end{equation}
  which covers the values $\tilde s=\ell_+-\eps$ and $\ell_+-1-\eps$ for all $\eps>0$. Thus, for $v\in\cC^2(\R;\cD^{s_+/2})$ we have $\Box_g v=h$ also in the sense of $\cC^0(\R;\cD^{s_+/2-1})$. Set then $u_D=\pa_t^N\Delta_{g_X}^k v$ and $f_D=\pa_t^N\Delta_{g_X}^k h$, where $\Delta_{g_X}\colon\cD^{\tilde s}\to\cD^{\tilde s-1}$ is defined via the functional calculus. Restricted to $M^\circ$, we have $u_D=u$ and $f_D=f$ by construction and recalling~\eqref{EqCDUStatRecover}. The proof is complete.
\end{proof}

We prove a converse result only for regular source terms.

\begin{prop}[Ultrastatic metrics and spectral theory]
\label{PropCDUStatSpec}
  Let $n\geq 3$. Let $\ell\in(\ell_-,\ell_+)$, and suppose $f\in\cC^k(\R;\cD^{\ell/2-1})$ vanishes for $t\leq 0$; here $k\in\N_0$, $k>\frac12+\lceil|\ell-2|\rceil$. Then on any non-refocusing domain of the form $\Omega=[t_-,t_+]\times X$, Theorem~\usref{ThmSUeNonrf} is applicable and produces a solution $u\in\He^{\sfs,\ell}(\Omega)^{\bullet,-}$ (for any $\sfs$ satisfying the assumptions of Theorem~\usref{ThmSUeNonrf})) which agrees with the solution defined via spectral theory by equation~\eqref{EqCDUDuhamel}.
\end{prop}

For $f$ which are only distributions in $t$, one can apply Theorem~\ref{ThmSUeNonrf} to the $N$-fold integral $\tilde f$ of $f$ in $t$ for sufficiently large $N$, and conclude from Proposition~\ref{PropCDUStatSpec} that the $N$-th derivative of the resulting solution $\tilde u$ agrees with the solution defined via~\eqref{EqCDUDuhamel}.

\begin{proof}[Proof of Proposition~\usref{PropCDUStatSpec}]
  The main task is to relate the function space for $f$ to a weighted edge Sobolev space. Upon localizing in $t$, we may assume that $f$ vanishes for $t\geq t_+$. By~\eqref{EqCDUStatDom}, we have, a fortiori, $f\in\dot H^k([t_-,t_+];\Hb^{\ell',\ell'}(X))$ where $\ell':=\ell-2$. For $k':=k-\max(\lceil\ell'\rceil,0)$, this implies that
  \begin{equation}
  \label{EqCDUStatSpecDtjf}
    \pa_t^j f\in\dot H_\bop^{\ell',\ell'}([t_-,t_+]\times X),\qquad 0\leq j\leq k'.
  \end{equation}

  We shall further relate this to membership in a weighted edge Sobolev space. To wit, for $\ell'\geq 0$ we have $\pa_t^j f\in\dot H_\eop^{\ell',\ell'}([t_-,t_+]\times X)$. (For $\ell'\in\N_0$, this follows from $\Ve(M)\subset\Vb(M)$, and then for general $\ell'\geq 0$ by interpolation.) To treat the case $\ell'<0$, fix a spanning set $\{V_j\colon 1\leq j\leq N\}\subset\Ve(M)$ of $\Ve(M)$, and note that for any $\eps>0$
  \begin{equation}
  \label{EqCDUStatSpecHb}
  \begin{split}
    &v \in \dot H_\bop^{\sigma,\rho}([t_-,t_+]\times X) \\
    &\qquad \implies \exists\,v_0,v_0^{(0)},v_0^{(i)}\in\dot H_\bop^{\sigma+1,\rho}([t_--\eps,t_++\eps]\times X),\quad v=v_0+\pa_t v_0^{(0)} + \sum_{i=1}^N V_i v_0^{(i)}.
  \end{split}
  \end{equation}
  Indeed, $\pa_t$ and the $V_i$ span $\Vb(M)$, so if $Q\in\Psib^{-2}(M)$ is a parametrix of the elliptic operator $A:=\pa_t^2+\sum_{i=1}^N V_i^2$ and $R=A Q-I\in\Psib^{-1}(M)$, then $v=A Q v-R v$ gives the above decomposition with $v_0=-R v$, $v_0^{(0)}=\pa_t Q v$, $v_0^{(i)}=V_i Q v$. The desired support property can be arranged by multiplication of this decomposition with a cutoff $\chi\in\CIc((t_--\eps,t_++\eps))$ which equals $1$ on $[t_-,t_+]$, and using that $\chi \pa_t v_0^{(0)}=\pa_t(\chi v_0^{(0)})+[\pa_t,\chi]v_0^{(0)}$, where $[\pa_t,\chi]v_0^{(0)}\in\dot H_\bop^{\sigma+1,\rho}([t_--\eps,t_++\eps]\times X)$ can be absorbed into $v_0$; likewise for the terms $\chi V_i v_0^{(i)}$. One can iterate~\eqref{EqCDUStatSpecHb} by applying it to $v_0,v_0^{(0)},v_0^{(j)}$ with $\sigma+1$ in place of $\sigma$. Let $k_0\leq k'$. Applying the $k_0$-th iteration of~\eqref{EqCDUStatSpecHb} to $v=\pa_t^{k_0}f\in\dot H_\bop^{\ell',\ell'}([t_-,t_+]\times X)$, we obtain (using the notation $V^\beta=V_1^{\beta_1}\cdots V_N^{\beta_N}$)
  \[
    \pa_t^{k_0} f = \sum_{p+|\beta|\leq k_0} \pa_t^p V^\beta f_{p,\beta},\qquad f_{p,\beta}\in\dot H_\bop^{\ell'+k_0,\ell'}([t_--\eps,t_++\eps]\times X),
  \]
  which we can integrate from $t=t_--\eps$ to obtain a representation
  \[
    f = \sum_{|\beta|\leq k_0} V^\beta f_\beta,\qquad f_\beta\in\Hb^{\ell'+k_0,\ell'}([t_--\eps,t_++\eps)\times X)^{\bullet,-}.
  \]
  Choosing $k_0\leq k'$ so that $\ell'+k_0\geq 0$ (the choice $k_0=\max(\lceil-\ell'\rceil,0)$ works), we have $f_\beta\in\Hb^{\ell'+k_0,\ell'}\subset\He^{\ell'+k_0,\ell'}$ and therefore $f\in\He^{\ell',\ell'}([t_--\eps,t_++\eps)\times X)^{\bullet,-}$; but since $f$ vanishes for $t\leq 0$ and $t\geq t_+$, and since these arguments apply to up to $k'-k_0$ $t$-derivatives of $f$, we in fact have
  \[
    \pa_t^j f \in \dot H_\eop^{\ell',\ell'}([t_-,t_+]\times X),\qquad 0\leq j\leq k'-k_0.
  \]
  A fortiori, this implies the same membership for $(r D_t)^j f$, $0\leq j\leq k'-k_0$; and since $r \pa_t$ is elliptic on the characteristic set $\Sigma\subset\Te^*M\setminus o$ of $\Box_g$, we conclude that $f\in\He^{s-1,\ell-2}$, $s:=\ell'+k'-k_0+1=\ell-1+k-\lceil|\ell-2|\rceil$, microlocally near $\Sigma$. If $k$ is such that $s>-\frac12+\ell$ (which gives the condition on $k$ in the statement of the Proposition), we may thus find $\sfs\in\CI(\Se^*M)$ so that $\sfs,\ell$ are $\Box_g$-admissible and
  \[
    f\in\dot H_\eop^{\sfs-1,\ell-2}([t_-,t_+]\times X) \subset \He^{\sfs-1,\ell-2}(\Omega)^{\bullet,-}.
  \]
  We can therefore indeed apply Theorem~\ref{ThmSUeNonrf} to solve $\Box_g u=f$ with $u\in\He^{\sfs,\ell}(\Omega)^{\bullet,-}$.

  On the other hand, the solution $u_D$ given by Duhamel's formula~\eqref{EqCDUDuhamel} is a distributional solution of class $\cC^{k+2}(\R;\Hb^{\ell,\ell}(X))$ for the values of $\ell$ considered here. Repeating the above arguments shows that $u_D\in\He^{\sfs,\ell}(\Omega)^{\bullet,-}$; and therefore $\Box_g(u-u_D)=f-f=0$ implies $u=u_D$ by the uniqueness part of Theorem~\ref{ThmSUeNonrf}.
\end{proof}

A very special case is given by Minkowski space $\cM=\R_t\times\R^n_x$ with metric $g=-\dd t^2+\dd x^2$, in which we blow up the curve $\cC=\R_t\times\{0\}$ to obtain $M=\R_t\times X$, $X=[\R^n;\{0\}]$. (The non-compactness of $X$ is irrelevant when working with source terms $f$ which vanish for sufficiently large negative $t$ and have compact spatial support on any bounded $t$-interval.)

\begin{prop}[Edge solvability theory on Minkowski space with artificial curve of cone points]
\label{PropExCDUMink}
   Let $\Omega=[t_-,t_+]\times X$. Suppose $f\in\CI(\Omega)^{\bullet,-}$ (i.e.\ $f$ is the restriction to $\Omega$ of a smooth function which vanishes for $t\leq t_-$) vanishes to infinite order at $\pa M=\R\times\pa X$, and denote by $u\in\He^{\sfs,\ell}(\Omega)^{\bullet,-}$ the solution given by Theorem~\usref{ThmSUeNonrf}, with orders satisfying $\ell\in(\ell_-,\ell_+)=(1-|\frac{n-2}{2}|,1+|\frac{n-2}{2}|)$ and~\eqref{EqExCThr}. Then:
   \begin{enumerate}
   \item for $n\geq 3$, $u$ is equal to the restriction to $\Omega\cap M^\circ$ of the forward solution $u_F\in\CI(\R\times\R^n)$ of $\Box_g u_F=f$ on $\R\times\R^n$;
   \item for $n=1$, $u$ is equal to the restriction to $x\neq 0$ of the forward solution $u_F$ of the initial-boundary value problem $\Box_g u_F=f$, $u_F|_{x=0}=0$.
   \end{enumerate}
\end{prop}
\begin{proof}
  If $\ell\in(\ell_-,\ell_+)$, then for compact $I\subset\R$ and $\bar r<\infty$, one has, for $n\geq 3$,
  \[
    \int_{\Sph^{n-1}}\int_I\int_0^{\bar r} |r^{-\ell} u(t,r,\omega)|^2\,r^{n-1}\,\dd t\,\dd r\,\dd\omega \lesssim \int_0^{\bar r} r^{n-2\ell}\,\frac{\dd r}{r} < \infty
  \]
  since $n-2\ell>n-2\ell_+=0$. In the case $n=1$, the vanishing of $u$ at $r=0$ gives the bound $\int_I\int_0^{\bar r}|r^{-\ell}u(t,r)|^2\,\dd t\,\dd r\lesssim\int_0^{\bar r}r^{3-2\ell}\,\frac{\dd r}{r}<\infty$ since $3-2\ell>3-2\ell_+=0$. Since $u$ is smooth, this implies $u\in\He^{\infty,\ell}(\Omega)^{\bullet,-}$. We trivially have $f\in\He^{\infty,\ell-2}(\Omega)^{\bullet,-}$. The uniqueness part of Theorem~\ref{ThmSUeNonrf} now implies the claim.
\end{proof}

\subsubsection{Globally hyperbolic spacetimes with artificial curves of cone points}
\label{SssExCGHyp}

Let $(\cM,g)$ be an $(n+1)$-dimensional globally hyperbolic spacetime, and let $\cC\subset M$ be an inextendible timelike curve. In Fermi normal coordinates $t\in\R$, $x=(x^1,\ldots,x^n)\in\R^n$ around $\cC$, in which $\cC=\{(t,0)\colon t\in\R\}$, we have
\[
  g=-\dd t^2+\dd x^2+\tilde g
\]
where $\tilde g\in\CI(M;S^2 T^*M)$ vanishes at $x=0$, and all Christoffel symbols $\Gamma_{\mu\nu}^\kappa(t,x)$ vanish at $x=0$ except possibly $\Gamma_{0 0}^j$, $1\leq j\leq n$. Therefore,
\[
  \Box_g = \pa_t^2 - \sum_{j=1}^n \pa_{x^j}^2 - \Gamma_{0 0}^j\pa_{x^j} + \tilde P,
\]
where the coefficients of $\tilde P\in\Diff^2(M)$ vanish at $x=0$. The passage to
\[
  M = [\cM;\cC]
\]
amounts to introducing polar coordinates $x=r\omega$, $\omega\in\Sph^{n-1}$, around $\cC$, in which $r^2\Box_g=r^2(-D_t^2+D_r^2-\frac{n-1}{r}i D_r+r^{-2}\Delta_{\Sph^{n-1}})+r^2\tilde P$. The term $r^2\tilde P\in r\Diffe^2(M)$ does not contribute to the edge normal operator, and therefore Lemma~\ref{LemmaExCSpec} applies to $\Box_g$. Correspondingly, Theorems~\ref{ThmSUeNonrf}, \ref{ThmSUb}, and \ref{ThmSUbPhg} give a \textit{solvability and uniqueness theory on edge Sobolev spaces on non-refocusing spacetime domains $\Omega\subset M$} as in~\eqref{EqExCThr}--\eqref{EqExCBox} with $\ell\in(\ell_-,\ell_+)$ (see~\eqref{EqExCUStatEll}).

By the same arguments as in Proposition~\ref{PropExCDUMink}, the solutions of $\Box_g u=f$ given by this theory for smooth $f$ coincide with the usual forward solution (subject to Dirichlet boundary conditions at $r=0$ in the case $n=1$). The point however is that we can handle source terms $f$ which are highly singular at $\cC$ in that they have only limited edge regularity there.\footnote{The solution of equations $\Box_g u=f$ where $f$ is a \emph{conormal} distribution on $\cM$ at $\cC$---which in $M$ amounts to \emph{infinite b-regularity} at $\pa M$---is easily accomplished using the symbol calculus for conormal distributions to reduce to the case that $f$ vanishes to infinite order at $\cC$ and thus is smooth, in which case solutions of $\Box_g u=f\in\CI(\cM)$ are given by the standard smooth theory.} We expect this to play an important role in applications to spacetime gluing problems.

\begin{example}[Spacetime domains which are not non-refocusing]
\label{ExExCDNonrf}
  Consider a globally hyperbolic spacetime $(M,g)$, a point $p\in M$, and a future null-geodesic $\gamma\colon I\subseteq\R\to M$ with $\gamma(0)=p$ which has a conjugate point at $s_0>0$. Let $s>s_0$ and $q:=\gamma(s)$, then by \cite[Theorem 10.51]{ONeillSemi}, there exists a timelike geodesic through $p$ and $q$; take $\cC$ to be a maximal timelike geodesic passing through $p$ and $q$. A spacetime domain $\Omega\subset M$ which contains both $p$ and $q$ cannot be non-refocusing, as the backward timelike curve segment of $\cC$ starting at $q$ must remain in $\Omega$ until hits $p$; otherwise it would have to exit through an initial boundary hypersurface, which however is impossible since then it cannot reach $p$ afterwards. This situation arises for example in the mass $\bhm>0$ Schwarzschild spacetime
  \[
    M = \R_t\times(2\bhm,\infty)\times\Sph^2,\qquad
    g = -\Bigl(1-\frac{2\bhm}{r}\Bigr)\,\dd t^2 + \Bigl(1-\frac{2\bhm}{r}\Bigr)^{-1}\,\dd r^2 + r^2 g_{\Sph^2},
  \]
  when $\gamma$ is a suitable null-geodesic. A concrete example is $\gamma\colon s\mapsto(t,r,\theta,\phi)=(s,r_0,\frac{\pi}{2},\phi+\alpha s)$ with $r_0=3\bhm$ and $\alpha=\sqrt{\bhm/r_0^3}$; the points $\gamma(0)$ and $\gamma(2\pi/\alpha)$ can be joined by a timelike geodesic segment, and we can take for $\cC$ the maximal extension of this segment.
\end{example}

\subsection{Wave equations with scaling critical singularities}
\label{SsExSc}

Let $n\geq 1$. For expositional simplicity, we restrict the geometry to Minkowski space $(M,g)$ here, so
\begin{gather*}
  M=[\R\times\R^n;\cC],\qquad \cC=\R\times\{0\}; \\
  \Box_g=-D_t^2+D_x^2 = -D_t^2 + D_r^2 - \frac{n-1}{r}i D_r + r^{-2}\Delta_{\Sph^{n-1}},
\end{gather*}
though our results easily generalize to conic geometries as studied in~\S\ref{SsExC}. Fix any spacetime domain $\Omega\subset M=[\R\times\R^n;\cC]$; it is non-refocusing in $(M,g)$.

\subsubsection{Inverse square potentials}
\label{SssExScV}

Consider an (asymptotically as $r\to 0$) inverse square potential
\[
  V(t,r,\omega) = r^{-2}V_0(t,r,\omega),\qquad
  V_0\in\CI(\R_t\times[0,\infty)_r\times\Sph^{n-1}).
\]
We do not require $V$ or $V_0$ to be real-valued.

\begin{lemma}[Spectral admissibility]
\label{LemmaExScV}
  Suppose that $V_0(t,0,\omega)=V_0(t)$, and $t_0\in\R$ is such that $V_0(t_0)\in\C\setminus(-\infty,-(\frac{n-2}{2})^2]$. Then $\Box_g+V$ is spectrally admissible with weight $\ell$ at $t_0$ for
  \begin{equation}
  \label{EqExScVEll}
    \ell \in \bigl(1-\mu(t_0),1+\mu(t_0)\bigr),\qquad \mu(t_0):=\Re\sqrt{\Bigl(\frac{n-2}{2}\Bigr)^2+V_0(t_0)}\,.
  \end{equation}
\end{lemma}
\begin{proof}
  Denote by $0=\lambda_0^2<\lambda_1^2\leq\lambda_2^2\leq\cdots\to\infty$ the eigenvalues of $\Delta_{\Sph^{n-1}}$ (explicitly, $\lambda_j^2=j(j+n-2)$). Restricted to the $\lambda_j^2$-eigenspace of $\Delta_{\Sph^{n-1}}$, we have
  \begin{equation}
  \label{EqExCSpecNj}
    N_j(\hat\sigma)u := -\hat N_{\eop,t_0}(\Box_g+V,\hat\sigma)u = u'' + \frac{n-1}{\hat r}u' - \frac{\lambda_j^2}{\hat r^2}u + \hat\sigma^2 u.
  \end{equation}
  Writing $u=\hat r^{-\frac{n-1}{2}}v$, the equation for $v$ reads
  \begin{equation}
  \label{EqExCSpecEqv}
    v'' - \Bigl[\Bigl(\frac{(n-1)(n-3)}{4}+\lambda_j^2+V_0(t_0)\Bigr)\hat r^{-2}-\hat\sigma^2\Bigr]v = 0.
  \end{equation}
  The indicial roots at $\hat r=0$ are $\frac12\pm\nu_j$, $\nu_j:=\sqrt{((n-2)/2)^2+\lambda_j^2+V_0(t_0)}$, with $\Re\nu_j>0$ by our assumption on $V_0(t_0)$. For $\hat\sigma=\pm 1$, consider the Wronskian $W$ of $v$ and $\bar v$. The root $\frac12-\nu_j$ is excluded and $|\hat r^{\frac12+\nu_j}\cdot\hat r^{\frac12+\nu_j-1}|=\cO(|\hat r^{2\nu_j}|)=o(1)$ as $\hat r\to 0$, so the Wronskian must be identically $0$. On the other hand, we must have $v=v_\infty e^{i\hat\sigma\hat r}+\tilde v$ where $|\tilde v|$, $|\tilde v'|=o(1)$ as $\hat r\to\infty$, so $W=-2 i\hat\sigma|v_\infty|^2$. Therefore $v_\infty=0$, and this implies $v=0$ and thus $u=0$.

  For $|\hat\sigma|=1$, $\Im\hat\sigma>0$, we multiply~\eqref{EqExCSpecEqv} by $\bar v$ and integrate by parts; for $\hat\sigma\neq i$, so $\hat\sigma^2\notin\R$, taking imaginary parts gives $\int|v|^2\,\dd\hat r=0$ and thus $v=0$; otherwise taking real parts implies $v=0$.

  The adjoint of $-\hat N_{\eop,t_0}(\Box_g+V,\hat\sigma)$ acts on the $\lambda_j^2$-eigenspace of $\Delta_{\Sph^{n-1}}$ as $N_j(\bar{\hat\sigma})$, and its injectivity as required in Definition~\ref{DefSUIAdm}\eqref{ItSUIAdmAdj} follows from completely analogous considerations.

  In the general case of complex $V_0(t_0)$, we instead need to argue similarly to \cite[Lemma~5.10]{HintzConicProp}. To wit, $v$ is a linear combination of $v_1(\hat r)=\hat r^{\frac12}H_{\nu_j}^{(1)}(\hat\sigma\hat r)$ and $v_2(\hat r)=\hat r^{\frac12}H_{\nu_j}^{(2)}(\hat\sigma\hat r)$. Of these two, only $v_1$ is outgoing at infinity when $\hat\sigma=\pm 1$ and exponentially decaying when $\Im\hat\sigma>0$ by \cite[\S{7.4.1}, equation~(4.03)]{OlverSpecial}. But given the required lower bound $\ell>1-\mu(t_0)$, $|v_1|\gtrsim|\hat r^{\frac12-\nu_j}|$ (using the combination of \cite[\S{7.4.2}, equation~(4.12)]{OlverSpecial} and \cite[\S{2.9.3}, equation~(9.09)]{OlverSpecial}) is too large at $\hat r=0$ in the context of Definition~\ref{DefSUIAdm}.
\end{proof}

If $V_0(t)\in\C\setminus(-\infty,-(\frac{n-2}{2})^2]$ for all $t$ so that $(t,0)\in\bar\Omega$, we can therefore apply Theorem~\ref{ThmSUeNonrf} to the equation
\[
  (\Box_g+V)u=f \in \He^{\sfs-1,\ell-2}(\Omega)^{\bullet,-}
\]
and obtain a unique solution $u\in\He^{\sfs,\ell}(\Omega)^{\bullet,-}$; here $\ell$ lies in the interval given by Lemma~\ref{LemmaExScV}, and $\sfs$ satisfies~\eqref{EqExCThr}. By Theorem~\ref{ThmSUb}, $u$ enjoys $k$ additional degrees of b-regularity (i.e.\ differentiability along $\pa_t$, $r\pa_r$, and spherical vector fields) if $f$ has the same additional degree of b-regularity.

\begin{rmk}[Mild damping: Theorem~\ref{ThmIBaby}]
\label{RmkExScVMild}
  In the context of Theorem~\ref{ThmIBaby}, we consider an additional term $a(t,r,\omega)\pa_t$ where $a=r^{-1}a_0$ with $a_0(t,0,\omega)=a_0(t)\geq 0$; this has the effect of introducing a term in~\eqref{EqSUNe} with $a(t_0,\omega)$ there given by $a_0(t_0)$ in present notation, and thus $\vartheta_{\rm out}(t_0)=\frac12 a_0(t_0)=-\vartheta_{\rm in}(t_0)$ in Definition~\ref{DefSUNThr}. If we choose $V_0(t_0)$ so that $\mu(t_0)>\frac12$ in~\eqref{EqExScVEll}, then the weight $\ell=\frac32$ is spectrally admissible, and then the conditions on the edge regularity order $\sfs$ in Definition~\ref{DefSULocAdm} are satisfied by the \emph{constant} order $\sfs=1$ as long as $a_0(t)>0$ for all $t$ so that $(t,0)\in\bar\Omega$. The spectral admissibility of $\Box_g+V+a\pa_t$ at $t_0$ persists for small $|a_0(t_0)|$ by a simple perturbative argument utilizing the Fredholm theory in Lemma~\ref{LemmaSUIInv}.
\end{rmk}

\subsubsection{Dirac--Coulomb operator}
\label{SssExScC}

We let $X=[\R^3;\{0\}]=[0,\infty)_r\times\Sph^2$ and $M=\R\times X$. The Dirac--Coulomb equation for a massive fermion in Minkowski space, as considered in \cite{BaskinWunschDiracCoulomb}, is
\begin{align*}
  &(i\slpa_\bfA-m)\psi=0,\qquad
  \slpa_\bfA = \gamma^\mu(\pa_\mu+i A_\mu), \\
  &\quad \gamma^0 = \begin{pmatrix} I & 0 \\ 0 & -I \end{pmatrix},\quad
    \gamma^j = \begin{pmatrix} 0 & \sigma_j \\ -\sigma_j & 0 \end{pmatrix}, \\
  &\quad \bfA=(A_0,A_1,A_2,A_3),\quad A_0=\frac{\sfZ}{r}+V,\quad V,A_1,A_2,A_3\in\CI(X),
\end{align*}
where $\psi\colon\R_t\times\R^3_x\to\C^4$, and
\[
  \sigma_1=\begin{pmatrix} 0 & 1 \\ 1 & 0 \end{pmatrix},\quad
  \sigma_2=\begin{pmatrix} 0 & -i \\ i & 0 \end{pmatrix},\quad
  \sigma_3=\begin{pmatrix} 1 & 0 \\ 0 & -1 \end{pmatrix}
\]
are the Pauli matrices. Following \cite[\S4.3]{BaskinWunschDiracCoulomb}, $i\slpa_\bfA-m$ gives rise to the wave operator
\begin{equation}
\label{EqExScCOp}
\begin{split}
  P &:= -(i\slpa_\bfA+m)(i\slpa_\bfA-m) = -\Bigl(D_t+\frac{\sfZ}{r}\Bigr)^2 + \Delta + m^2 + i\frac{\sfZ}{r^2}\alpha_r + r^{-1}\bfR, \\
    &\quad \alpha_r=\sum_{j=1}^3 \frac{x_j}{r}\gamma^0\gamma^j,\quad \bfR\in\Diffb^1(X;\C^4).
\end{split}
\end{equation}
This form of the operator is valid even when $\sfZ=\sfZ(t)$ is allowed to be time-dependent; terms involving $\sfZ'$ can be put into the operator $\bfR$. The operator $P$ fits into the general setup of~\S\ref{SSU}. The reduced normal operator is
\[
  \hat N_{\eop,t_0}(P,\hat\sigma) = D_{\hat r}^2 - \frac{2}{\hat r}i D_{\hat r} - \Bigl(\hat\sigma-\frac{\sfZ(t_0)}{\hat r}\Bigr)^2 + \hat r^{-2}\bigl(\Delta_{\Sph^2} + i\sfZ(t_0)\alpha_r\bigr).
\]
The proof of \cite[Lemma~5.3]{HintzConicProp} generalizes to show:

\begin{lemma}[Spectral admissibility]
\label{LemmaExScCAdm}
  Let $t_0\in\R$, and suppose that\footnote{Using a special function analysis, one can also treat the case of suitable $\sfZ(t_0)\in\C$; we shall not pursue this further here.} $\sfZ(t_0)\in\R$ satisfies $|\sfZ(t_0)|\neq\sqrt{\kappa^2-1/4}$ for all $\kappa\in\Z\setminus\{0\}$. Then the operator $P$ in~\eqref{EqExScCOp} is spectrally admissible with weight $\ell$ at $t_0$ for all $\ell\in(1-\delta,1+\delta)$ where $\delta=\min_{\kappa\in\Z\setminus\{0\}}|\frac12-\sqrt{\kappa^2-\sfZ(t_0)^2}|$.
\end{lemma}

If $\sfZ_0(t)\neq\pm\sqrt{\kappa^2-1/4}$ for all $\kappa\in\Z\setminus\{0\}$ and for all $t$ so that $(t,0)\in\bar\Omega$, we can then apply Theorem~\ref{ThmSUeNonrf} to the equation $P u=f \in \He^{\sfs-1,\ell-2}(\Omega)^{\bullet,-}$. This gives a unique solution $u\in\He^{\sfs,\ell}(\Omega)^{\bullet,-}$, and we have thus solved
\[
  (i\slpa_\bfA-m)\psi = f,\qquad \psi:=-(i\slpa_\bfA+m)u\in\He^{\sfs-1,\ell-1}(\Omega;\C^4)^{\bullet,-}.
\]
By Theorem~\ref{ThmSUb}, $u$ and $\psi$ enjoy $k$ additional degrees of b-regularity (i.e.\ differentiability along $\pa_t$, $r\pa_r$, and spherical vector fields) if $f$ has the same additional degree of b-regularity. (Following Remark~\ref{RmkSUbCoisotropic}, one may then entertain the proof of diffractive improvements for solutions of $P u=f$, though we shall not do so here.)

\appendix
\section{Edge geometry and analysis}
\label{SE}

We begin in~\S\ref{SsEBasic} by collecting basic notions of b- and edge-analysis before recalling edge pseudodifferential operators (with variable differential orders) in~\S\ref{SsEPsdo}. In~\S\ref{SsEInv}, we discuss invariant edge Sobolev spaces in some detail; this is used in the analysis of the edge normal operator in~\S\ref{SsSUI}. Apart from Remark~\ref{RmkEBlowup} and~\S\ref{SsEInv}, the material discussed here is essentially standard, or at least a variation on well-known themes. See also~\cite[\S2]{HintzVasyScrieb}.

\subsection{Basic notions of geometric singular analysis}
\label{SsEBasic}

Let $M$ be a smooth $n$-dimensional manifold with boundary $\pa M$. Then the Lie algebra $\Vb(M)\subset\cV(M)=\CI(M;T M)$ of \emph{b-vector fields} \cite{MelroseMendozaB,MelroseAPS} consists of all vector fields which are tangent to $\pa M$. In local coordinates $x\geq 0$, $y\in\R^{n-1}$ near a point in $\pa M$, such vector fields are linear combinations of $x\pa_x$, $\pa_{y^j}$ ($j=1,\ldots,n-1$) with coefficients in $\CI([0,\infty)\times\R^{n-1})$; thus, there is a vector bundle $\Tb M\to M$ (called the \emph{b-tangent bundle}), equipped with a bundle map $\Tb M\to T M$ which is an isomorphism over $M^\circ$, so that $\CI(M;\Tb M)=\Vb(M)$: a local frame of $\Tb M$ is $x\pa_x$, $\pa_{y^j}$. The dual bundle $\Tb^*M$, with local frame $\frac{\dd x}{x}$, $\dd y^j$, is the \emph{b-cotangent bundle}.

Suppose now that $\pa M$ is the total space of a fibration
\begin{equation}
\label{EqEFibration}
  Z - \pa M \xra{\phi} Y
\end{equation}
where $Z,Y$ are smooth manifolds without boundary. Following \cite{MazzeoEdge}, we can then define the Lie subalgebra $\Ve(M)\subset\Vb(M)$ of \emph{edge vector fields} to consist of all smooth vector fields which are tangent to the fibers of $\phi$. In local coordinates $x\geq 0$, $y\in\R^{n_Y}$, $z\in\R^{n_Z}$ (with $n=1+n_Y+n_Z$), such \emph{edge vector fields} are smooth linear combinations of the vector fields $x\pa_x$, $x\pa_{y^j}$, $\pa_{z^k}$, which thus provide a local frame for the \emph{edge tangent bundle} $\Te M\to M$.

Note that b-vector fields define continuous linear maps on $\CI(M)$ and $\CIdot(M)$ (smooth functions vanishing to infinite order at $\pa M$), and by duality also on the space of tempered distributions $\CmI(M)=\CIdot(M)^*$. The space of locally finite linear combinations of up to $m$-fold compositions of b-vector fields is denoted $\Diffb^m(M)$; if $\cE,\cF\to M$ are smooth vector bundles, then $\Diffb^m(M;\cE,\cF)$ (and $\Diffb^m(M;\cE)=\Diffb^m(M;\cE,\cF)$) consists of $(\rank\cF)\times(\rank\cE)$-matrices of elements of $\Diffb^m$ in local trivializations of $\cE,\cF$. The spaces $\Diffe^m(M)$ and $\Diffe^m(M;\cE,\cF)$ are defined analogously.

Let $\rho\in\CI(M)$ denote a \emph{boundary defining function}, i.e.\ $\pa M=\rho^{-1}(0)$, $\dd\rho\neq 0$ on $\pa M$, and $\rho>0$ on $M^\circ$. Then $\Vsc(M):=\rho\Vb(M)=\{\rho V\colon V\in\Vb(M)\}$ is the Lie algebra of \emph{scattering vector fields} \cite{MelroseEuclideanSpectralTheory}.

\textbf{Differential operators and their normal operators.} Let now $P\in\Diffe^m(M)$, so in local coordinates as above
\[
  P=\sum_{j+|\alpha|+|\beta|\leq m}p_{j\alpha}(x,y,z)(x D_x)^j(x D_y)^\alpha D_z^\beta
\]
where $D=\frac{1}{i}\pa$. We define its \emph{(edge) principal symbol} by
\[
  \sigmae^m(P)(x,y,z;\xi,\eta,\zeta):=\sum_{j+|\alpha|+|\beta|=m}p_{j\alpha\beta}(x,y,z)\xi^j\eta^\alpha\zeta^\beta;
\]
upon identifying $\xi,\eta,\zeta$ with linear coordinates on the fibers of $\Te^*M$ via $\xi=x\pa_x(\cdot)$, $\eta_j=x\pa_{y^j}(\cdot)$, $\zeta_j=\pa_{z^j}(\cdot)$ (or equivalently: writing the canonical 1-form as $\xi\frac{\dd x}{x}+\eta\cdot\frac{\dd y}{x}+\zeta\cdot\dd z$), this is a well-defined homogeneous polynomial in the fibers of $\Te^*M$.

While $\sigmae^m(P)$ captures $P$ modulo $\Diffe^{m-1}(M)$, i.e.\ to leading order in the sense of differential order, the \emph{(edge) normal operators} of $P$ capture $P$ modulo $\rho\Diffe^m(M)$, i.e.\ to leading order at $\pa M$. In local coordinates, the edge normal operator at the fiber $\phi^{-1}(y_0)$ of $\pa M$ is the operator
\begin{equation}
\label{EqENe}
  N_{\eop,y_0}(P) := \sum_{j+|\alpha|+|\beta|=m} p_{j\alpha\beta}(0,y_0,z)(x D_x)^j (x D_y)^\alpha D_z^\beta
\end{equation}
on $[0,\infty)_x\times\R^{n_Y}_y\times\R^{n_Z}_z$; this is invariant under translations $(x,y,z)\mapsto(x,y+c,z)$, $c\in\R^{n_Y}$, and dilations $(x,y,z)\mapsto(\lambda x,\lambda y,z)$, $\lambda>0$. More precisely (and globally in the fibers), we can regard
\[
  N_{\eop,y_0}(P) \in \Diff_{\eop,\rm I}^m({}^+N\phi^{-1}(y_0))
\]
as an invariant edge operator on the (non-strictly) inward pointing normal bundle
\[
  {}^+N\phi^{-1}(y_0)={}^+T_{\phi^{-1}(y_0)}M/T\phi^{-1}(y_0)
\]
which is equipped with a translation action of $T_{\phi^{-1}(y_0)}\pa M/T\phi^{-1}(y_0)$ and a dilation action (descending from dilations in the fibers of $T M$), and whose boundary (the zero section) is the total space of a fibration with base $T_{y_0}Y$ induced by $\phi_*\colon T_{\phi^{-1}(y_0)}M\to T_{y_0}Y$. Continuing the discussion of $N_{\eop,y_0}(P)$ in local coordinates as in~\eqref{EqENe}, we take the Fourier transform in $y$, i.e.\ formally replacing $D_y$ by multiplication with $\eta\in\R^{n_Y}$ (or $\C^{n_Y}$) to obtain the `spectral family'
\begin{equation}
\label{EqENeSpec}
  N_{\eop,y_0}(P,\eta) = \sum_{j+|\alpha|+|\beta|=m} p_{j\alpha\beta}(0,y_0,z)(x D_x)^j (x\eta)^\alpha D_z^\beta.
\end{equation}
Then, taking advantage of the dilation invariance $(x,\eta)\mapsto(\lambda x,\eta/\lambda)$, we take $\lambda=|\eta|$ (using the Euclidean norm, say) and introduce $\hat x=|\eta|x$, $\hat\eta=\frac{\eta}{|\hat\eta|}$, yielding the \emph{reduced normal operator}
\begin{equation}
\label{EqENeRed}
  \hat N_{\eop,y_0}(P,\hat\eta) = \sum_{j+|\alpha|+|\beta|=m} p_{j\alpha\beta}(0,y_0,z)(\hat x D_{\hat x})^j (\hat x\hat\eta)^\alpha D_z^\beta
\end{equation}
which contains the same information as $N_{\eop,y_0}(P,\eta)$. Indeed, if we set
\begin{subequations}
\begin{equation}
\label{EqEScale1}
  \hat M_\lambda \colon (\hat x,z) \mapsto (\hat x/\lambda,z),
\end{equation}
then
\begin{equation}
\label{EqEScale2}
  N_{\eop,y_0}(P,\eta) = (\hat M_{|\eta|})_*\hat N_{\eop,y_0}\Bigl(P,\frac{\eta}{|\eta|}\Bigr).
\end{equation}
\end{subequations}
On the space $\hat X:=[0,\infty]_{\hat x}\times Z_z$, we thus have, by inspection of~\eqref{EqENeRed},
\begin{equation}
\label{EqEDiffbsc}
  \hat N_{\eop,y_0}(P,\hat\eta) \in (1+\hat x)^m\Diff_{\bop,\scop}^m(\hat X),
\end{equation}
with smooth dependence on $\hat\eta\in\Sph^{n_Y-1}$, where $\Diff_{\bop,\scop}(\hat X)$ is defined in terms of the space $\cV_{\bop,\scop}(\hat X)$ of vector fields which are b-vector fields in $\hat x<\infty$ and scattering vector fields in $\hat x^{-1}<\infty$. (These are thus spanned, in local coordinates on $Z$, by $\frac{\hat x}{1+\hat x}\pa_{\hat x}$ and $(1+\hat x)^{-1}\pa_z$.) We write $\pa_0\hat X=\hat x^{-1}(0)$ and $\pa_\infty\hat X=\hat x^{-1}(\infty)$.

When $P\in\Diffe^m(M;\cE,\cF)$, the operator $N_{\eop,y_0}(P)$ is an invariant edge operator acting between sections of the bundles $\cE_{y_0}$, $\cF_{y_0}$ which are defined as the pullbacks of $\cE|_{\phi^{-1}(y_0)}$, $\cF|_{\phi^{-1}(y_0)}$ along the projection ${}^+N\phi^{-1}(y_0)\to\phi^{-1}(y_0)$; and we have $\hat N_{\eop,y_0}(P,\hat\eta)\in(1+\hat x)^m\Diff_{\bop,\scop}^m(\hat X;\cE_{y_0},\cF_{y_0})$.

Consider now the case $P\in\Diffb^m(M)$; then in coordinates $x\geq 0$, $y\in\R^{n-1}$, we have $P=\sum_{j+|\alpha|\leq m} p_{j\alpha}(x,y)(x D_x)^j D_y^\alpha$. We define its b-normal operator by
\[
  N_\bop(P) := \sum_{j+|\alpha|\leq m} p_{j\alpha}(0,y) (x D_x)^j D_y^\alpha;
\]
invariantly, this is an element of $\Diff_{\bop,\rm I}^m({}^+N\pa M)$, where the subscript `$\rm I$' indicates invariance under dilations in the fibers (i.e.\ $(x,y)\mapsto(\lambda x,y)$, $\lambda>0$, in local coordinates). Exploiting the dilation-invariance by passing to the Mellin transform, i.e.\ replacing $x\pa_x$ by multiplication by $\xi\in\C$, produces the indicial family
\[
  N_\bop(P,\xi) := \sum_{j+|\alpha|\leq m} p_{j\alpha}(0,y) (-i\xi)^j D_y^\alpha.
\]
When $x$ is a boundary defining function, then this defines a holomorphic (polynomial) family of elements of $\Diff^m(\pa M)$.

\textbf{Sobolev spaces.} Let $\rho\in\CI(M)$ denote a boundary defining function. A \emph{weighted b-density} $\mu$ (with weight $w\in\R$) is then an element of $\rho^w\CI(M;\Omegab M)$ where $\Omegab M$ is the density bundle associated with $\Tb M$; that is, $\mu$ is a smooth density on $M^\circ$ which near a boundary point and in coordinates $x\geq 0$, $y\in\R^{n-1}$ is given by $\rho^w|\frac{\dd x}{x}\,\dd y^1\cdots\dd y^{n-1}|$. Suppose that $M$ is compact. For $s\in\N_0$, $\alpha\in\R$, we can then define the weighted b-Sobolev space $\Hb^{s,\alpha}(M;\mu)=\rho^\alpha\Hb^s(M;\mu)$ to consist of all $u$ so that $\rho^{-\alpha}A u\in L^2(M;\mu)$ for all $A\in\Diffb^m(M)$. For $s\in\R$, the space $\Hb^{s,\alpha}$ can be defined using interpolation and duality. When $M$ is non-compact, the spaces $H_{\bop,\cp}^{s,\alpha}(M;\mu)$ and $H_{\bop,\loc}^{s,\alpha}(M;\mu)$, defined relative to $L^2_\cp(M;\mu)$ and $L^2_\loc(M;\mu)$, are well-defined; in practice, we will only be concerned with distributions having support in a fixed compact subset of $M$. Weighted edge Sobolev spaces $\He^{s,\alpha}(M;\mu)$ are defined analogously. When the density $\mu$ is clear from the context, we omit it from the notation. Spaces of sections of a vector bundle $\cE$ with weighted edge Sobolev regularity are denoted $\He^{s,\alpha}(M;\cE)$.

\textbf{Conormal functions and symbols.} Let $M$ be a manifold with corners; thus, $M$ is locally diffeomorphic to $[0,\infty)^k\times\R^{n-k}$, and we require each of its boundary hypersurfaces (i.e.\ the closures of the connected components of the set of points having neighborhoods diffeomorphic to $[0,\infty)\times\R^{n-1}$) to be embedded submanifolds. (Equivalently, each such boundary hypersurface $H\subset M$ admits a defining function $\rho_H\in\CI(M)$ which vanishes only at $H$ with $\dd\rho|_H\neq 0$ and is positive on $M^\circ$.) Write $\cH$ for the collection of boundary hypersurfaces. Let $\alpha=(\alpha_H)_{H\in\cH}$, $\alpha_H\in\R$; then $\cA^\alpha(M)$ is the space of all $u\in\CI(M^\circ)$ so that $P u\in w L^\infty_\loc(M)$ (i.e.\ $w^{-1}P u\in L^\infty_\loc(M)$) for all $P\in\Diffb^k(M)$, $k\in\N_0$, where $w=\prod_{H\in\cH}\rho_H^{\alpha_H}$. We call elements of $\cA^\alpha(M)$ \emph{conormal} with weight $\alpha_H$ at $H$. More generally, given a parameter $\delta\geq 0$, we can define $\cA^\alpha_\delta(M)$ to consist of all $u\in\CI(M^\circ)$ so that for all $k\in\N_0$ and $P\in\Diffb^k(M)$ we have $P u\in w(\prod_{H\in\cH}\rho_H^{-k\delta})L^\infty_\loc(M)$. (The space $\cA^\alpha_\delta(M)$ becomes smaller when $\delta$ decreases.)

If $\cH=\cH_1\sqcup\cH_2$ and we are given $\alpha=(\alpha_H)_{H\in\cH_1}$, one can define the space $\cA_{\cH_1}^\alpha(M)$ in the same manner, except now $w=\prod_{H\in\cH_1}\rho_H^{\alpha_H}$ and we allow $P$ to be any composition of smooth vector fields which are tangent to all elements of $\cH_1$ but not necessarily to the remaining boundary hypersurfaces; we say that the elements of $\cA_{\cH_1}^\alpha(M)$ are conormal at all $H\in\cH_1$ and smooth at the boundary hypersurfaces in $\cH_2$. When $\cH=\cH_1\sqcup\cH_2\sqcup\cH_3$, elements of the more general spaces $\cA_{\cH_1,\delta;\cH_2}^\alpha(M)$ are allowed to incur a loss of $\prod_{H\in\cH_1}\rho_H^{-\delta}$ for the application of each vector field tangent to all elements of $\cH_1$ and $\cH_2$. The main motivation for allowing for the loss $\delta>0$ is that this allows one to accommodate variable weights (or orders) at the hypersurfaces in $\cH_1$: given bounded functions $\upalpha_H\in\CI(H)$, $H\in\cH_1$, denote an arbitrary extension of $\upalpha_H$ to a smooth function on $M$ by the same symbol; we can then define
\[
  \cA^\upalpha(M) := \left(\prod_{H\in\cH_1} \rho_H^{\upalpha_H}\right) \bigcap_{\delta>0} \cA^0_{\cH_1,\delta}(M).
\]
The inclusion of $\delta$ here ensures that the resulting space is independent both of the choices of extensions of the $\upalpha_H$ to $M$ and of the choices of boundary defining functions.

If $\cE\to M$ is a smooth real vector bundle of rank $k$, then the \emph{fiber-radial compactification} $\bar\cE\to M$ is a fiber bundle with typical fiber a closed $k$-ball. To define it, we first recall the radial compactification of $\R^k$,
\[
  \ol{\R^k} := \Bigl(\R^k \sqcup \bigl( [0,\infty)_\rho\times\Sph^{n-1}_\omega \bigr) \Bigr) / \sim,\qquad (0,\infty)\times\Sph^{n-1} \ni (\rho,\omega) \sim \rho^{-1}\omega \in \R^n;
\]
thus $\rho^{-1}(0)=\pa\ol{\R^k}\cong\Sph^{k-1}$ is the boundary at infinity, and $\R^k$ is the interior of $\ol{\R^k}$. Since invertible linear maps on $\R^k$ extend to diffeomorphisms of $\ol{\R^k}$, we can define the fiber bundle $\bar\cE$ via its local trivialization $U\times\ol{\R^k}$ when $U\times\R^k$ is a local trivialization of $M$, with $U\subset M$ open. Fiber infinity $S\cE\to M$ is a sphere bundle (and a fiber subbundle of $\bar\cE\to M$), locally given by $U\times\Sph^{k-1}$ where $\Sph^{k-1}=\pa\ol{\R^k}$. Given $s\in\R$, we can now define the space of symbols $S^s(\cE)=\cA_{S\cE}^{-s}(\bar\cE)$ and $S^s_{1-\delta,\delta}(\cE)=\cA_{S\cE,\delta}^{-s}(\bar\cE)$. In the special case $\cE=T^*\R^n=\R^n\times\R^n\to\R^n$, $S^s(T^*\R^n)$ is the standard space of Kohn--Nirenberg symbols, i.e.\ functions $a=a(x,\xi)$ so that
\begin{equation}
\label{EqESymb}
  |\pa_x^\alpha\pa_\xi^\beta a(x,\xi)|\leq C_{K\alpha\beta}\la\xi\ra^{s-|\beta|}
\end{equation}
for all $\alpha,\beta\in\N_0^n$, $x\in K$ with $K$ compact, and $\xi\in\R^n$, while $S^s_{1-\delta,\delta}(T^*\R^n)$ is H\"ormander's $(\rho,\delta)$-class with $\rho=1-\delta$ \cite{HormanderFIO1}.

\textbf{Polyhomogeneity.} Let $M$ be a manifold with boundary, and consider a collar neighborhood $[0,1)_\rho\times\pa M$ of $\pa M$. Let $\cE\subset\C\times\N_0$ be an \emph{index set}, i.e.\ $(z,k)\in\cE$ implies $(z+1,k)\in\cE$ and also $(z,k-1)\in\cE$ when $k\geq 1$, and for all $C$ there are only finitely many $(z,k)\in\cE$ with $\Re z<C$. Let $\alpha<\min_{(z,k)\in\cE}\Re z$. Then the space
\[
  \cA_\phg^\cE(M) \subset \cA^\alpha(M)
\]
of \emph{polyhomogeneous conormal distributions} (at $\pa M$) consists of all $u\in\CI(M^\circ)$ so that there exist $u_{(z,k)}\in\CI(\pa M)$, $(z,k)\in\cE$, with the property that for all $C\in\R$, we have
\[
  u(\rho,y) - \sum_{\genfrac{}{}{0pt}{}{(z,k)\in\cE}{\Re z\leq C}} \chi(\rho)\rho^z(\log\rho)^k u_{(z,k)}(y) \in \cA^C(M);
\]
here $\chi\in\CIc([0,1))$ is identically $1$ near $0$. We write $u\sim\sum_{(z,k)\in\cE}\rho^z(\log\rho)^k u_{(z,k)}$.

\textbf{Blow-ups.} Let $M$ be a manifold with corners. We say that $S\subset M$ is a \emph{p-submanifold} if near each point in $S$ there exist local coordinates
\begin{equation}
\label{EqEpCoord}
  x^1,\ldots,x^k\geq 0,\quad
  y^1,\ldots,y^{n-k}\in\R,
\end{equation}
so that $S$ is defined by the vanishing of a subset of these coordinates. We are only concerned with the case of boundary p-submanifolds: these are contained in a boundary hypersurface of $M$, i.e.\ at least one of the coordinates $x^1,\ldots,x^k$ vanishes along it. Following \cite{MelroseDiffOnMwc}, the \emph{(real) blow-up} of $M$ along $S$ is then defined as
\[
  [M;S] := (M\setminus S) \sqcup S\,{}^+NS,
\]
where $S\,{}^+N S=({}^+N S\setminus o)/\R_+$ is the quotient by dilations of the inward pointing normal bundle ${}^+N S={}^+T_S M/T S$, where ${}^+T_p M\subset T_p M$ for $p\in S\subset M$ consists of all (non-strictly) inward pointing tangent vectors; thus $S\,{}^+N S$ is a bundle of intersections of spheres with orthants. The space $[M;S]$ can be given a natural smooth structure, defined by declaring polar coordinates around $S$ to be valid down to the polar coordinate origin; the blow-down map $\upbeta\colon[M;S]\to M$, which is the identity on $M\setminus S$ and the base projection on $S\,{}^+N S$, is then smooth, and a diffeomorphism from the complement of the \emph{front face} $S\,{}^+N S\subset[M;S]$ to $M\setminus S$. For a closed set $T\subset M$, we define its \emph{lift} $\upbeta^*T$ to $[M;S]$ to be
\[
  \upbeta^*T := \begin{cases} \upbeta^{-1}(T), & T\subset S, \\ \ol{\upbeta^{-1}(T\setminus S)}, & T\not\subset S. \end{cases}
\]
The lift of a vector field $V$ on $M^\circ$ to $[M;S]$ is simply $V$ again (note that $S\subset\pa M$ is disjoint from $M^\circ$). The lift of a b-vector field is smooth when $S$ is a boundary face (i.e.\ an intersection of boundary hypersurfaces), but not for general $S$.

\begin{rmk}[Edge-vector fields under blow-ups of fibers]
\label{RmkEBlowup}
  Consider a manifold $M$ with fibered boundary as in~\eqref{EqEFibration}. Consider the blow-up of $M$ along a fiber $\phi^{-1}(y_0)$, $y_0\in Y$; then the lift of $\pa M$ to $[M;\phi^{-1}(y_0)]$ is $[\pa M;\phi^{-1}(y_0)]$, and the map $\phi$ lifts (i.e.\ extends from the complement of the front face) to a fibration $Z-[\pa M;\phi^{-1}(y_0)]\to[Y;\{y_0\}]$. (The other boundary hypersurface of $[M;\phi^{-1}(y_0)]$, namely the front face of the blow-up, is equipped only with the trivial fibration with base $\{y_0\}$.) On the manifold with corners $[M;\phi^{-1}(y_0)]$, we can thus consider the space of edge-b-vector fields, i.e.\ all smooth vector fields which are tangent to the fibers of the boundary hypersurface $[\pa M;\phi^{-1}(y_0)]$. One can then check that this space is equal to the span, over $\CI([M;\phi^{-1}(y_0)])$, of the lift of $\Ve(M)$. We check only part of this in the special case that $\dim Y=1$: in local coordinates $x\geq 0$, $y\in\R$, $z\in\R^{n_Z}$ on $M$, with $y_0=0$, smooth coordinates near the interior of the front face are $x\geq 0$, $\hat y=\frac{y}{x}\in\R$, $z\in\R^{n_Z}$, and the edge vector fields $x\pa_x$, $x\pa_y$, $\pa_z$ become $x\pa_x-\hat y\pa_{\hat y}$, $\pa_{\hat y}$, $\pa_z$, which have the same span as $x\pa_x$, $\pa_{\hat y}$, $\pa_z$, as claimed.
\end{rmk}

\subsection{Edge pseudodifferential operators; variable orders}
\label{SsEPsdo}

We consider a manifold $M$ with fibered boundary $\pa M$ as in~\eqref{EqEFibration}, and focus on operators acting on complex-valued functions; the extension to operators acting between sections of vector bundles requires only notational modifications. Following \cite{MazzeoEdge}, we define the space $\Psie^s(M)$ of edge-ps.d.o.s of order $s\in\R$ to consist of all linear operators on $\CI(M)$ whose Schwartz kernels are properly supported and lift to the \emph{edge double space}
\[
  M^2_\eop := [M^2; \diag_\phi],\qquad \diag_\phi := \{ (p,p') \in (\pa M)^2 \colon \phi(p)=\phi(p') \},
\]
to conormal distributions of class $I^s(M^2_\eop,\diag_\eop;\pi_R^*\Omegae M)$ which vanish to infinite order at the left and right boundaries of $M^2_\eop$ (the lifts of $\pa M\times M$ and $M\times\pa M$). Here $\diag_\eop$ is the lift of the diagonal in $M^2$, and $\pi_R\colon M^2_\eop\to M$ is the right projection composed with the blow-down map; and $\Omegae M$ is the density bundle associated with $\Te M$. (The space $\Diffe^m(M)\subset\Psie^m(M)$ consists of all operators whose Schwartz kernels are Dirac distributions at $\diag_\eop$.) For the purposes of the present paper, we only consider operators whose Schwartz kernels are supported in a small neighborhood of $\diag_\eop$ in $M^2_\eop$. Such operators are sums of local quantizations: given a symbol $a\in S^s(\Te^*M)$ with support in $\Te^*_K M$ where $K$ is a compact subset of a coordinate chart $[0,\infty)_x\times\R^{n_Y}_y\times\R^{n_Z}_z$, we define $\Ope(a)\in\Psie^s(M)$ as follows: let $\chi\in\CIc$ be $1$ near $K$, and let $\phi\in\CIc((-\frac12,\frac12))$ be equal to $1$ near $0$. Then
\begin{equation}
\label{EqEPsdoOp}
\begin{split}
  &(\Ope(a)u)(x,y,z) \\
  &\qquad := (2\pi)^{-n} \int \exp\Bigl[i\Bigl(\frac{x-x'}{x}\xi+\frac{y-y'}{x}\cdot\eta+(z-z')\cdot\zeta\Bigr)\Bigr] \\
  &\qquad \hspace{6em} \times \phi\Bigl(\frac{x-x'}{x}\Bigr)\phi\Bigl(\Bigl|\frac{y-y'}{x}\Bigr|\Bigr)\chi(x',y',z') \\
  &\qquad \hspace{6em} \times a(x,y,z;\xi,\eta,\zeta) u(x',y',z')\,\dd\xi\,\dd\eta\,\dd\zeta\,\frac{\dd x'}{x'}\,\frac{\dd y'}{x'{}^{n_Y}}\,\dd z'.
\end{split}
\end{equation}
The size of the support of $\phi$ controls the localization near $\diag_\eop$. Compositions of two such operators (for possibly different charts) are again (sums of) operators of this form upon enlarging the supports of the cutoffs appropriately. The formula for the full symbol $c$ of $\Ope(a)\circ\Ope(b)\equiv\Ope(c)\bmod\Psie^{-\infty}(M)$ can be obtained by changing coordinates in the formula valid in the interior $x>0$, so
\begin{equation}
\label{EqEPsdoFullSymb}
  c \sim a b + \Bigl((\pa_\xi a) \bigl((x\pa_x + \xi\pa_\xi + \eta\pa_\eta)b\bigr) + \pa_\eta a\cdot x\pa_y b + \pa_\zeta a\cdot\pa_z b\Bigr) + \cdots
\end{equation}

Using a partition of unity, we can define a quantization map $\Ope\colon S^s(\Te^*M)\to\Psie^s(M)$ (which is surjective onto $\Psie^s(M)/\Psie^{-\infty}(M)$). More generally, one can allow $a$ to be a symbol of class $S^{\sfs}(\Te^*M)$ where $\sfs\in\CI(\Se^*M)$ (with $\Se^*M$ being the boundary at fiber infinity of $\ol{\Te^*}M$). We define the elliptic set $\Elle^\sfs(A)$ of $A=\Ope^\sfs(a)$ to be the elliptic set of the symbol $a$, i.e.\ in local coordinates the set of all $(x_0,y_0,z_0;\xi_0,\eta_0,\zeta_0)$ with $(\xi_0,\eta_0,\zeta_0)\neq 0$ so that $|a(x,y,z;\xi,\eta,\zeta)|\gtrsim(1+|\xi|+|\eta|+|\zeta|)^{\sfs(x,y,z;\xi,\eta,\zeta)}$ for all $(x,y,z)$ near $(x_0,y_0,z_0)$ and $(\xi,\eta,\zeta)$ in a conic neighborhood of $(\xi_0,\eta_0,\zeta_0)$. The operator wave front set $\WFe'(A)$ is the essential support of $a$, i.e.\ the set of all $(x_0,y_0,z_0;\xi_0,\eta_0,\zeta_0)$ with $(\xi_0,\eta_0,\zeta_0)\neq 0$ so that $|a(x,y,z;\xi,\eta,\zeta)|\lesssim(1+|\xi|+|\eta|+|\zeta|)^{-N}$ for all $N$ does \emph{not} hold in any conic neighborhood of $(x_0,y_0,z_0;\xi_0,\eta_0,\zeta_0)$. Both $\Elle(A)$ and $\WFe'(A)$ are conic subsets of $\Te^*M\setminus o$, and we may thus equivalently regarding them as subsets of $\ol{\Te^*M}\setminus o$ or $\Se^*M$.

Classes of weighted edge-ps.d.o.s are defined by
\[
  \Psie^{\sfs,\ell}(M) = \rho^{-\ell}\Psie^\sfs(M)
\]
where $\rho\in\CI(M)$ is a boundary defining function. The principal symbol map is
\[
  \sigmae^{\sfs,\ell}\colon\Psie^{\sfs,\ell}(M)\to\left(S^{\sfs,\ell}/\bigcap_{\delta>0}S^{\sfs-1+2\delta,\ell}\right)(\Te^*M),
\]
where $S^{\sfs,\ell}(\Te^*M)=\rho^{-\ell}S^\sfs(\Te^*M)$. (For constant orders $\sfs=s\in\R$, $\sigmae^{s,\ell}$ takes values in $S^{s,\ell}/S^{s-1,\ell}$.) We have the usual formula $\sigmae^{\sfs+\sfs'-1+2\delta,\ell+\ell'}(i[P,P'])=H_{\sigmae^{\sfs,\ell}(P)}\sigmae^{\sfs',\ell'}(P')$ for the principal symbol of the commutator of $P\in\Psie^{\sfs,\ell}(M)$ and $P'\in\Psie^{\sfs',\ell'}(M)$. For a detailed discussion of variable order operators in the scattering algebra, we refer the reader to \cite{VasyMinicourse}; see \cite[Appendix~A]{BaskinVasyWunschRadMink} for the b-case.

An element $A=\Ope(a)$ of $\Psie^\sfs(M)$ is an edge-operator with smooth coefficients on $M$ in that the underlying symbol $a$ is smooth at the boundary $\ol{\Te^*_{\pa M}}M$. One can allow for $a$ to be conormal there, i.e.\ $a\in\cA^{-\sfs,-\ell}_{\Se^*M,\delta;\ol{\Te^*_{\pa M}}M}(\ol{\Te^*}M)$; the weight $\ell$ is required to be constant. The composition formula $\Psie^{\sfs,\ell}(M)\circ\Psie^{\sfs',\ell'}(M)\subset\Psie^{\sfs+\sfs',\ell+\ell'}(M)$ can be deduced also for such conormal symbols from the corresponding result for constant orders (which can be proved using the techniques of \cite{MazzeoEdge}) and an inspection of the full symbol in local coordinates. Alternatively, one can use \cite[\S{1.2.4}, example~(4)]{HintzScaledBddGeo} to define $\Psie$ for such conormal symbols.

For notational simplicity, we now restrict to the case that $M$ is compact. (In this paper, when working on non-compact manifolds $M$ we always work in a fixed compact subsets of $M$.) Fixing a smooth weighted b-density (or equivalently a smooth weighted edge density) on $M$, we can then define variable order Sobolev spaces
\begin{equation}
\label{EqEHe}
  \He^{\sfs,\ell}(M),\qquad \sfs\in\CI(\Se^*M),\ \ell\in\R,
\end{equation}
via testing with elliptic elements of $\Psi^{\sfs,\ell}(M)$: namely, $u\in\He^{\sfs,\ell}(M)$ if and only if $u\in\He^{s_0,\ell}(M)$ where $s_0\leq\inf\sfs$, and $A u\in L^2(M)$ for some $A\in\Psie^{\sfs,\ell}(M)$ with elliptic principal symbol (or equivalently: for all $A\in\Psie^{\sfs,\ell}(M)$).

\textbf{Edge-ps.d.o.s with edge regular symbols.} In this paper, we shall encounter edge operators whose coefficients merely have edge regularity in the base. We define these as uniform pseudodifferential operators on $M^\circ$ upon equipping $M^\circ$ with the structure of a manifold with bounded geometry \cite{ShubinBounded}. (This idea is essentially used already in \cite[(3.23) Corollary]{MazzeoEdge} to deduce the boundedness properties of edge ps.d.o.s on edge Sobolev spaces.) This structure is given by restricting any fixed Riemannian edge metric, i.e.\ a smooth (on $M$) positive definite section of $S^2\,\Te^*M$, to $M^\circ$. Balls of radius $\sim 1$ with respect to this metric can be used to define a cover of $M^\circ$ by preferred coordinate charts into $\R^n$ so that all transition functions are uniformly (i.e.\ independently of the charts) bounded with all derivatives, and so that there is a constant $\mathfrak{C}$ so that every collection of $\mathfrak{C}+1$ different charts have empty intersection. In local coordinates $x\geq 0$, $y\in\R^{n_Y}$, $z\in\R^{n_Z}$ adapted to the boundary fibration of $M$, one may alternatively take as preferred charts the open sets
\[
  U_{j k l}=(2^{-j-1},2^{-j+1})\times(2^{-j}(k+(-2,2)^{n_Y}))\times(l+(-2,2)^{n_Z})
\]
where $j\in\N$, $k\in\Z^{n_Y}$, $l\in\Z^{n_Z}$, and $B^N\subset\R^N$ denotes the open unit ball, together with the maps
\[
  \phi_{j k l} \colon U_{j k l} \to \R^n,\ \ (x,y,z) \mapsto \bigl(2^j x, 2^j(y-2^{-j}k), z-l\bigr) \in (\tfrac12,2)\times (-2,2)^{n_Y}\times (-2,2)^{n_Z} \subset \R^n.
\]
Given a uniform symbol $a\in S^s_{\rm uni}(T^*M^\circ)$, i.e.\ when expressed in preferred charts, $a=a(x,\xi)$ satisfies uniform symbol bounds~\eqref{EqESymb}, we can define its quantization as follows: decompose $a$ using a partition of unity subordinate to the cover by preferred charts, quantize each localized piece using the standard quantization formula $\Op(a)=(2\pi)^{-n}\int e^{i(x-y)\cdot\xi}a(x,\xi)\,\dd\xi$ on $\R^n$ followed by a cutoff to the respective chart, and sum the pullbacks of these quantizations. Note that in a local coordinate chart $x,y,z,\xi,\eta,\zeta$ on $\Te^*M$, the membership $a\in S^s_{\rm uni}(T^*M^\circ)$, for a symbol whose support in $(x,y,z)$ is compactly contained in the chart, is equivalent to
\[
  |(x\pa_x)^j (x\pa_y)^\alpha \pa_z^\beta \pa_\xi^k \pa_\eta^\gamma \pa_\zeta^\kappa a(x,y,z;\xi,\eta,\zeta)| \leq C_{j\alpha\beta k\gamma\kappa}(1+|\xi|+|\eta|+|\zeta|)^{s-k-|\gamma|-|\kappa|}
\]
for all $j,k\in\N_0$, $\alpha,\gamma\in\N_0^{n_Y}$, $\beta,\kappa\in\N_0^{n_Z}$; this thus indeed amounts to edge regularity in the base variables.

The full space $\Psi_{\rm uni}^s(M^\circ)$ of uniform ps.d.o.s on $M^\circ$ is defined as the sum of $\Op_{\eop,\rm uni}(S^s_{\rm uni})$ and the space of residual operators, which have uniformly smooth Schwartz kernels with support in a bounded (with respect to a fixed Riemannian edge metric) neighborhood of the diagonal of $M^\circ$; see \cite[\S{A1.3}]{ShubinBounded} for details. We leave the extension to weighted operators (with weights $\rho^{-l}$ where $\rho$ is a boundary defining function) and operators with variable orders\footnote{More generally, one can allow for $\sfs\in\CI(S^*M^\circ)$ which, in preferred charts, are uniformly bounded with all derivatives.} $\sfs\in\CI(\Se^*M)$ to the reader.

Given $A\in\Psi_{\rm uni}^s(M^\circ)$ relative to the edge structure on $M$, we can not only define the elliptic set and operator wave front set of $A$ in the usual manner over $M^\circ$, but also in $\Se^*_{\pa M}M$: for example, $\varpi\in\Se^*_{\pa M}M$ lies in the elliptic set if there exists an open neighborhood $V\subset\Se^*M$ of $\varpi$ so that the principal symbol of $A$, in all preferred charts $\phi\colon U\subset M^\circ\to\R^n$ and in $S^*U\cap\phi_*(V)$, is uniformly elliptic. A similar definition can be given for operators with weights and variable orders.

For our purposes, it is important to have more refined notions of elliptic set and operator wave front set: in the context of Remark~\ref{RmkEBlowup}, suppose $a\in\cA^{-s}_{\upbeta^*\Se^*M}(\upbeta^*\ol{\Te^*}M)$ where $\upbeta\colon[M;\phi^{-1}(y_0)]\to M$ is the blow-down map; that is, $a$ is smooth in the base variables on $[M;\phi^{-1}(y_0)]$ and a symbol of order $s$ in the fibers of $\Te^*M$. Then $a|_{T^*M^\circ}$ is a uniform symbol (as follows from the fact that edge vector fields on $M$ lift to edge-b-vector fields---in particular, smooth vector fields---on $[M;\phi^{-1}(y_0)]$), so $A=\Op_{\eop,\rm uni}(a)$ is well-defined. But given such symbols $a$, we can define the elliptic set and operator wave front set of such $A=\Op_{\eop,\rm uni}(a)$ as subsets of the bundle $\upbeta^*\Se^*M\to[M;\phi^{-1}(y_0)]$. Such symbols arise in~\S\ref{SsSUPr} as products of smooth symbols of $\Te^*M$ with localizing factors of the form $\chi(\frac{y-y_0}{x})$ where, say, $\chi\in\CIc(\R^{n_Y})$.

\subsection{Invariant edge Sobolev spaces}
\label{SsEInv}

Part of the following discussion is an adaptation of \cite[\S2.5]{HintzVasyScrieb} and \cite[\S2.7.2]{Hintz3b}. We drop vector bundles from the notation. Using a choice of boundary defining function $x\geq 0$ and of local coordinates $y\in\R^{n_Y}$ on the base $Y$, we may use $\dd x$ and $\dd y$ as global coordinates on the space ${}^+N\phi^{-1}(y_0)$ on which the edge normal operator~\eqref{EqENe} is defined. Relabeling $\dd x,\dd y$ as $x,y$, the edge normal operator is thus an element of $\Diff_{\eop,\rm I}(\cN)$, i.e.\ an edge operator on
\[
  \cN := [0,\infty)_x\times\R^{n_Y}_y\times Z
\]
which is invariant under translations in $y$ and dilations in $(x,y)$; the edge structure is defined with respect to the fibration $(y,z)\mapsto y$ of the boundary at $x=0$. Denote by $\pi\colon\cN\to Z$ the projection. The space
\[
  \cV_{\eop,\rm I}(\cN) \subset \Ve(\cN)
\]
of invariant edge vector fields is then spanned over $\pi^*\CI(Z)$ by $x\pa_x$, $x\pa_{y^j}$ ($j=1,\ldots,n_Y$), and $\pi^*\cV(Z)$; we write $\Diff_{\eop,\rm I}^s(\cN)$ for the corresponding space of invariant edge differential operators.

We only consider the case that the base $Z$ is compact (without boundary). Fix an invariant edge density $\mu_0=|\frac{\dd x}{x}\frac{\dd y}{x^{n_Y}}\,\dd z|$ where we write $|\dd z|$ for any fixed positive smooth density on $Z$; for a weight $w\in\R$, we then define $L^2(\cN)$ using the density $x^w\mu_0$. For $s\in\N_0$, $\alpha\in\R$ we further define the invariant weighted edge Sobolev space
\[
  H_{\eop,\rm I}^{s,\alpha}(\cN) = x^\alpha H_{\eop,\rm I}^s(\cN)
\]
to consist of all $u\in x^\alpha L^2(\cN)$ so that $P u\in x^\alpha L^2(\cN)$ for all $P\in\Diff_{\bop,\rm I}^s(\cN)$. For general $s\in\R$, the space $H_{\eop,\rm I}^{s,\alpha}(\cN)$ is defined using interpolation and duality. Spaces of sections of bundles $\pi^*\cE\to\cN$, where $\cE\to Z$ is a smooth vector bundle are defined analogously; such bundles typically arise as restrictions of bundles on $M$ to fibers of $\pa M$.

We will encounter such spaces with variable differential orders $\sfs$. Such orders are required to be elements $\sfs\in\CI(\Se^*\cN)$ which are invariant; such functions are uniquely determined by their restriction to $\Se^*_Z\cN$ where $Z\cong\{0\}\times\{0\}\times Z\subset\cN$, and we shall thus identify them with their restrictions, i.e.\ $\sfs\in\CI(\Se^*_Z\cN)$. We then need to consider invariant edge-ps.d.o.s; these arise as quantizations of invariant symbols $a\in S^\sfs_I(\Te^*\cN):=S^\sfs(\Te^*_Z\cN)$. In the coordinates defined by writing edge covectors as $\xi\frac{\dd x}{x}+\eta\cdot\frac{\dd y}{x}+\zeta\cdot\dd z$ (using local coordinates on $Z$), such symbols are of the form $a=a(z;\xi,\eta,\zeta)$ and satisfy the bound $|\pa_z^\beta\pa_\xi^k\pa_\eta^\gamma\pa_\zeta^\kappa a(z;\xi,\eta,\zeta)|\leq C_{\delta,\beta k\gamma\kappa}(1+|\xi|+|\eta|+|\zeta|)^{\sfs(z,\xi,\eta,\zeta)-(1-\delta)(k+|\gamma|+|\kappa|)}$ for all $\delta>0$ and $k\in\N_0$, $\beta,\kappa\in\N_0^{n_Z}$, $\gamma\in\N_0^{n_Y}$. If $a$ has $z$-support contained in the set where a function $\chi\in\CIc(\R^{n_z})$ equals $1$, we define the invariant quantization of $a$ to have Schwartz kernel
\begin{align*}
  &\Op_{\eop,\rm I}(a)(x,y,z,x',y',z') \\
  &\qquad := (2\pi)^{-n} \int \exp\Bigl[i\Bigl(\log\Bigl(\frac{x}{x'}\Bigr)\xi+\frac{y-y'}{x'}\cdot\eta+(z-z')\cdot\zeta\Bigr)\Bigr] \\
  &\qquad \hspace{6em} \times \phi\Bigl(\log\Bigl(\frac{x}{x'}\Bigr)\Bigr)\phi\Bigl(\Bigl|\frac{y-y'}{x'}\Bigr|\Bigr)\chi(z')a(z;\xi,\eta,\zeta)\,\dd\xi\,\dd\eta\,\dd\zeta\,\Bigl|\frac{\dd x'}{x'}\,\frac{\dd y'}{x'{}^{n_Y}}\,\dd z\Bigr|
\end{align*}
where $\phi\in\CIc((-\frac12,\frac12))$ equals $1$ near $0$; the present choice of phase function is more convenient for present purposes than the one made in~\eqref{EqEPsdoOp}.\footnote{One can define $\Psi_{\eop,\rm I}^\sfs(\cN)$ as the space of sums of such quantizations and residual operators, i.e.\ invariant operators whose Schwartz kernels are smooth right edge densities vanishing to infinite order at the left and right boundary hypersurfaces of $\cN_\eop^2$; in this paper however, it suffices to work only with quantizations.} (One can drop all three cutoff functions, via convolution of $a$ in $\xi,\eta,\zeta$ with their inverse Fourier transforms, upon modifying $a$ by an invariant symbol of order $-\infty$.) The space
\[
  H_{\eop,I}^{\sfs,\alpha}(\cN) \subset H_{\eop,I}^{\inf\sfs,\alpha}(\cN)
\]
can now be defined in the usual manner via testing with an invariant edge-ps.d.o.\ of order $\sfs$ which is elliptic. The goal of the remainder of this section is to describe the behavior of these spaces under the Fourier transform in $y$.

First consider the phase space relationship: $x D_x$, $x D_y$, $D_z$ formally transform to $x D_x$, $x\eta$ (with $\eta\in\R^{n_Y}$ the dual variable to $y$), $D_z$, and upon recalling the map $\hat M_\lambda(\hat x,z)=(\hat x/\lambda,z)$ from~\eqref{EqEScale1}, their pullbacks under $\hat M_{|\eta|}$ are $\hat x D_{\hat x}=(1+\hat x)\cdot\frac{\hat x}{1+\hat x}D_{\hat x}$, $\hat x\hat\eta$ (where $\hat\eta=\frac{\eta}{|\eta|}$), $D_z=(1+\hat x)\cdot\frac{1}{1+\hat x}D_z$. On
\[
  \hat X:=[0,\infty]_{\hat x}\times Z,
\]
these are weighted b-scattering operators, i.e.\ they lie in $(1+\hat x)\Diff_{\bop,\scop}^1(\hat X)$, cf.\ \eqref{EqEDiffbsc}. Writing b-scattering covectors on $\hat X$ as
\[
  \xi_{\bop,\scop} (1+\hat x)\frac{\dd\hat x}{\hat x} + (1+\hat x)\zeta_{\bop,\scop},\qquad \zeta_{\bop,\scop}\in T^*Z,
\]
we are thus led to define the map
\begin{equation}
\label{EqEInvPhaseSpace}
  f_{\hat\eta} \colon \ol{{}^{\bop,\scop}T^*}\hat X \to \ol{\Te^*_Z}\cN,\qquad
  f_{\hat\eta}(\hat x,z;\xi_{\bop,\scop},\zeta_{\bop,\scop}) = \bigl(z; (1+\hat x)\xi_{\bop,\scop}, \hat x\hat\eta, (1+\hat x)\zeta_{\bop,\scop}\bigr).
\end{equation}
Our computations imply that for $A\in\Diff_{\eop,\rm I}^s(\cN)$ we have the formula
\begin{equation}
\label{EqEInvRedNeSymb}
  {}^{\bop,\scop}\upsigma^s(\hat N_\eop(A,\hat\eta)) = f_{\hat\eta}^*\bigl(\sigmae^s(A)\bigr)
\end{equation}
for the b-scattering principal symbol of the reduced normal operator (which is $y$-in\-de\-pen\-dent, and thus there is no need to specify $y$ here).

In order to extend this to the pseudodifferential case, we first define b-scattering ps.d.o.s on $\hat X$ whose Schwartz kernels are supported near the diagonal. We immediately consider variable orders $\sfs\in\CI({}^{\bop,\scop}S^*\hat X)$, $\sfr\in\CI(\ol{{}^{\bop,\scop}T^*_{\pa_\infty\hat X}}\hat X)$ (where $\pa_\infty\hat X:=\{\infty\}\times Z$) and symbols $a\in S^{\sfs,\sfr}(\ol{{}^{\bop,\scop}T^*}\hat X):=\cA_\cH^{-\sfs,-\sfr}(\ol{{}^{\bop,\scop}T^*}\hat X)$ where $\cH=\{{}^{\bop,\scop}S^*\hat X,\ol{{}^{\bop,\scop}T^*_{\pa_\infty\hat X}}\hat X)$. If $a$ has support contained in $\hat x<\infty$ then in a coordinate chart on $Z$, we quantize $a$ as
\begin{align*}
  \Op_\bop(a) &= (2\pi)^{-n}\iint \exp\Bigl[i\Bigl(\log\Bigl(\frac{\hat x}{\hat x'}\Bigr)\xi+(z-z')\cdot\zeta\Bigr)\Bigr] \\
    &\quad\hspace{6em} \times \phi\Bigl(\log\Bigl(\frac{\hat x}{\hat x'}\Bigr)\Bigr)\chi(z')a(\hat x,z;\xi,\zeta)\,\dd\xi\,\dd\zeta\,\Bigl|\frac{\dd\hat x'}{\hat x'}\,\dd z'\Bigr|
\end{align*}
where we write b-covectors as $\xi\frac{\dd\hat x}{\hat x}+\zeta\cdot\dd z$, whereas for $a$ with support contained in $\hat x^{-1}<\infty$, we write scattering covectors as $\xi_\scop\,\dd\hat x+\hat x^{-1}\zeta_\scop\cdot\dd z$ and use
\begin{equation}
\label{EqEInvOpsc}
\begin{split}
  \Op_\scop(a) &= (2\pi)^{-n}\iint \exp\Bigl[i\Bigl(\log\Bigl(\frac{\hat x}{\hat x'}\Bigr)\hat x'\xi_\scop+(z-z')\cdot\hat x'\zeta_\scop\Bigr)\Bigr] \\
    &\quad\hspace{6em} \times \phi\Bigl(\log\Bigl(\frac{\hat x}{\hat x'}\Bigr)\Bigr)\chi(z')a(\hat x,z;\xi_\scop,\zeta_\scop)\,\dd\xi_\scop\,\dd\zeta_\scop\cdot\hat x'{}^{n_Z}|\dd\hat x'\,\dd z'|.
\end{split}
\end{equation}
The principal symbol of such quantizations is well-defined in $S^{\sfs,\sfr}/\bigcap_{\delta>0}S^{\sfs-1+2\delta,\sfr-1+2\delta}$. Upon fixing a positive (weighted) b-density, we can define the associated function spaces $H_{\bop,\scop}^{\sfs,0,\sfr}(\hat X)$, and also more general weighted spaces
\[
  H_{\bop,\scop}^{\sfs,\alpha,\sfr}(\hat X) := \Bigl(\frac{\hat x}{\hat x+1}\Bigr)^\alpha H_{\bop,\scop}^{\sfs,0,\sfr}(\hat X).
\]

Consider now an invariant edge ps.d.o.\ $A=\Op_{\eop,\rm I}(a)$ with symbol $a\in S^\sfs_{\rm I}(\Te^*\cN)$ and Schwartz kernel
\[
  (2\pi)^{-n}\iiint \exp\Bigl[i\Bigl(\log\Bigl(\frac{x}{x'}\Bigr)\xi+\frac{y-y'}{x'}\cdot\eta+(z-z')\cdot\zeta\Bigr)\Bigr] a(z;\xi,\eta,\zeta)\,\dd\xi\,\dd\eta\,\dd\zeta\,\Bigl|\frac{\dd x'}{x'}\,\frac{\dd y'}{x'{}^{n_Y}}\,\dd z\Bigr|;
\]
we consider here the case that this Schwartz kernel is supported in a region of bounded $|\log\frac{x}{x'}|$, $|\frac{y-y'}{x'}|$, $|z-z'|$. In view of the translation invariance in $y$, $\cF(A u)(\eta;x,z)$ (where $\cF$ is the Fourier transform in $y$) is given by $(N_\eop(A,\eta)(\cF u)(\eta))(x,z)$ where $N_\eop(A,\eta)$ has Schwartz kernel
\[
  (x,z;x',z') \mapsto (2\pi)^{-(n-1)}\iint \exp\Bigl[i\Bigl(\log\Bigl(\frac{x}{x'}\Bigr)\xi + (z-z')\cdot\zeta\Bigr)\Bigr] a(z;\xi,x'\eta,\zeta)\,\dd\xi\,\dd\zeta\,\Bigl|\frac{\dd x'}{x'}\,\dd z\Bigr|.
\]
Passage to the reduced normal operator is effected via pullback along $\hat M_{|\eta|}$, i.e.\ formally replacing $x,x'$ by $\hat x/|\eta|,\hat x/|\eta|$; this gives
\begin{equation}
\label{EqEInvRedNeSK}
\begin{split}
  &\hat N_\eop(A,\hat\eta)(\hat x,z;\hat x',z') \\
  &\qquad = (2\pi)^{-(n-1)}\iint \exp\Bigl[i\Bigl(\log\Bigl(\frac{\hat x}{\hat x'}\Bigr)\xi + (z-z')\cdot\zeta\Bigr)\Bigr] a(z;\xi,\hat x'\hat\eta,\zeta)\,\dd\xi\,\dd\zeta\,\Bigl|\frac{\dd\hat x'}{\hat x'}\,\dd z\Bigr|.
\end{split}
\end{equation}

\begin{lemma}[Reduced normal operator]
\label{LemmaEInvRedNe}
  Let $\sfs\in\CI(\Se^*_Z\cN)$ and $A=\Op_{\eop,\rm I}(a)$ where $a\in S^\sfs_{\rm I}(\Te^*M)$. Then
  \[
    \hat N_\eop(A,\hat\eta) \in \Psi_{\bop,\scop}^{f_{\hat\eta}^*\sfs,f_{\hat\eta}^*\sfs}(\hat X)
  \]
  with smooth dependence on $\hat\eta\in\Sph^{n_Y-1}$ in $\Psi_{\bop,\scop}^{s_0,s_0}(\hat X)$ for any $s_0\geq\sup\sfs$, and its principal symbol is given by~\eqref{EqEInvRedNeSymb}.
\end{lemma}
\begin{proof}
  Note that the map $\Tb^*([0,\infty)\times Z)\ni(\hat x,z;\xi,\zeta)\mapsto(z;\xi,\hat x\hat\eta,\zeta)\in\Te^*_Z\cN$ is the same as the map~\eqref{EqEInvPhaseSpace}. The variable order and principal symbols statements follow in $\hat x<\infty$ then directly from the formula~\eqref{EqEInvRedNeSK}, as does the membership in $\Psi_\bop^{f_{\hat\eta}^*\sfs}([0,\infty)\times Z)$. For $\hat x,\hat x'\gtrsim 1$ on the other hand, we change variables via $\xi=\hat x'\xi_\scop$, $\zeta=\hat x'\zeta_\scop$ and obtain a quantization of the form~\eqref{EqEInvOpsc} with symbol $a_\scop(z;\xi_\scop,\zeta_\scop)\mapsto a(z;\hat x'\xi_\scop,\hat x'\hat\eta,\hat x'\zeta_\scop)$. This is bounded by $(1+|\hat x'|)^{\sfs(z;\xi_\scop,\hat\eta,\zeta_\scop)}(1+|\xi_\scop|+|\zeta_\scop|)^{\sfs(z;\xi_\scop,\hat\eta,\zeta_\scop)}$, and derivatives are estimated in a similar manner; so $a_\scop\in S^{f_{\hat\eta}^*\sfs,f_{\hat\eta}^*\sfs}$ indeed.
\end{proof}

\begin{lemma}[Fourier transform of invariant edge Sobolev spaces]
\label{LemmaEInvFT}
  Fix a positive smooth density $\mu_Z$ on $Z$ and the weighted b-densities $\mu_\cN=x^w|\frac{\dd x}{x}\,\dd y\,\mu_Z|$ and $\mu_{\hat X}=\hat x^w|\frac{\dd\hat x}{\hat x}\mu_Z|$ on $\cN$ and $\hat X$, respectively. Let $\sfs\in\CI(\Se^*_Z\cN)$, $\alpha\in\R$. Denoting by $(\cF u)(\eta;x,z):=\int e^{i y\cdot\eta}u(x,y,z)\,\dd y$ the Fourier transform in $y$ and recalling the notation $M_\lambda(\hat x,z)=(\hat x/\lambda,z)$, we then have an equivalence of squared norms
  \[
    \| u \|_{H_{\eop,\rm I}^{\sfs,\alpha}(\cN;\mu_\cN)}^2 \sim \int_{\Sph^{n_Y-1}}\int_0^\infty \|\hat M_{|\eta|}^*(\cF u(|\eta|\hat\eta;\cdot))\|_{H_{\bop,\scop}^{f_{\hat\eta}^*\sfs,\alpha,f_{\hat\eta}^*\sfs-\alpha}(\hat X;\mu_{\hat X})}^2\, |\eta|^{n_Y-1+2\alpha-w}\,\dd|\eta|\,\dd\hat\eta,
  \]
  i.e.\ the left hand side is bounded by a constant times the right hand side, and vice versa.
\end{lemma}
\begin{proof}
  Since $\hat M_{|\eta|}^*x^{-\alpha}=|\eta|^\alpha\hat x^{-\alpha}$, it suffices to consider the case $\alpha=0$. Similarly, in view of $\hat M_{|\eta|}^*x^w=\hat x^w|\eta|^{-w}$ we may reduce to $w=0$; thus, we work with the densities $|\frac{\dd x}{x}\dd y\,\mu_Z|$ and $|\frac{\dd\hat x}{\hat x}\mu_Z|$. The case $\sfs=0$ then follows from Plancherel's theorem and the fact that $\hat M_{|\eta|}^*|\frac{\dd x}{x}\mu_Z|=|\frac{\dd\hat x}{\hat x}\mu_Z|$; to wit,
  \begin{align*}
    \|u\|_{L^2(\cN;\mu_\cN)}^2 &\sim \int_{\R^{n_Y}} \|\cF u(\eta;\cdot)\|_{L^2(\frac{\dd x}{x}\mu_Z)}^2\,\dd\eta \\
      &= \int_{\Sph^{n_Y-1}}\int_0^\infty \| (\hat M_{|\eta|}^*\cF u)(|\eta|\hat\eta;\cdot)\|_{L^2(\frac{\dd\hat x}{\hat x}\mu_Z)}^2\,|\eta|^{n_Y-1}\,\dd|\eta|\,\dd\hat\eta.
  \end{align*}

  For $\sfs\geq 0$, fix $A=\Op_{\eop,\rm I}(a)$ where $a\in S^\sfs_{\rm I}(\Te^*\cN)$ is elliptic; then $\|u\|_{H_{\eop,\rm I}^\sfs}^2\sim\|u\|_{L^2}^2+\|A u\|_{L^2}^2$ is equivalent to the integral over $\Sph^{n_Y-1}\times[0,\infty)$ (with measure $|\eta|^{n_Y-1}\,\dd|\eta|\,\dd\hat\eta$) of
  \begin{align*}
    & \| \hat M_{|\eta|}^*(\cF u(|\eta|\hat\eta;\cdot))\|_{L^2}^2 + \|\hat M_{|\eta|}^*(\cF(A u)(|\eta|\hat\eta;\cdot))\|_{L^2}^2 \\
    &\qquad = \| \hat M_{|\eta|}^*(\cF u(|\eta|\hat\eta;\cdot))\|_{L^2}^2 + \|\hat N_\eop(A,\hat\eta)\hat M_{|\eta|}^*(\cF u(|\eta|\hat\eta;\cdot))\|_{L^2}^2 \\
    &\qquad \sim \| \hat M_{|\eta|}^*(\cF u(|\eta|\hat\eta;\cdot)) \|_{H_{\bop,\scop}^{f_{\hat\eta}^*\sfs,0,f_{\hat\eta}^*\sfs}(\hat X)}^2,
  \end{align*}
  where we used Lemma~\ref{LemmaEInvRedNe} in the last step to deduce that $\hat N_\eop(A,\hat\eta)\in\Psi_{\bop,\scop}^{f_{\hat\eta}^*\sfs,f_{\hat\eta}^*\sfs}(\hat X)$ is elliptic.

  For constant orders $s<0$, the claim follows via a duality argument: the boundedness of $\cF\colon H_{\eop,\rm I}^{-s}(\cN)\to L^2(\Sph^{n_Y-1}\times[0,\infty);H_{\bop,\scop}^{-s,-s}(\hat X);|\eta|^{n_Y-1}\dd|\eta|\,\dd\hat\eta)$ implies that of the adjoint
  \[
    \cF^*=2\pi\cF^{-1}\colon L^2(\Sph^{n_Y-1}\times[0,\infty);H_{\bop,\scop}^{s,s}(\hat X))\to H_{\eop,\rm I}^s(\cN);
  \]
  similarly for $\cF^{-1}$ and $(\cF^{-1})^*=(2\pi)^{-1}\cF$. For general variable orders $\sfs$ finally, one writes $\|u\|_{H_{\eop,\rm I}^\sfs}\sim\|u\|_{H_{\eop,\rm I}^{s_0}}+\|A u\|_{L^2}$ where $s_0\leq\inf\sfs$ and $A=\Op_{\eop,\rm I}(a)$ with $a\in S_{\rm I}^\sfs(\Te^*\cN)$ elliptic, thereby reducing the claim to already settled cases.
\end{proof}

\bibliographystyle{alphaurl}


\end{document}